\pdfoutput=1
\documentclass[11pt]{article}
\PassOptionsToPackage{table}{xcolor}
\usepackage[margin=1in]{geometry}
\usepackage{amsmath,amssymb,amsthm}
\usepackage{mathtools}
\usepackage{newtxtext}
\usepackage{newtxmath}
\usepackage{bm}
\usepackage{microtype}
\usepackage{natbib}
\bibpunct[, ]{(}{)}{,}{a}{}{,}%
\usepackage[flushleft]{threeparttable}
\usepackage{graphicx}
\usepackage{float}
\usepackage{booktabs}
\usepackage{longtable}
\usepackage{multirow}
\usepackage{tabularx}
\usepackage{enumitem}
\usepackage{comment}
\usepackage{algorithm}
\usepackage[noend]{algpseudocode}
\usepackage{tikz}
\usetikzlibrary{arrows.meta,positioning,shapes.geometric,patterns,decorations.pathreplacing}
\usepackage{needspace}
\usepackage[normalem]{ulem}
\usepackage{hyperref}
\hypersetup{colorlinks=true,linkcolor=black,citecolor=black,urlcolor=black,
  pdftitle={A Robust Chance Constrained Approach to Surgery Scheduling},
  pdfauthor={Marco Caserta, Antonio Garcia Romero, Miguel Vaquero}}
\usepackage{caption}
\usepackage{subcaption}

\theoremstyle{plain}

\newtheorem{proposition}{Proposition}
\newtheorem{lemma}{Lemma}
\newtheorem{corollary}{Corollary}
\theoremstyle{definition}

\newtheorem{remark}{Remark}

\makeatletter
\AtBeginDocument{\@ifundefined{@maketablecaption}{\global\let\@maketablecaption\@makecaption}{}}
\makeatother

\let\proof\relax \let\endproof\relax
\long\def\proof#1{\par\medskip\noindent{\itshape #1}\enspace\ignorespaces}
\def\endproof{\par\medskip}
\providecommand{\Halmos}{\ensuremath{\square}}

\definecolor{hdrgray}{gray}{0.93} 
\definecolor{ieblue}{RGB}{0,51,102}
\definecolor{ieorange}{RGB}{255,102,0}
\definecolor{fluidblue}{RGB}{100,149,237}
\definecolor{fluidbluedark}{RGB}{65,105,225}
\definecolor{fluidbluelight}{RGB}{140,180,210}
\definecolor{pressurecolor}{RGB}{220,20,60}
\definecolor{pistoncolor}{RGB}{169,169,169}
\definecolor{pistondark}{RGB}{128,128,128}
\definecolor{successgreen}{RGB}{34,139,34}
\definecolor{reservoircolor}{RGB}{100,70,140}
\definecolor{reservoirlight}{RGB}{180,160,210}
\definecolor{cBlueF}{RGB}{230,241,251}   
\definecolor{cBlueD}{RGB}{24, 95,165}    
\definecolor{cGrayF}{RGB}{241,239,232}   
\definecolor{cGrayD}{RGB}{95, 94, 90}    
\definecolor{cAmbrF}{RGB}{250,238,218}   
\definecolor{cAmbrD}{RGB}{186,117, 23}   
\definecolor{cPurpF}{RGB}{238,237,254}   
\definecolor{cPurpD}{RGB}{ 83, 74,183}   
\definecolor{cTealF}{RGB}{225,245,238}   
\definecolor{cTealD}{RGB}{ 15,110, 86}   
\definecolor{cConn} {RGB}{120,118,112}   
\definecolor{ieblue}{RGB}{0,84,159}
\definecolor{ieorange}{RGB}{210,120,0}
\definecolor{successgreen}{RGB}{15,110,86}
\definecolor{cMenu}{RGB}{227,238,250}   
\definecolor{cMenuB}{RGB}{74,126,187}
\definecolor{cEng}{RGB}{224,242,235}    
\definecolor{cEngB}{RGB}{56,142,114}
\definecolor{cBuf}{RGB}{240,240,240}    
\definecolor{cBufB}{RGB}{120,120,120}
\definecolor{cMilp}{RGB}{253,239,219}   
\definecolor{cMilpB}{RGB}{214,143,47}
\definecolor{cOut}{RGB}{232,240,223}    
\definecolor{cOutB}{RGB}{106,142,68}
\definecolor{cMode}{RGB}{238,233,246}   
\definecolor{cModeB}{RGB}{126,103,172}

\providecommand{\Acal}{\mathcal A}
\newcommand{\ECappref}[1]{\hyperref[#1]{\ref*{#1}}}
\newcommand{\ECfigref}[1]{\hyperref[#1]{Figure~\ref*{#1}}}

\newcommand{\MV}[1]{#1}

\newcommand{\GRN}[1]{#1}

\newcommand{\CHGon}{}

\newcommand{\marginMV}[1]{}
\newcommand{\marginMC}[1]{}
\newcommand{\marginAG}[1]{}

\newcommand{\ECSwitch}{%
  \clearpage
  \setcounter{section}{0}\renewcommand{\thesection}{EC.\arabic{section}}%
  \def\theHsection{EC.\arabic{section}}%
  \setcounter{equation}{0}\renewcommand{\theequation}{EC.\arabic{equation}}%
  \def\theHequation{EC.\arabic{equation}}%
  \setcounter{figure}{0}\renewcommand{\thefigure}{EC.\arabic{figure}}%
  \def\theHfigure{EC.\arabic{figure}}%
  \setcounter{table}{0}\renewcommand{\thetable}{EC.\arabic{table}}%
  \def\theHtable{EC.\arabic{table}}%
  \setcounter{theorem}{0}\renewcommand{\thetheorem}{EC.\arabic{theorem}}%
  \def\theHtheorem{EC.\arabic{theorem}}%
  \setcounter{proposition}{0}\renewcommand{\theproposition}{EC.\arabic{proposition}}%
  \def\theHproposition{EC.\arabic{proposition}}%
  \setcounter{lemma}{0}\renewcommand{\thelemma}{EC.\arabic{lemma}}%
  \def\theHlemma{EC.\arabic{lemma}}%
  \setcounter{corollary}{0}\renewcommand{\thecorollary}{EC.\arabic{corollary}}%
  \def\theHcorollary{EC.\arabic{corollary}}%
  \setcounter{remark}{0}\renewcommand{\theremark}{EC.\arabic{remark}}%
  \def\theHremark{EC.\arabic{remark}}%
  \setcounter{definition}{0}\renewcommand{\thedefinition}{EC.\arabic{definition}}%
  \def\theHdefinition{EC.\arabic{definition}}%
  \setcounter{algorithm}{0}\renewcommand{\thealgorithm}{EC.\arabic{algorithm}}%
  \def\theHalgorithm{EC.\arabic{algorithm}}%
}
\newcommand{\ECHead}[1]{%
  \begin{center}
    {\Large\bfseries Electronic Companion}\\[6pt]
    {\large\bfseries #1}
  \end{center}
  \bigskip}

\title{A Robust Chance Constrained Approach to Surgery
Scheduling}

\author{Marco Caserta\\[2pt]
  \small IE University, Madrid, Spain\\
  \small\texttt{marco.caserta@ie.edu}
  \and
  Antonio Garc\'ia Romero\\[2pt]
  \small IE University, Madrid, Spain\\
  \small\texttt{antonio.garcia@ie.edu}
  \and
  Miguel Vaquero\\[2pt]
  \small IE University, Madrid, Spain\\
  \small\texttt{miguel.vaquero@ie.edu}}
\date{}

\begin{document}

\maketitle

\begin{abstract}
  \textbf{Problem definition:} We study elective surgery scheduling
  under uncertain procedure durations. Schedules based on mean
  durations or fixed buffering rules can appear efficient ex ante but
  become fragile in execution, because early overruns propagate
  through the day and expose later surgeries to accumulated delay.
  The goal is to design schedules that explicitly control delay risk
  while remaining tractable for practical hospital planning.
  \textbf{Methodology/results:} We propose a robust
  chance-constrained framework that separates uncertainty
  quantification from schedule optimization. A buffer engine converts
  distributional information into reliability-dependent buffered
  durations, and the scheduling model then chooses surgery
  assignments, sequencing decisions, start times, and surgery-level
  reliability levels from a discrete menu. This makes reliability an
  endogenous scheduling decision rather than a fixed service-level
  parameter. The framework accommodates average-reliability,
  worst-day, and hard-target risk postures. Computational comparisons
  with common-reliability and uniform proportional-buffer benchmarks
  show that the menu’s value comes from exploiting surgery-level
  heterogeneity: protection is allocated where it has the greatest
  operational value, reducing avoidable conservatism while preserving
  explicit reliability control. In a rolling-origin case study at HLA
  Moncloa Hospital in Madrid, Spain, across 10 instances with 45 to
  227 surgeries, the approach reduces delays exceeding 90 minutes by
  97\%, lowers the 95th-percentile delay by approximately 620
  minutes, and cuts total overtime by 42\% relative to a
  deterministic mean-based baseline. An analysis of realized
  schedules further shows that the position-dependent cascade of
  delays observed under deterministic scheduling is suppressed under
  the robust schedules. \textbf{Managerial implications:} By placing
  buffers according to uncertainty and operational exposure, the
  framework prevents early overruns from compounding through the
  schedule. Hospitals can use the approach to translate heterogeneous
  duration data and risk preferences into schedules that are both
  more reliable in execution and less conservatively buffered than
  one-size-fits-all rules.
\end{abstract}



\vspace{1em}

\section{Introduction}

Operating room scheduling is one of the highest-leverage operational
decisions in hospitals. Each week, planners assign dozens to hundreds
of cases to operating rooms, days, and surgeons under tight capacity
and staffing constraints. When the plan executes as expected,
operating rooms run at high utilization. When it does not, the
operational costs are immediate: staff overtime, surgeon and patient
waiting, downstream disruptions (e.g., recovery beds and ancillary
services), and, in severe cases, cancellations or rescheduling.

A fundamental challenge is that surgery durations are uncertain and
often systematically miscalibrated. Even when planners use historical
averages and clinical judgment, realized durations exhibit
substantial variability and are frequently right-skewed. Two
operational failure modes follow. First, when cumulative overruns
push the day beyond the planned horizon, the hospital incurs overtime
and end-of-day congestion. Second, and often more disruptive, small
early overruns propagate through the room sequence, generating
cascading delays that disproportionately affect late-day cases and,
eventually, provoke case cancellations --- a concern that is
increasingly important as more hospitals are subject to value-based
payment schemes. Importantly, these disruptions are not well captured
by average performance metrics: two schedules with similar mean delay
can differ substantially in the frequency and severity of extreme delays.
The scheduling problem is therefore not only how much slack to add,
but how to allocate protection across the day so that early uncertainty
is absorbed before it cascades through the schedule.
This paper studies how hospitals should allocate slack in elective
surgery schedules to control execution risk while maintaining
operational efficiency. We ask: {\it how can a planner size and place
  buffers across surgeries so that the resulting schedule is reliable
  in execution---especially in the right tail of delays and end-of-day
  completion---without relying on ad hoc rules or a fully specified
duration distribution that is difficult to validate in practice?}

\subsection{A modular reliability menu for scheduling under uncertainty}

Our approach starts from a simple design principle: decouple the
estimation of reliability from the optimization of the schedule.
Specifically, we represent uncertainty via buffered durations: each
surgery is allocated its nominal duration plus a buffer calibrated to
ensure that the probability of exceeding the allocated time is at
most a chosen tail-risk level $\alpha$, as in
Figure~\ref{fig:buffer-concept}. Rather than imposing a single
reliability target for all cases, we introduce a discrete reliability
menu (the $\alpha$-menu). The scheduler chooses, for each surgery,
one reliability level from this menu and trades off (i) the
operational cost of allocating slack (idle time, overtime, and
sequencing effects) against (ii) the execution risk implied by smaller
buffers. The menu is discrete because it mirrors the small number of
service levels a planner can audit and because it keeps the resulting
formulation finite and transparent.

\begin{figure}[ht]
  \centering
  \resizebox{0.55\textwidth}{!}{%
    \begin{tikzpicture}[scale=0.9]
      \fill[ieblue!70] (0,0) rectangle (5,1.5);
      \node[white,font=\bfseries] at (2.5,0.75) {Mean duration};

      \fill[green!40] (5,0) rectangle (7,1.5);
      \node[font=\small] at (6,0.75) {Buffer};

      \fill[red!25] (7,0) rectangle (9,1.5);
      \node[font=\small,red!70!black] at (8,0.75) {Overrun};

      \begin{scope}[shift={(0,1.8)}]
        \draw[thick,ieblue,fill=ieblue!20] plot[smooth,tension=0.7]
        coordinates {
          (0.5,0) (1.2,0.15) (1.8,0.9) (2.2,1.6) (2.5,1.2) (3,0.7)
          (3.4,1.1) (3.8,2.3) (4.2,2.5) (4.5,1.8) (4.8,1.3)
          (5.2,1.5) (5.5,1.7) (5.8,1.4) (6.2,0.8) (6.5,0.45)
        (7,0.25) (7.4,0.18) (7.8,0.1) (8.3,0.04) (9,0.01) (9.5,0)};

        \node[font=\scriptsize,ieblue!80!black,align=center] at
        (4.75,0.75) {Distribution of\\ realized duration $X_j$};

        \draw[dashed,thick] (5,0) -- (5,2.6);
        \node[above,font=\small] at (5,2.6) {$\mu_j$};

        \draw[green!60!black,very thick] (7,0) -- (7,2.2);
        \node[above,font=\small,green!60!black] at (7,2.2) {$p_j$};

        \fill[red!40,opacity=0.6] plot[smooth,tension=0.7] coordinates {
        (7,0.25) (7.4,0.18) (7.8,0.1) (8.3,0.04) (9,0.01) (9.5,0)}
        -- (9.5,0) -- (7,0) -- cycle;

        \draw[<->,red,thick] (7,0.3) -- (8.5,0.3);
        \node[red,font=\small,above] at (7.75,0.3) {$\alpha$-tail};
      \end{scope}

      \draw[->,thick] (-0.5,0) -- (10,0) node[right] {Time};

      \draw[<->,thick] (0,-0.5) -- (5,-0.5);
      \node[below,font=\small] at (2.5,-0.5) {Mean: $\mu_j$};

      \draw[<->,thick,green!60!black] (5,-0.5) -- (7,-0.5);
      \node[below,font=\small,green!60!black] at (6,-0.5)
      {$b_j(\alpha)=p_j-\mu_j$};

      \node[right,font=\small] at (9,2.5) {$\Pr(X_j>p_j)\le \alpha$};

    \end{tikzpicture}%
  }

  \caption{The buffered-duration concept. The planned duration is
    $p_j=\mu_j+b_j(\alpha)$, where the buffer $b_j(\alpha)$ is sized so
  that $\Pr(X_j>p_j)\le \alpha$.}
  \label{fig:buffer-concept}
\end{figure}
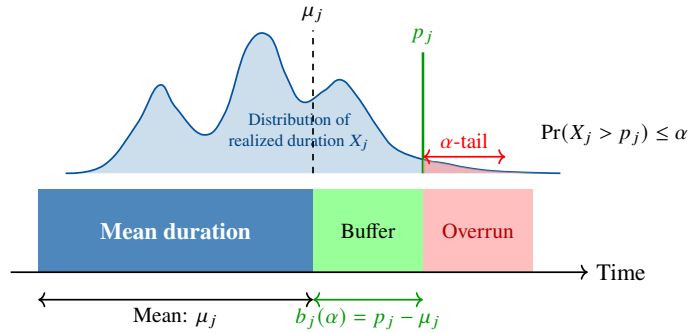

Crucially, the scheduling model does not depend on the method used to
compute buffers. It takes buffered durations as input and remains
agnostic as to whether buffers come from simple moment information or
from richer historical samples. We formalize this separation via a
\emph{modular buffer engine}: any method that maps a reliability
level and available data to a buffered duration can be plugged into
the scheduling problem without changing the optimization model. This
architecture is attractive in practice because hospitals vary widely
in data availability and modeling sophistication, yet they face the
same operational need: an implementable knob that links uncertainty
information to schedule robustness.

For the validated implementation in this paper, we use a
moment-based engine derived from Cantelli's inequality: it requires
only estimated means and variances and avoids specifying a full
duration distribution. To show that the scheduling architecture is not
tied to this particular bound, we also provide Wasserstein
distributionally robust engines that can populate the same
$\alpha$-menu when richer historical samples and an ambiguity-radius
calibration are available. We therefore treat Wasserstein as an
extensibility demonstration rather than as the empirical benchmark in
this study; \ECappref{ec:buffer-engines} of the Electronic Companion
illustrates this plug-in interface in a controlled synthetic check, and
systematic engine comparison and ambiguity-radius
calibration are left for future work.

\subsection{Risk posture at the day level}
\label{sec:day-level-posture}

Hospitals differ in how they perceive and manage execution risk, and
those differences are shaped in part by how they are paid. Under
fee-for-service, the dominant incentive is throughput, and the
relevant risk metric is average performance. Under value-based or
bundled payment contracts --- increasingly common in the United
States, Spain, and the United Kingdom --- a single severely disrupted
day can compromise quality metrics measured at the episode level,
trigger contractual penalties, and damage relationships with payers;
the relevant risk object shifts from the average week to the worst
day. To reflect this institutional heterogeneity, we model day-level
reliability through three policy modes: M1, an {\it average-reliability} mode
appropriate when throughput incentives dominate; M2, a {\it worst-day} mode that
protects against the most disrupted day, natural under value-based or
reputationally sensitive environments; and M3, a {\it hard-target} mode that
enforces an explicit maximum violation probability, aligned with
SLA-driven or regulatory contexts. We return to the institutional
rationale for this menu in Section~\ref{sec:conclusions}.

\subsection{Contributions}
\label{sec:contributions}

First, we treat per-case reliability as a decision variable. A single
mixed-integer formulation chooses each surgery's tail-risk level
$\alpha$ from a discrete menu jointly with its assignment and
sequence, so reliability is set case by case as part of the
optimization. The mapping from a reliability level to a buffered
duration is handled by an interchangeable buffer engine, so the
scheduling model is agnostic to how buffers are computed: our
validated experiments use a moment-based Cantelli engine requiring
only means and variances, and we provide Wasserstein engines
($W_\infty$, $W_1$) as plug-in extensions for data-rich settings. A
sample-path feasibility result links the day-level log-budget --- the
sum of the selected log-reliability terms on a day --- to schedule
disruption probability under the corresponding dependence assumption;
proofs and constructions are deferred to the Electronic Companion.

Second, we cast three classical day-level risk postures --- M1
(average reliability), M2 (worst-day protection), and M3 (a hard
per-day chance constraint) --- within a single MILP through a common
log-reliability proxy, and we characterize analytically when
case-level reliability improves on a common rule. A fixed-schedule
relaxation in the Electronic Companion
(\ECappref{ec:menu-value-condensed}) shows that the three postures
reduce to the same variance-sorted allocation rule and differ only in
how they price each day. Relative to any common-reliability rule, the
menu's advantage is captured  by a scale-invariant
heterogeneity index that equals one when all cases share the same
duration variance and grows as variability concentrates in fewer
cases, giving a prediction of when case-level reliability
control is worthwhile. M1--M3 make differences in payment structure,
patient mix, and contractual exposure actionable in the scheduling
model.

Third, we evaluate the framework in a rolling-origin case study at HLA
Moncloa Hospital, a medium-sized private multispecialty hospital in
Madrid, Spain. Relative to the hospital's deterministic mean-based
baseline, the proposed approach reduces delays exceeding 90 minutes by
97\%, lowers the 95th-percentile delay by approximately 620 minutes,
and cuts total overtime by 42\% across 10 surgical schedules ranging
from 45 to 227 procedures. Because this is a single-site back-test,
these results should be interpreted as evidence of the underlying
mechanism rather than as universal performance estimates. A paired
regression on the realized schedules separates and quantifies two
mechanisms: per-surgery shock absorption and the flattening of
position-dependent delay propagation across the day. By concentrating
buffers on high-uncertainty procedures, variance-aware buffer
placement suppresses the cascade that would otherwise build through
the daily sequence and stabilizes end-of-day completion.

\section{Literature Review}
\label{sec:literature-review}

Surgical services occupy a central position in hospital operations:
they are where clinical value and financial performance converge most
directly, where resource intensity is highest, and where scheduling
failures generate system-wide delay and downstream congestion that
ripple across recovery, ancillary services, and inpatient capacity
\citep{Gupta2007}.

Foundational reviews organize this literature by planning stage---from
capacity planning to schedule construction, execution, and control---and
by technical characteristics, including solution methods and the treatment
of uncertainty. The aim is to
maximize surgical throughput while keeping service reliable for
patients \citep{Cardoen2010, MaySpanglerStrumVargas2011}. Two related streams
address parts of this aim without modeling duration uncertainty
directly. The first studies the joint optimization of weekly planning
and scheduling, showing that substantial gains in capacity are
achievable at fixed resources through better assignment and
sequencing, even in deterministic models
\citep{NaderiRoshanaeiBegenAlemanUrbach2021,
RoshanaeiLuongAlemanUrbach2017}. The second studies the dynamic
allocation of aggregate operating room capacity across surgical
specialties, accounting for the evolving urgency composition of their
waiting lists; \citet{CarewNagarajanShechterArnejaSkarsgard2021}
formulate this problem as a Markov decision process. These streams are
complementary to the daily scheduling problem we study: their decisions
concern how much capacity a service receives rather than how a given
day's cases are sequenced and buffered, so duration variability is not
managed through the schedule itself.

The literature most directly relevant to our work concerns
daily surgery scheduling under uncertainty, where stochastic and
robust formulations have progressively replaced expected-value
models. Early work incorporated uncertainty through several
distinct paradigms: \citet{LamiriXieDolguiGrimaud2008} plan operating
room capacity under random emergency demand with deterministic elective
durations; \citet{DentonMillerBalasubramanianHuschka2010} develop
two-stage stochastic and robust models for assigning surgeries with
uncertain durations; \citet{GulDentonFowlerHuschkaBard2011} evaluate
sequencing and appointment-time heuristics for an outpatient procedure
center through discrete-event simulation and a bi-criteria genetic
algorithm; and \citet{FreemanMeloukMittenthal2016} propose an
explicitly scenario-based operating-theater model with uncertain
elective durations and urgent arrivals. Chance-constrained approaches
subsequently introduced service level
guarantees directly into operating room planning and scheduling
models \citep{JebaliDiabat2017, AzarCarrascoMondschein2022}. A
parallel methodological line, originating with the
convex-optimization approach of moment-class probability
bounds presented in \citet{BertsimasPopescu2005}, has developed distributionally
robust optimization (DRO) tools that require only limited
distributional information. \citet{MakRongZhang2015} apply this
machinery to single-server appointment scheduling: they derive a
closed-form min-max-optimal time allowance under marginal-moment
ambiguity, prove that under a mild condition sequencing jobs by
increasing variance is
optimal, and articulate the mechanism by which delays at the
beginning of the day propagate downstream and motivate larger
allowances for earlier jobs. In a closely related moment-based
vein, \citet{ZhouParlarVerterFraser2021} minimize the expected time
span of a surgical session subject to a cap on each patient's expected
waiting time, and derive a closed-form, distribution-free error bound
for a simple sequencing rule. As with the allowance rules above, the
service guarantee is fixed ex ante and stated in expectation, rather
than being a per-case reliability level that the schedule itself
selects. In surgical settings, robust and distributionally robust
models have represented uncertainty through deterministic uncertainty
sets \citep{BandiGupta2020}, moment-and-support ambiguity sets
\citep{WangZhangTang2019, ShehadehPadman2021}, feature-informed
ambiguity sets \citep{WangZhangZhouTang2023}, and Wasserstein balls
\citep{Shehadeh2022Wasserstein, WangZhangTang2024}. Closest to our setting are
\citet{DengShenDenton2019}, who impose distributionally robust chance
constraints on waiting and overtime over a $\phi$-divergence ambiguity
set while choosing which rooms to open and how to allocate, sequence
and time cases; \citet{TsangEtAl2025}, who jointly optimize room and
anesthesiologist allocation, assignment, sequencing and scheduling
under both stochastic and distributionally robust formulations; and
\citet{FuQiYangYe2024}, who develop
a punctuality-index objective for elective sequencing and scheduling
under uncertainty, with an explicit distributionally robust extension.

Across these streams, uncertainty is represented through scenarios,
chance constraints, robust uncertainty sets, or distributional
ambiguity sets. In the closest formulations, however, the service
level, risk calibration, or robustness criterion is fixed before the
scheduling problem is solved or embedded in a fixed objective, and the
reliability actually delivered is an outcome rather than a controlled
quantity. Our formulation departs from this in three ways. First, each
surgery's reliability is selected from a discrete, interpretable menu
of tail-risk levels, so protection can differ case by case and mirror
the small set of service levels a planner can govern. Second, the
selected case-level guarantees combine into a day-level reliability
budget that a single MILP can enforce under three risk postures ---
average reliability, worst-day protection, or a hard daily target ---
while jointly deciding assignment, sequencing, and timing across rooms
and surgeons, rather than sequencing a single server. Third, we
characterize analytically when this case-level control is worthwhile:
relative to any common service level, its advantage is governed by a
scale-invariant measure of how unevenly duration variability is spread
across a day's caseload, a comparison none of the cited models make.
Because the reliability menu is separated from the engine that
computes buffers, the same formulation admits moment-based Cantelli or
data-driven Wasserstein inputs.

\section{Uncertainty Model and Reliability Menu}
\label{sec:uncertainty-model}

Our robust approach is built around a modular buffer engine: a
mapping that, given a reliability level and case-specific
information, produces a deterministic buffered duration.%

We work with a discrete reliability menu
$\mathcal{T}=\{1,\ldots,T\}$, with associated tail probabilities
$0<\alpha_1<\cdots<\alpha_T<1$ (smaller $\alpha_t$ is more conservative).
For each feasible assignment of surgery $j$ to operating room $r$ and
surgeon $k$, let $X_{jrk}$ denote the random duration (full index
sets and decision variables are introduced in Section~\ref{sec:model}).

A buffer engine maps $(\alpha_t,\text{available data})$ into a
buffered duration $p_{jrk}^{(t)}$ such that
\begin{equation}
  \label{eq:engine-guarantee}
  \Pr\!\bigl(X_{jrk} \le p_{jrk}^{(t)}\bigr)\;\ge\; 1-\alpha_t.
\end{equation}
The collection $\{p_{jrk}^{(t)}\}_{t\in\mathcal{T}}$ provides a menu
of planning times that spans a spectrum of conservatism: choosing a
smaller $\alpha_t$ increases protection but consumes more capacity.
The MILP scheduler selects one level $t$ for each scheduled case,
trading off efficiency against execution risk. Importantly, the
scheduling formulation is buffer-agnostic: it uses
$\{p_{jrk}^{(t)}\}$ as inputs regardless of which engine produced
them. We later translate these case-level guarantees into day-level
reliability controls using aggregate budget constraints.

\begin{figure}[ht]
  \centering
  \resizebox{0.9\textwidth}{!}{%
    \begin{tikzpicture}[
        font=\small,
        box/.style={draw, rounded corners=2pt, align=center, inner sep=5pt,
        minimum height=14mm, line width=0.7pt},
        mode/.style={draw=cModeB, fill=cMode, rounded corners=2pt, align=center,
          inner sep=4pt, font=\footnotesize, text width=20mm,
        minimum height=11mm, line width=0.7pt},
        arr/.style={-{Latex[length=2.4mm]}, semithick, black!60},
        lbl/.style={font=\scriptsize\itshape, text=cModeB}
      ]

      \node[box, draw=cMenuB, fill=cMenu, text width=25mm] (menu)
      {$\alpha$-menu\\[2pt] $\{0.01,\ldots,0.10\}$};
      \node[box, draw=cEngB, fill=cEng, right=8mm of menu, text
      width=23mm] (engine)
      {Buffer engine\\[2pt] Cantelli, $W_\infty$, or $W_1$};
      \node[box, draw=cBufB, fill=cBuf, right=8mm of engine, text
      width=21mm] (buf)
      {Buffered durations\\[2pt] $p^{(t)}_{jrk}$};
      \node[box, draw=cMilpB, fill=cMilp, right=8mm of buf, text
      width=27mm] (milp)
      {MILP scheduler\\[2pt] assignment, sequence,\\ \mbox{reliability level}};
      \node[box, draw=cOutB, fill=cOut, right=8mm of milp, align=center] (out)
      {Day-level reliability\\[3pt]
        $\displaystyle\prod_{j\in\mathcal{S}_d}\!(1-\alpha_j)\ge
      1-\varepsilon$\quad$\forall d$};

      \draw[arr] (menu) -- (engine);
      \draw[arr] (engine) -- (buf);
      \draw[arr] (buf) -- (milp);
      \draw[arr] (milp) -- (out);

      \node[mode, below=10mm of milp] (m2)
      {\textbf{M2} worst-day\\ risk averse};
      \node[mode, left=4mm of m2]  (m1)
      {\textbf{M1} average\\ cost sensitive};
      \node[mode, right=4mm of m2] (m3)
      {\textbf{M3} fixed\\ regulatory compliance};

      \draw[arr, cModeB] (m1.north) -- ([xshift=-9mm]milp.south);
      \draw[arr, cModeB] (m2.north) -- (milp.south);
      \draw[arr, cModeB] (m3.north) -- ([xshift=9mm]milp.south);

      \node[lbl, below=1.5mm of m2] {risk posture $\Psi(\cdot)$};

    \end{tikzpicture}
  }
  \caption{Modular architecture. A buffer engine maps an
    $\alpha$-menu level and case data to buffered durations, which the
  MILP scheduler consumes as parameters.}
  \label{fig:modular-architecture}
\end{figure}
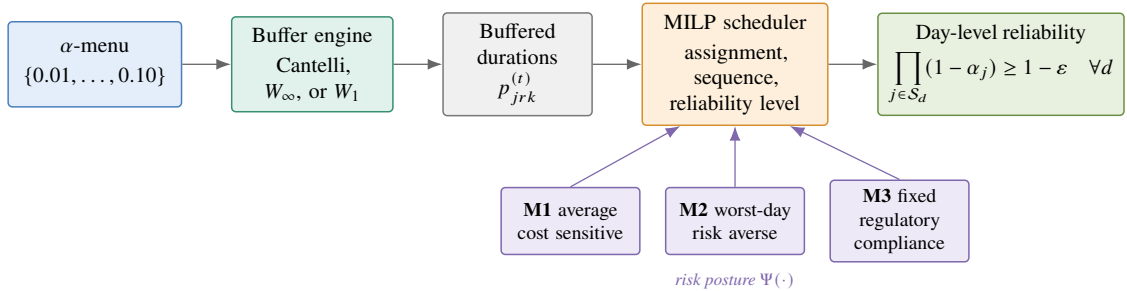

We next present three engines---a moment-based Cantelli
engine and two data-driven Wasserstein engines---that instantiate
\eqref{eq:engine-guarantee} in different information environments.

\subsection{Buffer Engine 1: Cantelli (Moment-Based)}
\label{sec:cantelli-engine}

We assume that, for each case, we can estimate the mean $\mu$ and
variance $\sigma^2$ of its duration from historical records or expert
judgment; the Cantelli engine then constructs buffered durations
using only these two moments, without imposing any distributional form.
\begin{lemma}[Cantelli's Inequality]
  \label{lem:cantelli}
  Let $X$ be a real random variable with mean $\mu$ and variance
  $\sigma^2 \in (0,\infty)$. Then for any $a>0$,
  \begin{equation}
    \label{eq:cantelli}
    \Pr(X-\mu \ge a) \;\le\; \frac{\sigma^2}{\sigma^2+a^2}.
  \end{equation}
  Equivalently,
  \begin{equation}
    \label{eq:cantelli-equiv}
    \Pr(X \le \mu+a) \;\ge\; \frac{a^2}{\sigma^2+a^2}.
  \end{equation}
\end{lemma}

\noindent
Lemma~\ref{lem:cantelli} is classical; we omit its standard proof and record
only the extremal distribution that shows the bound is tight in \ECappref{ec:cantelli}.

\begin{corollary}[Buffer achieving tail probability $\alpha$]
  \label{cor:buffer}
  Fix $\alpha\in(0,1)$. If $X$ has mean $\mu$ and standard deviation
  $\sigma>0$, then with $a=\sigma\sqrt{(1-\alpha)/\alpha}$ we obtain
  $\Pr(X \le \mu+a)\ge 1-\alpha$.
\end{corollary}

For each feasible $(j,r,k)$ and each menu level $t\in\mathcal{T}$,
the Cantelli engine sets $p_{jrk}^{(t)} = \mu_{jrk} +
\sqrt{(1-\alpha_t)/\alpha_t}\,\sigma_{jrk}$.
By Corollary~\ref{cor:buffer}, this choice satisfies the engine
guarantee \eqref{eq:engine-guarantee} for any duration distribution
with the specified mean and variance. Note that Cantelli buffering is
distribution-free but can be conservative because it protects against
worst-case distributions in the moment class.

\subsection{Buffer Engine 2: Wasserstein-$W_\infty$ (Max-Shift)}
\label{sec:Winf-engine}

Given $N$ historical durations $x_1\le \cdots \le x_N$ for a case and an
ambiguity radius $\rho\ge 0$, the Wasserstein-$W_\infty$ ambiguity set contains
every distribution $Q$ obtainable from the empirical
$\hat P_N=\frac{1}{N}\sum_{i=1}^N \delta_{x_i}$ by moving no unit of probability
mass more than $\rho$ (equivalently, $W_\infty(Q,\hat P_N)\le\rho$). The worst
case is a uniform right-shift of the sample by $\rho$, which yields a
closed-form
buffered duration.

\begin{proposition}[$W_\infty$ buffered duration]
  \label{prop:Winf-buffer}
  Let $\hat Q_{1-\alpha}:=\inf\{t:\hat F_N(t)\ge 1-\alpha\}$ be the empirical
  $(1-\alpha)$-quantile of $\hat P_N$ (left-quantile convention). Then the
  smallest buffered duration $p$ satisfying
  $\sup_{Q:\,W_\infty(Q,\hat P_N)\le \rho} Q(X>p)\le\alpha$ is
  \begin{equation}
    \label{eq:Winf-buffer}
    p^{(t)} = \hat Q_{1-\alpha_t} + \rho.
  \end{equation}
\end{proposition}

Computing \eqref{eq:Winf-buffer} requires only sorting the samples
and reading off
a quantile: it is as if every observed duration could be up to $\rho$ minutes
longer than in the historical record. A short proof is in
\ECappref{ec:winf}.

\subsection{Buffer Engine 3: Wasserstein-$W_1$ (Average-Shift)}
\label{sec:W1-engine}

The Wasserstein-$W_1$ engine uses the same samples and radius but bounds the
\emph{average} displacement rather than the maximum: $W_1(Q,\hat P_N)\le\rho$
permits moving a small amount of mass far, provided the mass-weighted average
displacement stays within $\rho$. The worst-case exceedance probability above a
threshold $\tau$, $\overline{\Phi}_\rho(\tau):=\sup_{Q:\,W_1(Q,\hat P_N)\le\rho}
Q(X>\tau)$, is a fractional-knapsack value.

\begin{proposition}[Worst-case tail over a $W_1$-ball]
  \label{prop:W1-tail}
  Let $x_1\le \cdots \le x_N$ and $\rho> 0$, and set
  $I(\tau):=\{\,i: x_i\le \tau\,\}$. For any $\tau\ge 0$,
  \begin{equation}
    \label{eq:W1-knapsack}
    \begin{aligned}
      \overline{\Phi}_\rho(\tau)
      =&\; \frac{1}{N}\sum_{i=1}^N \mathbf{1}\{x_i>\tau\}
      \;+\; \frac{1}{N}\max_{z}\ \sum_{i\in I(\tau)} z_i \\[2pt]
      \text{s.t. }&
      \sum_{i\in I(\tau)} (\tau-x_i)\,z_i \le N\rho,\quad
      0\le z_i \le 1 .
    \end{aligned}
  \end{equation}
  The maximization is a fractional knapsack, solved greedily by spending the
  transport budget on the samples closest to $\tau$ from below; given sorted
  samples it evaluates in $O(N)$ time.
\end{proposition}

The proof---which transports mass from just below $\tau$ into the
tail---is given
in \ECappref{ec:w1-proof}. The $W_1$ buffered duration is the smallest
threshold whose worst-case tail does not exceed $\alpha_t$,
\begin{equation}
  \label{eq:W1-buffer}
  p^{(t)}=\inf\{\tau\ge 0:\overline{\Phi}_\rho(\tau)\le \alpha_t\},
\end{equation}
computed by bisection on $\tau\in[x_1,\,x_N+\rho/\alpha_t]$ (the
  upper end because
$\overline{\Phi}_\rho(\tau)\le\rho/(\tau-x_N)$ for $\tau>x_N$).

\begin{remark}[Role and calibration of $\rho$]
  \label{rem:rho-role}
  The radius $\rho$ controls robustness: $\rho=0$ recovers the
  empirical-quantile engines, and larger $\rho$ weakly inflates buffered
  durations to hedge sampling error or distributional shift. Its meaning is
  engine-specific ($W_\infty$ caps the maximum displacement, $W_1$ the average),
  so it should be calibrated separately per engine---e.g., by out-of-sample
  validation, as \ECappref{ec:buffer-engines} does before inserting the
  buffered durations into the same scheduling model.
\end{remark}

Table~\ref{tab:buffer-engine-comparison} summarizes and contrasts the
three buffer engines.  They differ only in how the buffered durations
\(p^{(t)}_{jrk}\) are produced; once these values are computed, the same
mixed-integer scheduling model applies.

\begin{table}[ht]
  \centering
  \caption{Comparison of the three buffer engines.}
  \label{tab:buffer-engine-comparison}
  \small
  \setlength{\tabcolsep}{6pt}
  \renewcommand{\arraystretch}{1.25}
  \begin{tabular}{@{} >{\raggedright\arraybackslash}p{0.13\linewidth}
      >{\raggedright\arraybackslash}p{0.17\linewidth}
      >{\raggedright\arraybackslash}p{0.34\linewidth}
    >{\raggedright\arraybackslash}p{0.24\linewidth} @{}}
    \toprule
    \rowcolor{hdrgray}\textbf{Engine} & \textbf{Required input} & \textbf{Worst-case
    interpretation} & \textbf{Main tradeoff} \\
    \midrule
    Cantelli (moment-based)
    & Mean and variance $(\mu,\sigma^2)$
    & Distribution-free bound: protects against every law with the
    given first two moments
    & Broadly applicable, but conservative \\
    $W_\infty$ (max-displacement)
    & Samples and radius $\rho$
    & Caps the maximum displacement of probability mass; upper-tail
    worst case is a uniform right-shift by $\rho$
    & Simple max-shift reading; sensitive to worst-case pointwise shift \\
    $W_1$ (average-displacement)
    & Samples and radius $\rho$
    & Caps the mass-weighted average displacement; the adversary
    budgets transport to stress the upper tail
    & More data-adaptive tail stress; less interpretable than $W_\infty$ \\
    \bottomrule
  \end{tabular}
\end{table}

The buffer engines deliver per-surgery tail guarantees, but
hospitals evaluate reliability at the day level.
Section~\ref{sec:model} closes this gap: it aggregates per-surgery
tail probabilities into a day-level reliability constraint and
presents the MILP, with three risk-posture variants (M1, M2, M3)
for alternative day-level reliability targets.

\section{Models Description and Formulations}
\label{sec:model}

Table~\ref{tab:notation} summarizes the notation. For compactness,
define $\mathcal{V}_j:=\{(d,r,k,t):(j,d,r,k,t)\in\mathcal{V}\}$,
$\mathcal{V}_{jdr}:=\{(k,t):(j,d,r,k,t)\in\mathcal{V}\}$,
$\mathcal{V}_{jdk}:=\{(r,t):(j,d,r,k,t)\in\mathcal{V}\}$,
$\mathcal{V}_d:=\{(j,r,k,t):(j,d,r,k,t)\in\mathcal{V}\}$.

{\footnotesize\renewcommand{\arraystretch}{0.9}
  \makeatletter\let\@makecaption\@maketablecaption\makeatother
  \captionof{table}{Notation.}\label{tab:notation}
  \begin{longtable}{@{}l p{0.62\linewidth}@{}}

    \toprule
    \rowcolor{hdrgray}\multicolumn{2}{@{}l}{\textit{Sets and indices}}\\
    \midrule
    $j\in\mathcal{J},\ d\in\mathcal{D},\ r\in\mathcal{R},\ k\in\mathcal{K}$
    & surgeries, days, operating rooms, surgeons\\
    $t\in\mathcal{T}=\{1,\ldots,T\}$ & reliability-menu level (tail
    probability $\alpha_t$)\\
    $\mathcal{V}\subseteq\mathcal{J}\!\times\!\mathcal{D}\!\times\!\mathcal{R}\!\times\!\mathcal{K}\!\times\!\mathcal{T}$
    & feasible assignment tuples $(j,d,r,k,t)$\\
    $\mathcal{V}_{\!JDK}\subseteq\mathcal{J}\!\times\!\mathcal{D}\!\times\!\mathcal{K}$
    & induced feasible triples $(j,d,k)$\\
    $\mathcal{V}_j,\ \mathcal{V}_{jdr},\ \mathcal{V}_{jdk},\ \mathcal{V}_d$
    & restricted feasible-index sets used in summations\\
    $j<i$ & index convention for pairwise sequencing variables\\
    \addlinespace[4pt]
    \rowcolor{hdrgray}\multicolumn{2}{@{}l}{\textit{Parameters}}\\
    \midrule
    $p_{jrk}^{(t)}$ & buffered duration of surgery $j$ at level $t$ (min)\\
    $\alpha_t$ & tail probability for menu level $t$\\
    $\ell_t=\log(1-\alpha_t)$ & log-reliability constant for menu level $t$\\
    $H_d$ & regular room horizon on day $d$ (min)\\
    $C_{kd}$ & surgeon $k$ regular capacity on day $d$ (min)\\
    $\bar O^{\mathrm{rm}},\ \bar O^{\mathrm{sg}}$ & maximum overtime
    per room and per surgeon (min)\\
    $M_d^{\mathrm{rm}},\ M_d^{\mathrm{sg}},\ M^{\mathrm{link}}$ &
    big-$M$ constants for room, surgeon, and synchronization constraints\\
    $c^{\mathrm{idle}},\ c^{\mathrm{rm}},\ c^{\mathrm{sg}},\ w_\varepsilon$
    & cost coefficients (idle, room overtime, surgeon overtime,
    reliability weight)\\
    \addlinespace[4pt]
    \rowcolor{hdrgray}\multicolumn{2}{@{}l}{\textit{Decision variables}}\\
    \midrule
    $w_{jdrkt}\in\{0,1\}$ & 1 if surgery $j$ is assigned to day $d$,
    room $r$, surgeon $k$, level $t$\\
    $z_{jdr}\in\{0,1\}$ & 1 if surgery $j$ is assigned to room $r$ on day $d$\\
    $y_{jdk}\in\{0,1\}$ & 1 if surgery $j$ is assigned to surgeon $k$
    on day $d$\\
    $s^{\mathrm{rm}}_{jdr},\ s^{\mathrm{sg}}_{jdk}\ge 0$ & start
    times on the room and surgeon timelines (min)\\
    $u^{\mathrm{rm}}_{ji\,dr},\ u^{\mathrm{sg}}_{ji\,dk}\in\{0,1\}$ &
    pairwise sequencing variables ($j<i$)\\
    $\delta^{\mathrm{rm}}_{jdr},\ \delta^{\mathrm{sg}}_{jdk}\ge 0$ &
    buffered durations induced on the room and surgeon timelines (min)\\
    $o^{\mathrm{rm}}_{dr},\ o^{\mathrm{sg}}_{dk}\ge 0$ & room and
    surgeon overtime (min)\\
    $I_{dr}\ge 0$ & unused room capacity (min)\\
    $b_d\le 0,\ b_{\min}\le 0$ & day-level and worst-day log-budgets
    ($b_{\min}$ used in M2 only)\\
    \addlinespace[4pt]
    \rowcolor{hdrgray}\multicolumn{2}{@{}l}{\textit{Derived quantities}}\\
    \midrule
    $\varepsilon_d:=1-e^{b_d},\ \varepsilon_{\max}:=\max_{d\in\mathcal{D}}\varepsilon_d$
    & day-level and worst-day violation proxies\\
    $\mathcal J_d$ & surgeries assigned to day $d$ by the
    schedule (those $j$ with $\sum_{(r,k,t)}w_{jdrkt}=1$)\\
    \bottomrule
\end{longtable}\addtocounter{table}{-1}}

We first present the compact core formulation, then the additional
activation and non-overlap constraint blocks of the implemented MILP.

\begingroup
\setlength{\abovedisplayskip}{4pt}
\setlength{\belowdisplayskip}{4pt}
\setlength{\abovedisplayshortskip}{4pt}
\setlength{\belowdisplayshortskip}{4pt}
\allowdisplaybreaks
\setlength{\jot}{2pt}
\begin{subequations}
  \label{eq:core-formulation}
  \begin{align}
    (\text{M}_{c}) \min\
    & c^{\mathrm{idle}}\sum_{d,r} I_{dr}
    + c^{\mathrm{rm}}\sum_{d,r} o^{\mathrm{rm}}_{dr}\notag\\*
    & + c^{\mathrm{sg}}\sum_{d,k} o^{\mathrm{sg}}_{dk}
    + \Psi(\cdot)
    \label{eq:core-obj}\\
    \text{s.t.}\quad
    & \sum_{(d,r,k,t)\in\mathcal V_j} w_{jdrkt} = 1
    \label{eq:core-assign}\\
    & z_{jdr} = \sum_{(k,t)\in\mathcal V_{jdr}} w_{jdrkt}
    \label{eq:core-z}\\
    & y_{jdk} = \sum_{(r,t)\in\mathcal V_{jdk}} w_{jdrkt}
    \label{eq:core-y}\\
    & \delta^{\mathrm{rm}}_{jdr}
    = \sum_{(k,t)\in\mathcal V_{jdr}} p_{jrk}^{(t)} w_{jdrkt}
    \label{eq:core-drm}\\
    & \delta^{\mathrm{sg}}_{jdk}
    = \sum_{(r,t)\in\mathcal V_{jdk}} p_{jrk}^{(t)} w_{jdrkt}
    \label{eq:core-dsg}\\
    & s^{\mathrm{sg}}_{jdk}
    \ge s^{\mathrm{rm}}_{jdr} - M^{\mathrm{link}}(1-w_{jdrkt})
    \label{eq:core-sync-lb}\\
    & s^{\mathrm{sg}}_{jdk}
    \le s^{\mathrm{rm}}_{jdr} + M^{\mathrm{link}}(1-w_{jdrkt})
    \label{eq:core-sync-ub}\\
    & s^{\mathrm{rm}}_{jdr} + \delta^{\mathrm{rm}}_{jdr}
    \le H_d + o^{\mathrm{rm}}_{dr} + M_d^{\mathrm{rm}}(1-z_{jdr})
    \label{eq:core-room-comp}\\
    & \sum_{j:(j,d,k)\in\mathcal V_{\!JDK}} \delta^{\mathrm{sg}}_{jdk}
    \le C_{kd} + o^{\mathrm{sg}}_{dk}
    \label{eq:core-sg-cap}\\
    & s^{\mathrm{sg}}_{jdk} + \delta^{\mathrm{sg}}_{jdk}
    \le C_{kd} + o^{\mathrm{sg}}_{dk} + M_d^{\mathrm{sg}}(1-y_{jdk})
    \label{eq:core-sg-comp}\\
    & I_{dr} = H_d + o^{\mathrm{rm}}_{dr}
    - \sum_{j\in\mathcal J}\delta^{\mathrm{rm}}_{jdr}
    \label{eq:core-idle}\\
    & 0 \le o^{\mathrm{rm}}_{dr} \le \bar O^{\mathrm{rm}}
    \label{eq:core-orm}\\
    & 0 \le o^{\mathrm{sg}}_{dk} \le \bar O^{\mathrm{sg}}
    \label{eq:core-osg}\\
    & \sum_{(j,r,k,t)\in\mathcal V_d} \ell_t w_{jdrkt} \ge b_d
    \label{eq:core-rel}
  \end{align}
\end{subequations}
\endgroup

\noindent{\it Big-M.} Let
$P_{\max}:=\max\{p_{jrk}^{(t)}:(j,d,r,k,t)\in\mathcal V\}$; we set
$M_d^{\mathrm{rm}}=H_d+\bar O^{\mathrm{rm}}+P_{\max}$,
$M_d^{\mathrm{sg}}=\max_{k\in\mathcal K}C_{kd}+\bar O^{\mathrm{sg}}+P_{\max}$,
and $M^{\mathrm{link}}=\max_{d\in\mathcal D}H_d+\bar O^{\mathrm{rm}}$.

\noindent\textit{Index domains.}
Constraint~\eqref{eq:core-assign} applies for all $j\in\mathcal J$.
Constraints~\eqref{eq:core-z}, \eqref{eq:core-drm}, and
\eqref{eq:core-room-comp}
apply for all $j\in\mathcal J$, $d\in\mathcal D$, and $r\in\mathcal R$.
Constraints~\eqref{eq:core-y}, \eqref{eq:core-dsg}, and \eqref{eq:core-sg-comp}
apply for all $(j,d,k)\in\mathcal V_{\!JDK}$.
Constraints~\eqref{eq:core-sync-lb}--\eqref{eq:core-sync-ub}
apply for all $(j,d,r,k,t)\in\mathcal V$.
Constraints~\eqref{eq:core-sg-cap} and \eqref{eq:core-osg}
apply for all $d\in\mathcal D$ and $k\in\mathcal K$.
Constraints~\eqref{eq:core-idle} and \eqref{eq:core-orm}
apply for all $d\in\mathcal D$ and $r\in\mathcal R$.
Constraint~\eqref{eq:core-rel} applies for all $d\in\mathcal D$.

The core formulation captures the main structure of the problem.
Constraint~\eqref{eq:core-assign} assigns each surgery exactly once,
jointly selecting its day, room, surgeon, and reliability level.
Constraints~\eqref{eq:core-z}--\eqref{eq:core-dsg} define the induced
room and surgeon assignment indicators and the corresponding buffered
durations.
Constraints~\eqref{eq:core-sync-lb}--\eqref{eq:core-sync-ub}
synchronize the room and surgeon start times for the chosen
assignment. Constraints~\eqref{eq:core-room-comp} and
\eqref{eq:core-sg-cap}--\eqref{eq:core-sg-comp} ensure that buffered
workloads fit within room and surgeon capacity, allowing bounded
overtime. Constraints~\eqref{eq:core-idle}--\eqref{eq:core-osg}
define idle time and bound overtime: $I_{dr}$ measures unused
staffed room time relative to the effective horizon
$H_d+o^{\mathrm{rm}}_{dr}$, while $o^{\mathrm{rm}}_{dr}$ prices any
extension beyond regular hours; because the objective penalizes both
$I_{dr}$ and $o^{\mathrm{rm}}_{dr}$, extending the room horizon without
scheduled work is never attractive. Finally,
constraint~\eqref{eq:core-rel} aggregates the selected menu levels
into a day-level log-budget, which is then controlled differently by
the model variants. The risk term $\Psi(\cdot)$ in the objective
takes one of three forms, corresponding to the three model variants
developed in Section~\ref{sec:three-models}.

The four additional blocks --- start-time activation,
sequencing-variable activation, and room and surgeon non-overlap ---
follow. The two activation blocks are linking
constraints: they do not change the feasible set of integer schedules,
but they tighten the linear programming relaxation and improve
branch-and-bound convergence in our implementation. The room and
surgeon non-overlap blocks are structural: they enforce the pairwise
disjunctions that complete the model, and the sample-path guarantee of
Section~\ref{sec:probabilistic-interp} relies on them.

\begingroup
\setlength{\abovedisplayskip}{6pt plus 1pt minus 1pt}
\setlength{\belowdisplayskip}{6pt plus 1pt minus 1pt}
\setlength{\abovedisplayshortskip}{4pt plus 1pt minus 1pt}
\setlength{\belowdisplayshortskip}{6pt plus 1pt minus 1pt}
\setlength{\jot}{2pt}

\paragraph{Start-time activation constraints.}
To prevent unassigned surgeries from carrying arbitrary start times, we add
\begin{align}
  s^{\mathrm{rm}}_{jdr} &\le M_d^{\mathrm{rm}} z_{jdr}
  \label{eq:s-rm-act}\\
  s^{\mathrm{sg}}_{jdk} &\le M_d^{\mathrm{sg}} y_{jdk}
  \label{eq:s-sg-act}
\end{align}
Constraint~\eqref{eq:s-rm-act} applies for all $j\in\mathcal J$,
$d\in\mathcal D$, and $r\in\mathcal R$, while
constraint~\eqref{eq:s-sg-act} applies for all $(j,d,k)\in\mathcal
V_{\!JDK}$. These inequalities force start times to zero unless the
corresponding room-day or surgeon-day assignment is selected.

\paragraph{Sequencing-variable activation constraints.}
To ensure that pairwise sequencing variables are active only when
both surgeries are assigned to the same resource on the same day, we impose
\begin{align}
  u^{\mathrm{rm}}_{ji\,dr} &\le z_{jdr}
  \label{eq:u-rm-link1}\\
  u^{\mathrm{rm}}_{ji\,dr} &\le z_{idr}
  \label{eq:u-rm-link2}\\
  u^{\mathrm{sg}}_{ji\,dk} &\le y_{jdk}
  \label{eq:u-sg-link1}\\
  u^{\mathrm{sg}}_{ji\,dk} &\le y_{idk}
  \label{eq:u-sg-link2}
\end{align}
Constraints~\eqref{eq:u-rm-link1}--\eqref{eq:u-rm-link2} apply for
all $j<i$, $d\in\mathcal D$, and $r\in\mathcal R$, and
constraints~\eqref{eq:u-sg-link1}--\eqref{eq:u-sg-link2} apply for
all $j<i$, $d\in\mathcal D$, and $k\in\mathcal K$. These inequalities
force the sequencing variables to zero whenever at least one of the
two surgeries is absent from the corresponding room-day or surgeon-day timeline.

\paragraph{Room non-overlap constraints.}
The room timeline is made conflict-free through the pairwise
disjunctive constraints
\begin{align}
  s^{\mathrm{rm}}_{jdr} + \delta^{\mathrm{rm}}_{jdr}
  &\le s^{\mathrm{rm}}_{idr}
  + M_d^{\mathrm{rm}}(1-u^{\mathrm{rm}}_{ji\,dr})
  \notag\\
  &\quad + M_d^{\mathrm{rm}}(2-z_{jdr}-z_{idr})
  \label{eq:room-seq1}\\
  s^{\mathrm{rm}}_{idr} + \delta^{\mathrm{rm}}_{idr}
  &\le s^{\mathrm{rm}}_{jdr}
  + M_d^{\mathrm{rm}}u^{\mathrm{rm}}_{ji\,dr}
  \notag\\
  &\quad + M_d^{\mathrm{rm}}(2-z_{jdr}-z_{idr})
  \label{eq:room-seq2}
\end{align}
for all $j<i$, $d\in\mathcal D$, and $r\in\mathcal R$. These
constraints enforce that, whenever two surgeries are assigned to the
same room on the same day, one must precede the other without overlap.

\paragraph{Surgeon non-overlap constraints.}
Similarly, surgeon availability is enforced through
\begin{align}
  s^{\mathrm{sg}}_{jdk} + \delta^{\mathrm{sg}}_{jdk}
  &\le s^{\mathrm{sg}}_{idk}
  + M_d^{\mathrm{sg}}(1-u^{\mathrm{sg}}_{ji\,dk})
  \notag\\
  &\quad + M_d^{\mathrm{sg}}(2-y_{jdk}-y_{idk})
  \label{eq:sg-seq1}\\
  s^{\mathrm{sg}}_{idk} + \delta^{\mathrm{sg}}_{idk}
  &\le s^{\mathrm{sg}}_{jdk}
  + M_d^{\mathrm{sg}}u^{\mathrm{sg}}_{ji\,dk}
  \notag\\
  &\quad + M_d^{\mathrm{sg}}(2-y_{jdk}-y_{idk})
  \label{eq:sg-seq2}
\end{align}
for all $j<i$, $d\in\mathcal D$, and $k\in\mathcal K$. These
constraints ensure that no surgeon is assigned overlapping surgeries
on the same day.
\endgroup

\noindent
Having defined the base model \eqref{eq:core-obj}--\eqref{eq:sg-seq2}
as M$_b$, which schedules cases using buffered durations and
day-level reliability accounting, we next clarify how these two
ingredients translate into execution feasibility and probabilistic risk control.

\subsection{Probabilistic Interpretation}
\label{sec:probabilistic-interp}

The day-level log-budget $b_d$ in constraint~\eqref{eq:core-rel}
admits a probabilistic interpretation that links the MILP feasibility
region to a chance constraint on schedule disruption, developed below
in three steps: an independence-based bound, a dependence-robust
union-bound alternative, and a sample-path feasibility result.

\paragraph{Day-level reliability under independence.}
For each surgery $j\in\mathcal{J}_d$ --- writing $\mathcal{J}_d$
for the set of surgeries the schedule assigns to day $d$, i.e., those
$j$ with $\sum_{(r,k,t)}w_{jdrkt}=1$ ---, let $E_j := \{X_j \le p_j\}$
denote the event that the realized duration $X_j$ does not exceed the
buffered duration $p_j$ assigned by the chosen menu level $t(j)$, and
let $E_d:=\bigcap_{j\in\mathcal{J}_d}\{X_j\le p_j\}$ denote the event
that every surgery on day $d$ finishes within its buffer. Each menu level
guarantees $\Pr(E_j^c) \le \alpha_{t(j)}$. Under conditional
independence of the within-buffer events across surgeries on day $d$,
and using $b_d \le \sum_{j\in\mathcal{J}_d}\ell_{t(j)}$ (enforced
by constraint~\eqref{eq:core-rel}, and tight whenever the reliability
reward makes it bind) with $\ell_t =
\log(1-\alpha_t)$,
\[
  \Pr(E_d) \;\ge\; \prod_{j\in\mathcal{J}_d}(1-\alpha_{t(j)}) \;\ge\; e^{b_d},
\]
and therefore
\begin{equation}
  \label{eq:eps-bound}
  \Pr(E_d^c) \;\le\; 1 - e^{b_d} \;=:\; \varepsilon_d.
\end{equation}
Hence $\varepsilon_d$ is an upper bound on the day-level violation
probability under independence: the log-budget
constraint~\eqref{eq:core-rel} enforces
$b_d\le\sum_{j\in\mathcal{J}_d}\ell_{t(j)}$, so
$\varepsilon_d \ge 1-\prod_{j\in\mathcal{J}_d}(1-\alpha_{t(j)})
\ge \Pr(E_d^c)$.

\paragraph{Dependence-robust alternative.}
The independence assumption is convenient but not always credible:
durations may be correlated through shared teams, equipment,
workflows, or patient-mix effects. The marginal Cantelli guarantee
$\Pr(E_j^c) \le \alpha_{t(j)}$ remains valid regardless of
dependence, but the product formula does not. A dependence-robust
bound is obtained from Boole's inequality:
\begin{equation}
  \label{eq:union-bound-main}
  \Pr(E_d^c) \;\le\; \sum_{j\in\mathcal J_d}\Pr(E_j^c) \;\le\;
  \sum_{j\in\mathcal J_d}\alpha_{t(j)}.
\end{equation}
Replacing the log-budget~\eqref{eq:core-rel} with the linear budget
$\sum_{j\in\mathcal J_d}\alpha_{t(j)} \le \varepsilon$ guarantees
$\Pr(E_d^c) \le \varepsilon$ without any assumption on the joint
distribution. The linear budget is more conservative than the
log-budget when violations are nearly independent, but it is
worst-case tight under arbitrary dependence: no smaller
dependence-robust budget can be guaranteed from the marginal Cantelli
bounds alone. The formal tightness construction is given in
\ECappref{sec:ec-bonferroni} of the Electronic Companion.

\paragraph{From buffer events to schedule feasibility.}
Both bounds control buffer violations; the next proposition shows these imply feasibility of the realized schedule.
\begin{proposition}[Sample-path feasibility and day-level risk bound]
  \label{prop:sample-path}
  Fix a day \(d\) and consider any feasible solution the
  buffered model --- the core block~\eqref{eq:core-formulation}
  together with the activation and non-overlap
  blocks~\eqref{eq:s-rm-act}--\eqref{eq:sg-seq2} --- under any of the
  three posture objectives. Let \(t(j)\)
  denote the reliability-menu level selected for surgery \(j\in
  \mathcal J_d\). Let \(F_d\) denote the event that the realized schedule on day \(d\)
  is feasible, i.e., it has no room-timeline overlap, no surgeon-timeline
  overlap, and completes within the planned room and surgeon horizons, including
  allowed overtime. Then
  $
  E_d \subseteq F_d.
  $
  Consequently,
  $
  \Pr(F_d^c)\le \Pr(E_d^c).
  $
  Under the independence aggregation of
  Section~\ref{sec:probabilistic-interp}, if
  $
  B_d := \sum_{j\in \mathcal J_d}\log(1-\alpha_{t(j)}),
  $
  then
  $
  \Pr(F_d^c)
  \le
  \Pr(E_d^c)
  \le
  1-e^{B_d}.
  $
  Since the model enforces \(b_d\le B_d\) through
  constraint~\eqref{eq:core-rel},
  substituting \(b_d\) for \(B_d\) yields the weaker but still valid bound
  $
  \Pr(F_d^c)\le 1-e^{b_d},
  $
  expressed in the budget variable the MILP actually controls. Under
  arbitrary dependence, the same sample-path
  inclusion combined
  with the union bound gives
  $
  \Pr(F_d^c)\le \sum_{j\in \mathcal J_d}\alpha_{t(j)}.
  $
\end{proposition}

The full  proof is in
\ECappref{ec:samplepath} of the Electronic Companion.  The
inclusion may be strict: a schedule can remain feasible even when
some surgery exceeds its buffer, provided downstream slack absorbs the excess.

\subsection{Risk Posture and Model Variants}
\label{sec:three-models}

Hospitals differ in how they manage execution risk in elective
surgery schedules. Some prioritize average performance over the
horizon (accepting that some days may be more fragile), while others
seek protection against bad days that create reputational harm,
cancellations, or downstream congestion. We capture these risk
postures through three model variants that share the same assignment,
capacity, and sequencing constraints
\eqref{eq:core-assign}--\eqref{eq:core-rel} and differ only in how
the day-level log-budgets $\{b_d\}$ enter the optimization.

\paragraph{Cost components.}
For compactness, let
\begin{equation}
  \label{eq:cop}
  C_{\mathrm{op}} :=
  c^{\mathrm{idle}}\sum_{d,r} I_{dr}
  + c^{\mathrm{rm}}\sum_{d,r} o^{\mathrm{rm}}_{dr}
  + c^{\mathrm{sg}}\sum_{d,k} o^{\mathrm{sg}}_{dk}
\end{equation}
denote the operational-cost component of the objective. All three
variants minimize $C_{\mathrm{op}}$ plus a risk term $\Psi$ whose
form depends on the selected posture.

\paragraph{{\bf M1: Average (mean reliability).}}
\label{sec:m1}
M1 rewards higher average day-level reliability, allowing the
optimizer to improve the days where reliability gains are cheapest.
This is appropriate when the hospital evaluates performance over the
full horizon (e.g., weekly KPIs) and tolerates uneven reliability by
day. Letting $\bar b := \frac{1}{|\mathcal{D}|}\sum_{d\in\mathcal{D}}
b_d$ denote the average day-level log-budget, the M1 objective is to
minimize $C_{\mathrm{op}} - w_{\varepsilon}\,\bar b$.
Since $b_d \le 0$, the second term encourages $b_d$ values closer to
$0$ (higher reliability); normalization by $|\mathcal{D}|$ keeps
$w_\varepsilon$ comparable across horizons of different length.

\paragraph{{\bf M2: Worst-case (protect the weakest day).}}
\label{sec:m2}
M2 protects against the worst day in the horizon, redistributing
workload away from the bottleneck day to produce more balanced
day-level reliability. It is appropriate when management is sensitive
to tail events --- severe delays on a single day that trigger
cancellations, staff dissatisfaction, or patient complaints. We
introduce $b_{\min} \le 0$ to represent the minimum day-level
log-budget and impose $b_{\min} \le b_d$ for all $d \in \mathcal{D}$.
The M2 objective is to minimize $C_{\mathrm{op}} - w_{\varepsilon}\,
b_{\min}$.

\paragraph{{\bf M3: Fixed (hard reliability target).}}
\label{sec:m3}
M3 enforces a minimum reliability level on every day and minimizes
operational cost subject to that floor. It is the natural choice when
the hospital adopts a policy target (or external requirement) such as
``keep the probability of any buffer violation below
$\varepsilon_{\mathrm{target}}$ each day.'' Let
$\varepsilon_{\mathrm{target}} \in (0,1)$ denote the upper bound on
the day-level violation proxy and define $b_{\mathrm{rhs}} :=
\log(1-\varepsilon_{\mathrm{target}})$. We impose $b_d \ge
b_{\mathrm{rhs}}$ for all $d \in \mathcal{D}$, and the M3 objective
reduces to minimizing $C_{\mathrm{op}}$.

\begin{table}[t]
  \centering
  \caption{Risk interpretation of the three postures.}
  \label{tab:risk-postures}
  \small
  \setlength{\tabcolsep}{6pt}
  \renewcommand{\arraystretch}{1.15}
  \begin{tabular}{@{} l ccc @{}}
    \toprule
    \rowcolor{hdrgray}& \textbf{M1} \textit{(average)} & \textbf{M2} \textit{(worst-day)}
    & \textbf{M3} \textit{(hard target)} \\
    \midrule
    Action on $\{b_d\}$
    & reward $\bar b=\tfrac{1}{|\mathcal D|}\sum_{d} b_d$
    & reward $b_{\min}$, $b_{\min}\!\le\! b_d\,\forall d$
    & enforce $b_d\!\ge\!\log(1-\varepsilon_{\mathrm{t}})\,\forall d$ \\
    \addlinespace[2pt]
    \cmidrule(l{0pt}r{0pt}){2-4}
    \addlinespace[2pt]
    Probabilistic translation
    & $\tfrac{1}{|\mathcal D|}\sum_{d}\Pr(F_d^c)\le 1-e^{\bar b}$
    & $\max_{d}\Pr(F_d^c)\le 1-e^{b_{\min}}$
    & $\Pr(F_d^c)\le \varepsilon_{\mathrm{t}}\,\forall d$ \\
    \addlinespace[2pt]
    \cmidrule(l{0pt}r{0pt}){2-4}
    \addlinespace[2pt]
    Object controlled
    & avg.\ log-reliability
    & worst-day reliability
    & per-day reliability floor \\
    \addlinespace[2pt]
    \cmidrule(l{0pt}r{0pt}){2-4}
    \addlinespace[2pt]
    Managerial reading
    & average-horizon reliability
    & worst-day protection
    & hard daily target \\
    \bottomrule
  \end{tabular}
  \par\smallskip
  {\footnotesize\raggedright Under the independence interpretation of
    Section~\ref{sec:probabilistic-interp};
    $\varepsilon_{\mathrm{t}}\!=\!\varepsilon_{\mathrm{target}}$.
    Schedule-infeasibility events $F_d^c$ are linked to
    buffer-violation events via Proposition~\ref{prop:sample-path}. For
    M1, the bound is the loose endpoint of $\tfrac{1}{|\mathcal
    D|}\sum_d \Pr(F_d^c)\le 1-\tfrac{1}{|\mathcal D|}\sum_d e^{b_d}\le
  1-e^{\bar b}$ (Jensen).\par}
\end{table}

\ECappref{ec:menu-value-condensed} of the Electronic Companion
makes the unification precise: in a fixed-schedule relaxation, M1, M2,
and M3 produce the same case-level reliability allocation on any
day --- distributing protection according to each surgery's variance --- and
differ only in the total reliability budget assigned to each day
(Table~\ref{tab:risk-postures}).

\section{Computational Experiments}
\label{sec:computational-experiments}

This section evaluates how the three risk postures (M1--M3) trade
off operating cost and execution risk across a set of synthetic
instances designed to stress different capacity bottlenecks. Two
supplementary computational analyses are reported in the Electronic
Companion: \MV{\ECappref{ec:mechanism}} presents mechanism
diagnostics for the
$\alpha$-menu --- the depolarization of $\alpha$ selections under increasing
reliability penalty, and the correlation between assigned $\alpha$
and surgery duration --- and \MV{\ECappref{ec:buffer-engines}}
reports a controlled
synthetic plug-in check showing how Cantelli, $W_\infty$, and $W_1$
buffers enter the same engine-agnostic scheduling model.
\subsection{Experimental Setup}
\label{sec:exp-setup}
All experiments were conducted on a workstation with an Intel Core
i7-11700K processor (3.60GHz, 8 cores, 16 threads) and 32 GB RAM. The
MIP models were implemented in Pyomo and solved with Gurobi 11.0.0.
Unless stated otherwise, we impose a time limit of 1{,}800 seconds
and a target MIP gap of 1\%.

The experiments use four synthetic instances (A--D in
Table~\ref{tab:instance-characteristics}) that share the common generator
described in this subsection and differ in size, resource balance ($|J|$,
$|R|$, $|K|$), and surgical case mix. We consider a one-week planning
horizon (Monday--Friday) with daily
operating-room availability decreasing through the week: 600 minutes
on Monday and Tuesday, 540 on Wednesday, 480 on Thursday, and 420 on
Friday. Surgeon daily capacities are heterogeneous, cycling through
$\{480, 540, 600, 480, 520, 560\}$ minutes; the last ${\sim}40\%$ of
surgeons are unavailable on Friday, creating a compound bottleneck on
that day. Each instance draws its caseload from three size classes ---
small (expected duration 60--90\,min, coefficient of variation 0.25), medium
(120--150\,min, 0.20), and large (180--240\,min, 0.18) --- with the proportion
of each class varying across instances.

To capture environment-dependent performance, expected durations are
modulated by room efficiency factors $\phi_r \in [0.95, 1.05]$ and
surgeon efficiency factors $\psi_k \in [0.90, 1.10]$, applied
multiplicatively: $\mu_{jrk} = \mu_j \cdot \phi_r \cdot \psi_k$,
$\sigma_{jrk} = \sigma_j \cdot \phi_r \cdot \psi_k$. A combination
$(j,d,r,k,t)$ is included in the feasible set $\mathcal{V}$ if the
doctor is available on day $d$ and the buffered duration does not
exceed the room horizon or doctor capacity (including overtime).
Buffered durations are computed via Cantelli's inequality:
$p_{jrk}^{(t)} = \mu_{jrk} + \sqrt{(1 -
\alpha_t)/\alpha_t}\;\sigma_{jrk}$. The reliability menu offers eight
risk levels, $\alpha \in \{0.005, 0.01, 0.02, 0.03, 0.04, 0.05,
0.075, 0.10\}$. Maximum allowable overtime is 120 minutes per room
and 60 minutes per surgeon per day.

Table~\ref{tab:instance-characteristics} summarises the four
instances. The buffered load columns show total scheduled demand
divided by total capacity at three reliability levels: raw
mean~$\bar\mu$, moderate buffer ($\alpha = 0.05$, Cantelli factor
$\sqrt{(1{-}\alpha)/\alpha} = 4.36$), and the most conservative menu
entry ($\alpha = 0.005$, factor~14.1). A load above 100\% means
overtime or less conservative $\alpha$ choices are unavoidable.
Instance~A has the highest room load (3~rooms for 25~surgeries);
Instance~D's doctors are already over capacity at $\alpha = 0.05$
(105\%), making doctor overtime dominant. Instances B and~C have
similar load balance, but C is a harder optimisation problem
(50~surgeries, ${\sim}$105k binary variables). Each instance is
solved with  two
random seeds. The seed re-draws only the per-surgery expected
durations within their class ranges (standard deviations then follow
from the fixed class coefficients of variation), while all structural
parameters --- capacities, availability, and the room and surgeon
efficiency grids --- are identical across seeds. All values in
Table~\ref{tab:instance-characteristics} are averaged over both seeds.

\begin{table}[ht]
  \centering
  \caption{Instance characteristics, averaged over two random seeds.}
  \label{tab:instance-characteristics}
  \small
  \begin{tabular}{l l rrrr  rrr  rrr}
    \toprule
    \rowcolor{hdrgray}& & & & & & \multicolumn{3}{c}{Room Load} &
    \multicolumn{3}{c}{Doctor Load} \\
    \cmidrule(lr){7-9} \cmidrule(lr){10-12}
    \rowcolor{hdrgray}Inst & Profile & $|J|$ & $|R|$ & $|K|$ & $|D|$
    & $\bar\mu$ & $.05$ & $.005$
    & $\bar\mu$ & $.05$ & $.005$ \\
    \midrule
    A & Tight (room)      & 25 & 3 & 4 & 5 & 51\% &  93\% & 189\% &
    40\% &  74\% & 149\% \\
    B & Balanced           & 40 & 4 & 4 & 5 & 44\% &  85\% & 176\% &
    47\% &  90\% & 186\% \\
    C & Balanced (large)   & 50 & 5 & 5 & 5 & 44\% &  85\% & 176\% &
    48\% &  93\% & 192\% \\
    D & Doctor-constrained & 35 & 5 & 3 & 5 & 31\% &  60\% & 123\% &
    55\% & 105\% & 217\% \\
    \bottomrule
  \end{tabular}
\end{table}

In our evaluation, we report (i) operating cost $C_{\mathrm{op}}$ and
(ii) day-level
reliability proxies $\varepsilon_d:=1-e^{b_d}$ and
$\varepsilon_{\max}:=\max_{d}\varepsilon_d$. Evaluation under
realized durations using sample-path replay against held-out
historical data is the focus of the case study in Section~\ref{sec:case-study}.

\subsection{Risk Postures: Comparing M1--M3}
\label{sec:exp-risk-postures}

For concreteness, we focus
first on Instance~C, the balanced-large synthetic
instance with 50 surgeries, 5 rooms, and 5 surgeons. The same comparison
is reported across the remaining synthetic instances below, with
supporting mechanism diagnostics provided in the Electronic
Companion (\ECappref{ec:mechanism}).

For M1 and M2
we sweep the reliability weight $w_\varepsilon \in \{1, 1{,}000,
100{,}000\}$. For M3 we use bisection over
$\varepsilon_{\mathrm{target}} \in [0.10, 0.50]$ with tolerance
$0.05$ to find the smallest feasible target; unlike M1 and M2, M3 has
no reliability-penalty
parameter. Table~\ref{tab:m1m2m3-summary} summarizes operational cost,
overtime, and reliability metrics across all seven runs ($\bar\alpha$
denotes the per-instance average of the chosen reliability levels).

\begin{table}[ht]
  \centering
  \caption[M1, M2, and M3 across $w_\varepsilon$ values on
  Instance~C]{M1, M2, and M3 across $w_\varepsilon$ values on
    Instance~C (a single representative seed; the multi-instance
    averages over two seeds are in Table~\ref{tab:multi-instance-outcomes}).
    $C_{\mathrm{op}} = \text{idle} + 3\,\text{room OT} +
    1.5\,\text{doctor OT}$; $\varepsilon_d = 1 -
  \prod_{j\in\text{day}\,d}(1-\alpha_j)$. Times in minutes.}
  \label{tab:m1m2m3-summary}
  \small
  \begin{tabular}{ll rrr r rrr r}
    \toprule
    \rowcolor{hdrgray}& & \multicolumn{3}{c}{Overtime \& idle} & &
    \multicolumn{3}{c}{Reliability} & \\
    \cmidrule(lr){3-5} \cmidrule(lr){7-9}
    \rowcolor{hdrgray}Model & $w_\varepsilon$ & Room OT & Doctor OT & Idle &
    $C_{\mathrm{op}}$ & $\varepsilon_{\mathrm{worst}}$ &
    $\varepsilon_{\mathrm{best}}$ & $\varepsilon$-spread & $\bar\alpha$ \\
    \midrule
    M1 & 1           & 1    & 17   & 1{,}739  & 1{,}767  & 0.574 &
    0.300 & 0.273 & 0.058 \\
    M1 & 1{,}000     & 1    & 29   & 1{,}724  & 1{,}771  & 0.501 &
    0.261 & 0.239 & 0.049 \\
    M1 & 100{,}000   & 1{,}026 & 1{,}040 & 1{,}162 & 5{,}801 & 0.328
    & 0.123 & 0.205 & 0.028 \\
    \midrule
    M2 & 1           & 0    & 16   & 1{,}751  & 1{,}777  & 0.589 &
    0.261 & 0.329 & 0.061 \\
    M2 & 1{,}000     & 1    & 38   & 1{,}709  & 1{,}768  & 0.463 &
    0.348 & 0.115 & 0.055 \\
    M2 & 100{,}000   & 876  & 888  & 1{,}225  & 5{,}185  & 0.263 &
    0.256 & 0.006 & 0.030 \\
    \midrule
    M3 & bisection   & 152  & 269  & 1{,}514  & 2{,}373  & 0.362 &
    0.359 & 0.003 & 0.044 \\
    \bottomrule
  \end{tabular}
\end{table}

\paragraph{Reliability and schedule structure across days.}
At $w_\varepsilon = 100{,}000$, M2 equalizes day-level reliability
within a 0.6 percentage-point band, while M1 still leaves a
20+ percentage-point spread. Figure~\ref{fig:daily-epsilon-bars}
shows the per-day failure
probability $\varepsilon_d$ for M1, M2, and M3. At low $w_\varepsilon$,
both M1 and M2 exhibit high and variable failure probabilities
($\varepsilon_d \in [0.26, 0.59]$). As $w_\varepsilon$ increases, M1
improves reliability on the cheapest days---Friday, which has the
shortest horizon and fewest surgeries, drops to $\varepsilon =
0.12$---while the most congested day (Wednesday) remains at
$\varepsilon = 0.33$. M2 instead equalizes: at $w_\varepsilon = 100{,}000$,
all days fall in the narrow band $[0.256, 0.263]$. M3 achieves
near-perfect uniformity at $\varepsilon \approx 0.36$ without
requiring $w_\varepsilon$ calibration.

The behavior of M2 has a useful managerial interpretation. Let
$C_{\mathrm{op}}(x)$ be the operating cost in~\eqref{eq:cop} and
$m(x)=\min_{d\in\mathcal D} b_d(x)$ the worst-day log-budget of a
schedule $x$. Under M2, moving to a feasible schedule $\tilde x$ is
attractive whenever $C_{\mathrm{op}}(\tilde x)-C_{\mathrm{op}}(x)
< w_\varepsilon\bigl(m(\tilde x)-m(x)\bigr)$, so $w_\varepsilon$ is an
exchange rate between operating cost and worst-day reliability: M2
spends capacity, overtime, or reassignment to raise the least reliable
day whenever the gain justifies the cost. It therefore lifts the
current bottleneck day until another becomes binding, leveling
day-level reliability. This leveling is only a tendency: capacity,
eligibility, sequencing, or menu restrictions can prevent exact
equalization.

The same pattern holds across all simulations.
\ECappref{ec:m2-equalization} of the Electronic Companion gives
theoretical support and a counterexample showing when equalization can
fail (\ECappref{ec:m2-counterexample-compact}), and
\ECappref{ec:mechanism} examines how M1 and M2 allocate
reliability across surgeries.

\begin{figure}[ht]
  \centering
  \includegraphics[width=0.68\textwidth]{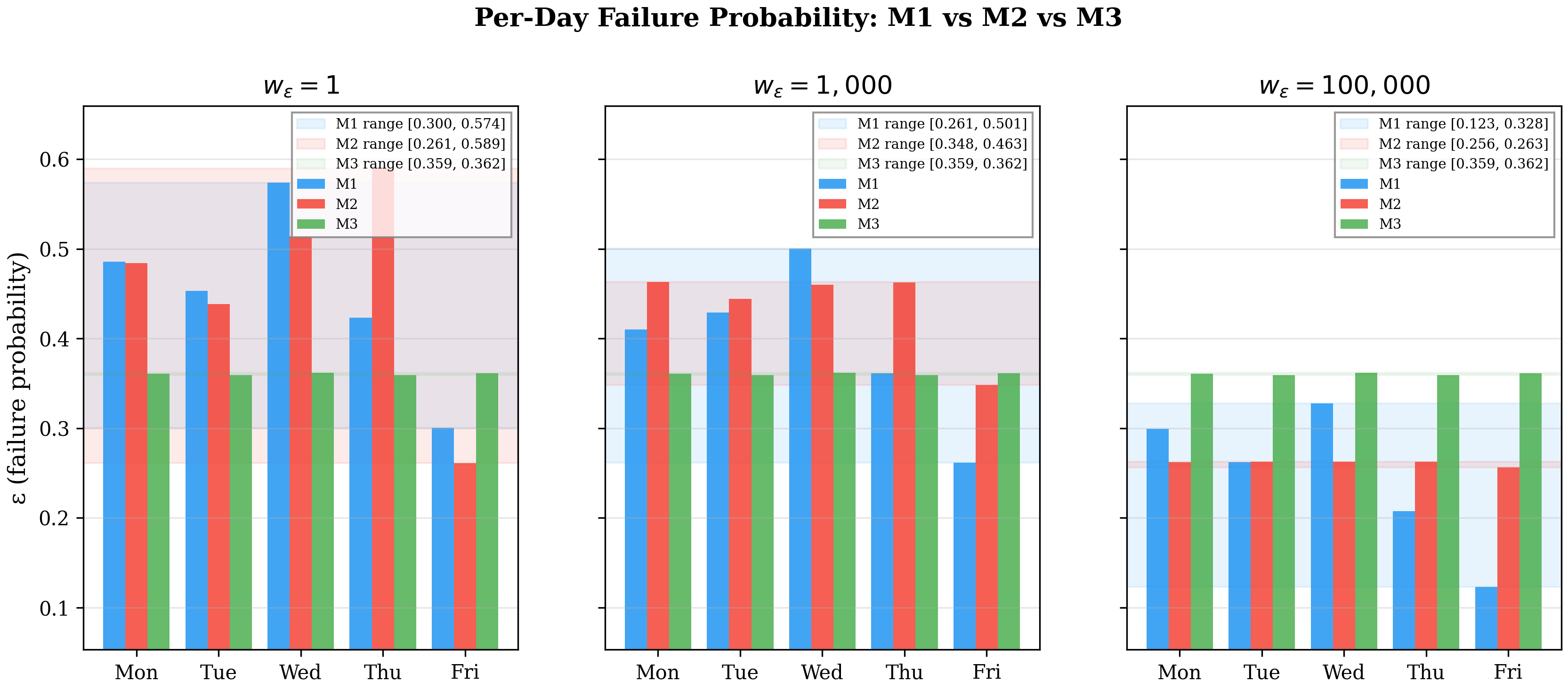}
  \caption{Per-day failure probability $\varepsilon_d$ for M1 (blue),
  M2 (red), and M3 (green) at the three $w_\varepsilon$ values.}
  \label{fig:daily-epsilon-bars}
\end{figure}

The way the optimizer uses the menu is itself informative. As the
reliability penalty grows, extreme $\alpha$ values disappear and
selections concentrate in a moderate band; within that band, longer
surgeries receive systematically higher $\alpha$ (Pearson $r=0.58$;
  diagnostics in \ECappref{ec:mechanism},
\ECfigref{ec-fig:alpha-stacked}--\ECfigref{ec-fig:alpha-vs-duration}).
The key question is therefore whether surgery-specific differentiation
within that band matters.

A fixed-schedule relaxation, derived in
\ECappref{ec:menu-value-condensed}, answers this exactly. Fix the
assignment and sequence, and suppose a given pool of buffer minutes
must be spread across a day's cases. Relative to the optimal
case-by-case allocation of those minutes, the best single common
reliability target inflates the day's failure-bound
surrogate by a factor of
exactly
\begin{equation}
  \label{eq:gamma-index}
  \Gamma_{\rm fix}\;=\; \frac{n\,\bigl(\sum_j \sigma_j\bigr)^2}
  {\bigl(\sum_j \sigma_j^{2/3}\bigr)^3},
\end{equation}
where $n$ is the number of surgeries. The index $\Gamma_{\rm fix}$ depends only on the
day's caseload: it equals one when all surgeries have the same $\sigma_j$ and
grows as variability concentrates in a few cases. The $2/3$ exponent
reflects a simple tradeoff. Reaching reliability $\alpha$ on case $j$
costs $\sigma_j\,\kappa(\alpha)$ buffer minutes, with
$\kappa(\alpha)=\sqrt{(1-\alpha)/\alpha}$, but adds $-\log(1-\alpha)$
to the day's budget irrespective of $\sigma_j$; reliability is thus
cheapest on low-variance cases, and equalizing the reliability bought
per buffer-minute across cases makes the optimal buffers grow as
$\sigma_j^{2/3}$ --- sublinearly, rather than in proportion to
$\sigma_j$. The menu can therefore protect low-variance cases more
tightly while still giving high-variance cases larger absolute
buffers, which no single $\alpha$ can do. Because $\Gamma_{\rm fix}$ is
scale-invariant and at least one, the prediction is sharp: the menu
offers no advantage on a homogeneous day and gains value as
heterogeneity grows.

\paragraph{What the solved instances show.}
To test whether the relaxation's prediction survives in the full
problem, we let the scheduler jointly optimize assignment, sequencing,
and reliability on a 50-surgery, 5-room, 5-surgeon week (the
Instance-C skeleton) and compare three policy families under an
identical $1{,}800$\,s budget per design point: the case-level
$\alpha$-menu and a common target applied to every surgery
(uniform-$\alpha$), both under the worst-day posture (M2) and traced
over $w_\varepsilon \in \{1, 3, 10, 30, 100\}\times 10^3$ and over
fixed reliability levels respectively; and uniform proportional
buffers $p_{jrk}=(1+\beta)\mu_{jrk}$
($\beta \in \{10, 20, 30, 50, 65, 80, 100, 125\}\%$) that minimize
operating cost ($w_\varepsilon = 0$). Every resulting schedule is
scored by its Cantelli worst-day failure bound, so all families are
compared on one reliability metric: the vertical axis of
Figure~\ref{fig:menu-dose-response} reports this certified worst-day
failure probability and the horizontal axis the aggregate planned
room-plus-surgeon overtime reservation. To vary heterogeneity we build
four dispersion profiles that differ only in how variance is spread
across surgeries, sharing the same skeleton and the same total
dispersion $\sum_j \sigma_j$, giving $\Gamma_{\rm fix} = 1.02$, $1.30$, and
$1.51$ across the realistic range and $1.70$ for a profile that
extends the trend.

Figure~\ref{fig:menu-dose-response} shows the resulting Pareto frontiers
under the worst-day posture (M2). The relaxation's conclusion
holds in the fully solved model: as $\Gamma_{\rm fix}$ increases, the case-level menu
progressively outperforms the scalar policies. The lowest certified worst-day
failure bound achieved by the menu decreases monotonically with $\Gamma_{\rm fix}$,
from $0.278$ to $0.238$, $0.197$, and $0.180$, whereas the best
uniform-$\alpha$ and proportional-buffer policies never improve beyond
$0.335$ and $0.414$ on any profile. Greater heterogeneity also
lowers the overtime needed for strong reliability: at
$\Gamma_{\rm fix}=1.70$ the menu certifies a $0.180$ bound while
reserving $1{,}248$ planned room-plus-surgeon overtime minutes per week,
against $0.278$ at $1{,}450$ minutes when $\Gamma_{\rm fix}=1.02$ --- stronger
protection at lower reserved overtime as heterogeneity grows.

\begin{figure}[!tb]
  \centering
  \CHGon
  \includegraphics[width=\textwidth]{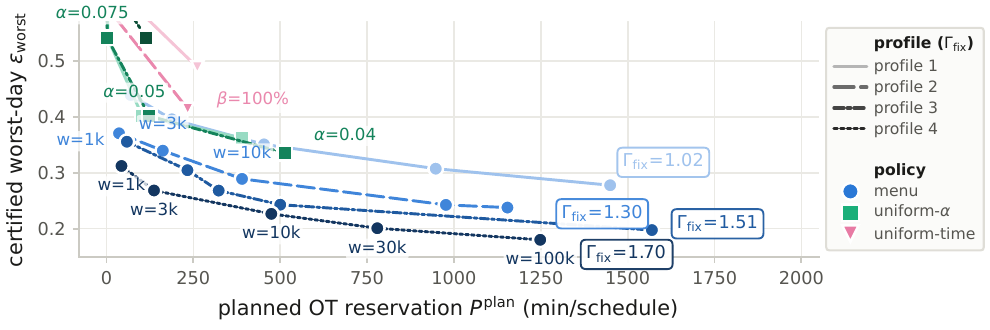}
  \caption{Certification--reservation frontiers (Instance~C, seed 42),
    zoomed to the decision-relevant band. Each point is the best
    schedule found within a 1{,}800\,s solver budget. Shade and line
    style indicate dispersion
    profiles 1--4 ($\Gamma_{\rm fix} = 1.02$, \GRN{$1.30$}, $1.51$, and
    $1.70$); profile 2 additionally matches profile 3's total
    variance $\sum_j \sigma_j^2$, so that any difference between them
    isolates $\Gamma_{\rm fix}$ from total variance. Off-frame variants reaching
  $\varepsilon$ up to $1$ are omitted.}
  \label{fig:menu-dose-response}
\end{figure}

The menu has a second advantage: it is self-calibrating. A
common-$\alpha$ policy must choose a single reliability level for all
surgeries, but the competitive value of that level depends on the
instance --- the surgery mix, room horizons, surgeon capacities, and daily
loads all affect which common $\alpha$ performs best. In practice, finding
it requires either solving the menu model first or sweeping over a grid of
fixed targets. The menu bypasses this problem: it finds the right
allocation endogenously. Across our four test instances, 99\% of selected
$\alpha$ values fall in the band $[0.02, 0.05]$, with a weighted mean that
shifts from $\bar\alpha \approx 0.026$ in balanced instances to $0.040$ in
doctor-constrained ones --- a structure-dependent adaptation that no fixed
rule can replicate without instance-specific tuning.

\paragraph{Robustness across instances.}
Table~\ref{tab:multi-instance-outcomes} reports M1, M2, and M3
results on all four instances averaged over 2 seeds. The patterns
observed for Instance~C
hold across the test set: M2 collapses the $\varepsilon$-spread to
$\le 0.007$ on every instance, while M1 leaves spreads of
0.088--0.150; M3 matches M2's uniformity at substantially lower
operating cost where the hard reliability target is achievable.
\begin{table}[ht]
  \centering
  \caption{M1/M2/M3 across all four instances at $w_\varepsilon =
    100{,}000$ (M3 by bisection), averaged over two seeds.
    $\varepsilon_{\mathrm{worst}}$: worst-day failure probability;
    $\varepsilon_\mu$: mean day-level failure probability;
  $\varepsilon$-spread: max$-$min day spread.}
  \label{tab:multi-instance-outcomes}
  \small
  \begin{tabular}{l l  rrr  rrr  r}
    \toprule
    \rowcolor{hdrgray}Inst & Model
    & $\varepsilon_{\mathrm{worst}}$ & $\varepsilon_\mu$ & $\varepsilon$-spread
    & Room OT & Doc OT & Idle
    & $C_{\mathrm{op}}$ \\
    \midrule
    \multirow{3}{*}{A}
    & M1 & 0.163 & 0.125 & 0.090 &  691 &  543 &  141 & 3{,}027 \\
    & M2 & 0.137 & 0.137 & 0.001 &  459 &  515 &  135 & 2{,}283 \\
    & M3 & 0.150 & 0.148 & 0.004 &  319 &  379 &  166 & 1{,}692 \\
    \midrule
    \multirow{3}{*}{B}
    & M1 & 0.216 & 0.178 & 0.088 &  904 &  826 &  697 & 4{,}647 \\
    & M2 & 0.192 & 0.192 & 0.001 &  704 &  614 &  723 & 3{,}755 \\
    & M3 & 0.225 & 0.224 & 0.004 &  293 &  375 &  777 & 2{,}218 \\
    \midrule
    \multirow{3}{*}{C}
    & M1 & 0.299 & 0.238 & 0.150 & 1{,}014 & 1{,}031 & 1{,}173 & 5{,}760 \\
    & M2 & 0.259 & 0.258 & 0.007 &  809 &  873 & 1{,}218 & 4{,}954 \\
    & M3 & 0.346 & 0.343 & 0.005 &  188 &  251 & 1{,}528 & 2{,}468 \\
    \midrule
    \multirow{3}{*}{D}
    & M1 & 0.279 & 0.233 & 0.129 &  787 &  661 & 5{,}834 & 9{,}186 \\
    & M2 & 0.249 & 0.248 & 0.002 &  680 &  572 & 5{,}833 & 8{,}731 \\
    & M3 & 0.283 & 0.282 & 0.002 &  408 &  336 & 5{,}877 & 7{,}606 \\
    \bottomrule
  \end{tabular}
\end{table}

\section{Case Study: Elective Surgery Scheduling at HLA Moncloa Hospital}
\label{sec:case-study}

This section evaluates the proposed approach using historical data
from HLA Moncloa Hospital, a private medium-sized multi-specialty
hospital in Madrid, Spain. The purpose is back-testing: we quantify
how schedules produced using only pre-schedule information would have
performed under the realized surgery durations observed in the
hospital data.

\subsection{Setting and Design}
\label{sec:case-study-design}

Elective cases are scheduled several days in advance over a rolling
planning horizon of two to five days. Operating rooms operate under a
regular daily horizon of 720 minutes; overtime is permitted but
operationally undesirable. The hospital's baseline planning practice
constructs schedules from expected (mean) procedure durations derived
from historical records, with buffers used informally and not
governed by an explicit reliability target.

We conduct a rolling-origin back-test using elective-surgery data
from 2016--2018. We set the first evaluation date to the beginning of
week 40 of 2018 (chosen to be as late as possible in the dataset
while avoiding the holiday weeks at year-end) and generate ten
consecutive, non-overlapping test instances. For each instance
$\ell=1,\dots,10$, we randomly draw a planning-horizon length $L_\ell
\in \{2,3,4,5\}$ days and a number of operating rooms $R_\ell \in
\{3,\dots,10\}$, define a time window of $L_\ell$ consecutive working
days starting at the current evaluation date, and extract all
elective surgeries performed during that window together with the
surgeons who actually performed them. Because we do not observe a
complete surgeon--procedure qualification matrix, we keep the surgeon
assignment fixed at its observed value (each surgery can only be
assigned to its historical surgeon). Nominal duration parameters used
for planning are computed in an out-of-sample fashion: for each
surgery--surgeon combination, we estimate $(\mu,\sigma)$ using all
historical observations available up to the day immediately preceding
the start of the window. We then solve the scheduling model for that
instance and evaluate the resulting schedule on the realized
durations from the window. The next instance starts on the first
working day after the previous window ends. The protocol generated
instances with 45 to 227 surgeries and between 5 and 19 surgeons; we
set regular room capacity $H = 720$ minutes per day and maximum
overtime $\bar{O}^{\text{room}} = \bar{O}^{\text{doc}} = 240$ minutes.

This design avoids information leakage and reflects how a planner
would update estimates over time. Three modeling choices warrant
comment. First, instances are consecutive in time, hence not i.i.d.;
we mitigate this through non-overlapping windows and within-instance
comparisons (robust vs.\ deterministic) under the same realized
workload. Second, randomizing $L_\ell$ and $R_\ell$ improves external
validity but reduces strict cross-instance comparability, so we
emphasize distributional effects (tail-delay reduction) stable across
that heterogeneity. Third, fixed surgeon assignments isolate the
value of uncertainty-aware buffering and sequencing from that of
cross-training or surgeon reassignment.

For each instance we construct six scheduling policies, three implementing the
planner heuristics, and three capturing distinct reliability policies.
The planner heuristics reflect current practice: a
\emph{deterministic} schedule built on mean
durations ($\sigma=0$)
and \emph{fixed proportional buffers} that inflate every mean
duration by a constant
$30\%$ or $50\%$. The \emph{reliability policies} impose an explicit day-level
reliability control: a single \emph{common~$\alpha$} applied to every
case (we report
  the loose but always-feasible $\alpha=0.10$ and, separately, the
  best-feasible common
level per instance), and the case-level \emph{$\alpha$-menu}. All policies are
reoptimized on identical data under the worst-day posture~(M2 model with
  Cantelli-based buffers with $\alpha \in
{0.005,\ 0.01,\ 0.02,\ 0.03,\ 0.04,\ 0.05,\ 0.075,\ 0.10}$).
Each run uses a
1800-second solver time limit. Each schedule is evaluated by
discrete-event simulation under the realized durations, propagating
delays through the room sequence whenever a surgery exceeds its
allocated time.

Following the payment-based logic of
Section~\ref{sec:day-level-posture}, a hospital's relevant risk object
depends on how it is paid. HLA Moncloa operates under increasingly
common value-based and bundled-payment arrangements, where a single
disrupted day --- not the average week --- drives episode-level quality
penalties and payer relationships. We therefore adopt the worst-day
posture (M2), which our framework designates as natural for such
settings and which directly optimizes our core metric, the
worst-day failure probability $\varepsilon_{\max}$.

\subsection{Aggregate Results: Comparison Across Policies}
\label{sec:case-study-results}

Table~\ref{tab:policy-comparison} reports performance
metrics averaged across the ten instances. We compare the six
scheduling policies along seven performance metrics, all
averaged across instances and realizations. The worst-day failure probability
$\varepsilon_{\max}$ is the fraction of instance-days on which at
least one surgery exceeds its planned (buffered) duration, so that
lower values indicate more
reliable day-level performance. The number of surgeries delayed
counts, on average, how many surgeries per instance begin after their
planned start time,
capturing how widely disruption spreads across the day rather than
its magnitude. The maximum delay and the 95th-percentile delay
report, respectively, the
mean of the largest and of the 95th-percentile start-time slippage in
minutes, summarizing the extreme and upper-tail behavior of the delay
distribution. The
count of delays exceeding 90 minutes records the average number of
severely delayed surgeries per instance, a threshold of particular
operational concern. The
last-surgery delay gives the mean delay of the final surgery of the
day in minutes, measuring end-of-day slippage and the associated risk
of after-hours
running. Finally, overtime reports realized room overtime in minutes.

\begin{table}[t]
  \centering
  \caption{Policy comparison averaged across the ten case-study
  instances. All values are means across instances and realizations.}
  \label{tab:policy-comparison}
  \small
  \setlength{\tabcolsep}{4.5pt}
  \renewcommand{\arraystretch}{1.05}
  \begin{tabular}{@{}l ccccc c c@{}}
    \toprule
    \rowcolor{hdrgray}& $\varepsilon_{\max}$ & Surg.\ & Max & 95th pct.\ & Delays & Last-surg.\
    & Overtime \\
    \rowcolor{hdrgray}Policy & (worst day) & delayed & delay & delay & $>$90\,min & delay
    & (min)\\
    \midrule
    \multicolumn{8}{@{}l}{\textit{Reliability policies}}\\
    \quad $\alpha$-menu (robust)      & 0.20 & 0.1
    & 22 & 7 & 0.1 & 8 &  347 \\
    \quad Common $\alpha$ (best feas.)& 0.26 & 0.2 & 25 &  16 & 0.2 &
    19 &  391 \\
    \quad Common $\alpha=0.10$        & 0.29 & 0.2 & 25 &  18 & 0.2 &
    20 & 392 \\
    \midrule
    \multicolumn{8}{@{}l}{\textit{Planner heuristics}}\\
    \quad Deterministic ($\sigma=0$)  & 1.00 & 7.9 & 631 & 626 & 3.0
    & 251 & 597 \\
    \quad Fixed buffer $+30\%$         & 0.86 & 1.8 & 194 & 101 & 0.4
    &  65 &  598 \\
    \quad Fixed buffer $+50\%$         & 0.73 & 1.3 & 353 & 182 & 0.6
    & 119 & 603 \\
    \bottomrule
  \end{tabular}
\end{table}

The comparison separates two questions. Against the \emph{planner
heuristics}, the
robust $\alpha$-menu is dominant on every dimension: it cuts the
worst-day failure
probability from~$1.00$ (a bad day on every instance) to~$0.20$,
reduces the number of
delayed surgeries from~$7.9$ to~$0.1$, and compresses the delay tail
by more than an order of magnitude (95th-percentile delay $7$
  vs.\ $626$ minutes for the deterministic
baseline; $7$ vs.\ $182$ even against $+50\%$ uniform padding). Against the
\emph{scalar reliability policies}, the picture is more refined: the
menu achieves the
lowest worst-day failure probability ($0.20$ vs.\ $0.26$--$0.29$),
reflecting its ability
to concentrate protection on the cases that most endanger the day.
In addition, the robust schedule obtains a slightly smaller \emph{maximum}
single-surgery delay. Finally, the menu's overtime ($347$
minutes) is better than both the scalar policies and the planner
heuristics, confirming that its reliability gains do not come at a
systematic overtime
cost once reserved slack is reclaimed.

The $+50\%$ buffer does not uniformly improve on $+30\%$ (average
maximum delay $353$ vs.\ $194$ minutes). The reason is a packing
effect: inflating every buffer makes each surgery a larger, less
divisible block, and under binding room capacity larger blocks pack
more coarsely---some rooms end up with long sequential chains while
others are left with unused slack. A single early overrun in a long
chain then cascades through more downstream cases. The lesson is that
the \emph{amount} of uniform slack is not a reliable lever;
reliability depends on matching buffers to each surgery's variance,
as the $\alpha$-menu does.

\subsection{Mechanism Analysis: The Cascade-Control Test}
\label{sec:case-study-mechanism}

The aggregate tail reductions above show that robust scheduling
outperforms the deterministic baseline, but they do not reveal why.
Two mechanisms could produce these gains: per-surgery shock
absorption (each Cantelli buffer absorbs duration variability
before it propagates) and cascade control (residual shocks that
survive the buffer compound less across the day). The first is
visible directly in residual-overrun statistics; to test the
second, we run a paired regression on the realized schedules.

For each surgery $j$ at position $k$ in room $r$ on day $d$ in
instance $\ell$ under schedule $s\in\{\text{rob},\text{det}\}$, let
$D$ denote the realized delay (start-time slippage), $R_s = 1$ if
$s = \text{rob}$ and $0$ otherwise, and define the cumulative
upstream pressure as
\[
  P_{rdks\ell} \;=\; \sum_{k' < k}\,
  \max\!\bigl(0,\; X_{k'} - \mu_{k'}\bigr),
\]
where $X_{k'}$ and $\mu_{k'}$ denote the realized duration and the
population mean of the surgery at position $k'$ in room $r$ on day
$d$ under schedule $s$ in instance $\ell$. The quantity $P$ depends
only on realized data and population means, so both schedules face
the same physical disturbance at each position --- a necessary
property for a between-schedule comparison of how shocks propagate.
(Measuring shocks against the schedule's own buffered duration would
  conflate the buffer's absorption effect with the propagation we want
to isolate; we return to this in the Robustness discussion below.)
We then estimate
\begin{align*}
  D_{jrdks\ell} = {} & \beta_0 + \beta_1\,P + \beta_2\,k + \beta_3\,R_s \\
  & + \beta_4\,(R_s \times P) + \beta_5\,(k \times P) \\
  & + \beta_6\,(R_s \times k \times P) + \alpha_\ell +
  \varepsilon_{jrdks\ell},
\end{align*}
with instance fixed effects $\alpha_\ell$ and standard errors
clustered at the instance level $\ell$. The cascade-control
hypothesis predicts
$\beta_5 > 0$ (cascade build-up under deterministic) and $\beta_6 < 0$
(build-up suppressed under robust).

Table~\ref{tab:cascade-regression} reports the regression and
Figure~\ref{fig:cascade-propagation} plots the implied propagation
rate as a function of position. The fit explains 92\% of the
variation in realized delay and the joint test that $\beta_4 =
\beta_6 = 0$ is rejected at $p < 10^{-3}$ ($F = 153.3$). The
regression discriminates two distinct mechanisms by which the
robust schedule outperforms the deterministic baseline.

The first mechanism is surgery-level absorption. Before propagation
enters the picture, the Cantelli buffers absorb most of each
surgery's individual shock. The mean residual overrun (realized
minus planned) drops from 52 minutes per surgery under the
deterministic schedule to 17 minutes under the robust schedule ---
a 67\% reduction in the shock that would otherwise enter the
cascade. This is the operationally simplest mechanism: each surgery
has a buffered duration calibrated to a target reliability
$\alpha$, and under typical realized durations the surgery
completes within it.

The second mechanism is cascade-flattening across the day. The
residual shocks that survive the buffer enter the daily cascade,
where the regression shows a sharp structural difference. Under the
deterministic schedule, propagation grows from $0.21$ minutes of
delay per upstream-shock-minute at position 1 to $0.73$ at position
10 --- a cascade build-up, with $\beta_5 = +0.057$ capturing the
increase per position. Under the robust
schedule, the position-dependent growth is suppressed: $\beta_6 =
-0.043$ cancels most of the build-up, and
propagation rises only mildly from $0.38$ at position 1 to $0.51$
at position 10. The two propagation lines cross around position 5
(Figure~\ref{fig:cascade-propagation}): early in the day, the
deterministic schedule's inter-surgery gaps absorb residual shocks
somewhat better than the robust schedule does (which has packed its
slack into per-surgery buffers); by mid-day, the deterministic
schedule has exhausted those gaps, and from there on the robust
schedule is structurally more protected against further
compounding.

Together, these two mechanisms --- absorption at the surgery level
and flattening of the cascade across the day --- account for the
tail reductions reported in
Table~\ref{tab:policy-comparison}. The 67\%
reduction in residual overrun directly compresses each individual
delay; the absence of compounding with position under the robust
schedule prevents surgeries late in the day from inheriting the
worst of the day's accumulated shocks. The flattening result is
robust to alternative clustering of standard errors, controls for
surgery characteristics, instance exclusion, alternative
definitions of the outcome and the shock measure, and a quadratic
in position; details in \ECappref{ec:cascade-robustness} of
the Electronic Companion.

\begin{table}[tbp]
  \centering
  \caption{Cascade-control regression. Dependent variable:
    realized start delay (minutes). Standard errors clustered at
    the instance level; instance fixed effects suppressed.
    $n = 1219$ surgery-schedule observations across $10$ instances.
  $R^2 = 0.92$.}
  \label{tab:cascade-regression}
  \begin{tabular}{@{}lrr@{}}
    \toprule
    \rowcolor{hdrgray}Term & Coefficient & SE \\
    \midrule
    $P$ (cum.\ upstream pressure) & $+0.212^{***}$ & $0.058$ \\
    $R$ (robust indicator) & $-12.44^{***}$ & $0.43$ \\
    $R \times P$ ($\beta_4$) & $+0.171^{**}$ & $0.058$ \\
    $k$ (position-in-day) & $+8.99^{**}$ & $2.88$ \\
    $k \times P$ ($\beta_5$) & $+0.057^{***}$ & $0.002$ \\
    $R \times k \times P$ ($\beta_6$) &
    $-0.043^{***}$ & $0.003$ \\
    \midrule
    \multicolumn{3}{@{}l}{Joint $F$ test ($\beta_4 = \beta_6 = 0$):
    $F = 153.3$, $p < 10^{-3}$} \\
    \bottomrule
    \multicolumn{3}{@{}l}{\footnotesize $^{**}\,p < 0.01$;
    $^{***}\,p < 0.001$. Instance FE included.}
  \end{tabular}
\end{table}

\begin{figure}[t]
  \centering
  \includegraphics[width=0.55\textwidth]{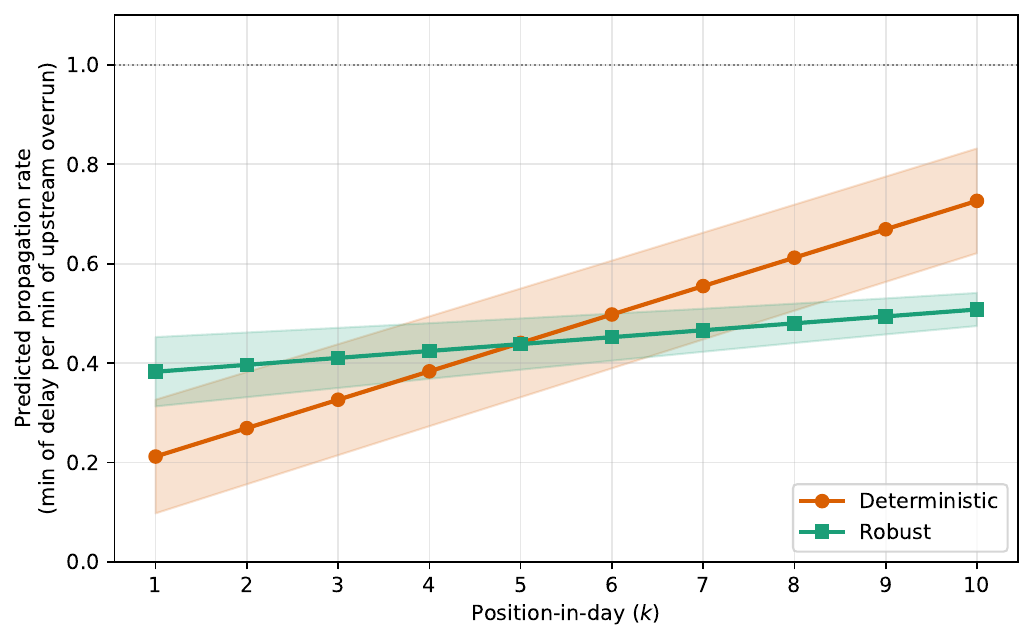}
  \caption{Implied propagation rate (minutes of delay per minute of
    cumulative upstream shock) by schedule and position-in-day, with
  95\% confidence bands.}
  \label{fig:cascade-propagation}
\end{figure}

\subsection{Managerial Implications}
\label{sec:case-study-implications}

The results of the case study carry three implications for how
hospitals should evaluate and adopt scheduling under uncertainty.

\paragraph{The value comes from controlling the tail.}
The most pronounced benefits of robust scheduling over the
deterministic approach appear in the right tail of the delay distribution.
Severe delays exceeding 90 minutes are nearly eliminated, falling from
$3.0$ to $0.1$ surgeries per instance, and the
$95$th-percentile delay drops from $626$ to $7$ minutes (a reduction of
roughly $620$ minutes, or ten hours). The robust $\alpha$-menu dominates
the deterministic baseline on every metric, cutting the worst-day failure
probability from $1.00$ to $0.20$, the number of delayed surgeries from
$7.9$ to $0.1$, and the end-of-day (last-surgery) delay from $251$ to
$8$ minutes.

The cascade-control regression of
Section~\ref{sec:case-study-mechanism} identifies the source: the
$\alpha$-menu's variance-aware buffers flatten the propagation curve,
insulating late-day surgeries from accumulated upstream drift. The
adoption question is whether buffers respond to the variance of
each procedure.

\paragraph{Two emergent consequences: less overtime and
endogenously balanced workload.}
Two findings run against the intuition that adding buffers must come
at a cost. First, total overtime falls by 42\% under the robust
schedule rather than increasing: planned slack absorbs realized
shocks before they become reactive overtime.
Second, although the model contains no term that balances workload, the robust
schedules are also more level across rooms than the
deterministic baseline, with the workload coefficient of variation
falling modestly from $0.32$ to $0.27$, a 15\% reduction. This mild
balancing is not
separately optimized for: it arises from the same mechanism as the
reliability gains --- placing buffers in proportion to each surgery's
duration variance --- so that reliability control and a more even
distribution of load emerge together from a single design choice.

\paragraph{The economic value extends beyond the operating room.}
The overtime reduction translates into avoided staffing and
facility costs, and the value of OR schedule reliability extends
beyond the OR itself: chronic unpredictability in the schedule
propagates through downstream wards, drives reliance on agency
staffing at premium rates, and contributes to OR nurse turnover.
The robust schedule eliminates approximately 250 minutes per
surgical week of overtime, or roughly 12{,}500 minutes per year
over $\sim$50 surgical weeks.
Table~\ref{tab:economic-impact} converts this into annualized cost
savings under three assumptions on the cost rate of OR overtime,
drawn from the published US literature on OR economics.

The figures are illustrative: the savings in minutes come from the
back test at HLA Moncloa, while the dollar conversion applies US
overtime cost rates for the OR \citep{Moody2020,
ChildersMaggardGibbons2018, Smith2022, DexterEpsteinMarsh2001}
and excludes the downstream effects on wards described above. If we read
it as an order of magnitude, the conservative scenario amounts to
roughly a quarter of a million dollars per year of avoided overtime
cost for a single surgical service of medium size.

\begin{table}[!htb]
  \centering
  \caption{Annualized overtime cost savings under three
    assumptions on the cost rate of OR overtime. Rates apply a
    $1.5\times$--$1.75\times$ overtime multiplier
    (\citealt{DexterEpsteinMarsh2001}) to regular-time benchmarks
  adjusted to 2024 USD.}
  \label{tab:economic-impact}
  \renewcommand{\arraystretch}{0.9}
  \begin{tabular}{@{}lrr@{}}
    \toprule
    \rowcolor{hdrgray}Scenario & Cost (\$/min) & Savings \\
    \midrule

    Conservative & \$22 & $\sim$\$0.28M \\
    Moderate     & \$60 & $\sim$\$0.75M \\
    Optimistic   & \$80 & $\sim$\$1.0M  \\
    \bottomrule
  \end{tabular}
\end{table}

\subsection{Robustness and Sensitivity}
\label{sec:case-study-robustness}

To ensure that the observed improvements are not driven by a
particular instance realization or parameter configuration, we
conduct several robustness checks.

\paragraph{Calibration check (nominal vs.\ realized durations).}
We check whether the gains merely reflect underestimated nominal
durations --- if $\hat\mu$ understates true durations, the robust
approach would only be correcting a biased mean. For each surgery we
compute $X/\hat\mu$, with $X$ the realized duration and $\hat\mu$
estimated out-of-sample from data before the window.
Table~\ref{tab:calibration} shows modest over-runs (surgery-weighted
mean $1.06$; per-instance range $0.96$--$1.20$), and the robust
schedule dominates the deterministic baseline at both ends ---
instance~7 (ratio $1.20$, worst-calibrated) and instance~5 (ratio
$0.96$, realized shorter than nominal).

\begin{table}[t]
  \centering
  \caption{Out-of-sample calibration of nominal duration estimates.
  Last row: surgery-count-weighted summary.}
  \label{tab:calibration}
  \footnotesize
  \setlength{\tabcolsep}{4pt}
  \renewcommand{\arraystretch}{0.95}
  \begin{tabular}{@{}cccccc@{}}
    \toprule
    \rowcolor{hdrgray}& \multicolumn{2}{c}{Instance design} &  &
    \multicolumn{2}{c}{Calibration} \\
    \cmidrule(lr){2-3}\cmidrule(l){5-6}
    \rowcolor{hdrgray}Instance $\ell$ & $L_\ell$ (days) & $R_\ell$ (rooms) & \# Surg. &
    $\overline{X/\hat{\mu}}$ & $\mathrm{med}(X/\hat{\mu})$ \\
    \midrule
    1  & 2 & 6  & 58  & 0.98 & 0.97 \\
    2  & 5 & 10 & 214 & 1.07 & 1.04 \\
    3  & 4 & 7  & 96  & 1.03 & 1.02 \\
    4  & 3 & 9  & 145 & 1.12 & 1.09 \\
    5  & 2 & 8  & 77  & 0.96 & 0.95 \\
    6  & 5 & 9  & 227 & 1.05 & 1.03 \\
    7  & 2 & 3  & 45  & 1.20 & 1.15 \\
    8  & 4 & 8  & 180 & 1.10 & 1.07 \\
    9  & 2 & 3  & 45  & 1.01 & 1.00 \\
    10 & 3 & 7  & 132 & 1.06 & 1.05 \\
    \midrule
    All & \multicolumn{2}{c}{--} & 1219 & 1.06 & 1.04 \\
    \bottomrule
  \end{tabular}
\end{table}

\paragraph{The menu allocates reliability heterogeneously.} Across the ten
instances the optimizer spreads surgeries over all eight menu levels
rather than collapsing onto one: $4.0\%$ at $\alpha=0.005$, $3.2\%$ at
$0.01$, $12.1\%$ at $0.02$, $50.8\%$ at $0.03$, $6.5\%$ at $0.04$,
$4.8\%$ at $0.05$, $0.8\%$ at $0.075$, and $17.7\%$ at $0.10$. This is
the mechanism of Section~\ref{sec:exp-risk-postures} at work on field
data: a case's capacity cost of protection scales with $\sigma_j$
while its reliability-budget cost is $\sigma$-independent, so the
optimizer tightens reliability where it is comparatively cheap ---
something a single global $\alpha$ cannot do, and which is muted only
when idle capacity is abundant.

\section{Conclusions}
\label{sec:conclusions}

In this paper we develop a robust chance-constrained framework for
elective surgery scheduling under uncertain durations. The framework
is built around a modular $\alpha$-menu architecture: a finite set of
reliability levels $\{\alpha_1, \dots, \alpha_T\}$ defines per-surgery
buffered durations, and the optimizer selects one level per case
subject to a day-level reliability budget. The mapping that produces
buffered durations from reliability levels is implemented by
interchangeable engines --- Cantelli (moment-based),
Wasserstein-$W_\infty$ (max-shift), and Wasserstein-$W_1$
(average-shift) --- that share a common interface with the rest of
the model. Three risk-posture variants (M1 average, M2 worst-day,
M3 hard-target) embed the reliability budget into the objective in
different ways, each corresponding to a distinct way a hospital can
internalize tail risk. We evaluate the framework on ten
rolling-origin instances drawn from HLA Moncloa Hospital and identify
the empirical mechanism through which it delivers its gains.

As discussed in Section~\ref{sec:day-level-posture}, the three
postures map to distinct payment environments --- average-reliability M1
to throughput-driven (fee-for-service) settings, worst-day M2 to
value-based or reputationally exposed ones (bundled payments, Spain's
\emph{concierto}, NHS block contracts), and hard-target M3 to SLA or
regulatory floors; M3 is also the most tractable, requiring no weight
parameter.
In the synthetic experiments, M2 achieves day-level reliability
within a band of less
than one percentage point across all five days of the planning
horizon --- a degree of uniformity that M1 cannot match even at very
high reliability weights. The
framework thus makes the alignment between scheduling objective and
financial risk structure explicit: a hospital adopting
M1 under contracts that penalize tail events is misaligned with its
risk environment, just as one adopting M2 under pure fee-for-service
may be overinvesting in robustness relative to its actual exposure.

The case study quantifies the operational
value of the framework: the
robust schedule cuts severe delays by 97\% and total overtime by
42\% relative to the deterministic baseline, translating to roughly
\$0.3--\$1.0 million per year of avoided overtime cost under
standard US OR rates. The cascade-control regression identifies the
source of these gains: variance-aware buffers --- slack scaled with each
surgery's duration uncertainty rather than a uniform inflation of nominal
durations. The framework's value should therefore be largest where
two conditions hold: rooms host long chains of consecutive surgeries
on the same day (so the cascade has room to build) and surgery
variance differs widely across cases (so allocating slack
selectively pays off relative to a uniform allocation). Hospitals
with short rooms, with one surgery per day, or with little variance
across procedures are expected to see proportionally smaller gains.

The present work has two scope limitations. First, we model only
elective surgeries. Emergencies can be added to the framework in two
ways without changing the core formulation. If the emergency
is expected (a
typical daily volume of trauma cases, for example), it can be
included up front as an anticipated emergency case with
duration drawn from
historical emergency data; the model then schedules it like any
other case and leaves a gap of the appropriate size. If the
emergency arrives during the day, the model can be re-solved on the
fly: surgeries already completed are locked in place, the new case
is inserted, and the rest of the day is rescheduled. In both cases,
the buffers placed on elective surgeries provide slack that absorbs
part of the emergency. Second, our estimates and the Wasserstein
radius $\rho$ are calibrated under an i.i.d.\ assumption, whereas
hospital duration distributions are inherently nonstationary ---
protocols evolve, surgical teams change, patient demographics
shift --- so $\rho$ can be reinterpreted as protection against
distributional drift rather than only against sampling noise.
\citet{KeehanAndersonWiesemann2025} develop principled methods for
setting $\rho$ when historical observations are generated by a
time-varying process; integrating these methods into the
buffer-engine layer would complete the data-driven pipeline from
historical records to robust schedules. More broadly, the results
suggest that reliability is best treated as an allocable scheduling
resource: when duration uncertainty is heterogeneous and delays cascade
through the day, choosing \emph{where} to place protection can
matter as much
as choosing \emph{how much} to add.

\bibliographystyle{plainnat}
\bibliography{references}

@article{Gupta2007,
  author  = {Gupta, Diwakar},
  title   = {Surgical Suites' Operations Management},
  journal = {Production and Operations Management},
  year    = {2007},
  volume  = {16},
  number  = {6},
  pages   = {689--700},
  doi     = {10.1111/j.1937-5956.2007.tb00289.x}
}

@article{Cardoen2010,
  author    = {Cardoen, Brecht and Demeulemeester, Erik and Beli{\"e}n, Jeroen},
  title     = {Operating room planning and scheduling: A literature review},
  journal   = {European Journal of Operational Research},
  volume    = {201},
  number    = {3},
  pages     = {921--932},
  year      = {2010},
  doi       = {10.1016/j.ejor.2009.04.011}
}

@article{MaySpanglerStrumVargas2011,
  author  = {May, Jerrold H. and Spangler, William E. and Strum, David P. and Vargas, Luis G.},
  title   = {The Surgical Scheduling Problem: Current Research and Future Opportunities},
  journal = {Production and Operations Management},
  year    = {2011},
  volume  = {20},
  number  = {3},
  pages   = {392--405},
  doi     = {10.1111/j.1937-5956.2011.01221.x}
}

@article{LamiriXieDolguiGrimaud2008,
  author  = {Lamiri, Mehdi and Xie, Xiaolan and Dolgui, Alexandre and Grimaud, Fr{\'e}d{\'e}ric},
  title   = {A Stochastic Model for Operating Room Planning with Elective and Emergency Demand for Surgery},
  journal = {European Journal of Operational Research},
  year    = {2008},
  volume  = {185},
  number  = {3},
  pages   = {1026--1037},
  doi     = {10.1016/j.ejor.2006.02.057}
}

@article{DentonMillerBalasubramanianHuschka2010,
  author  = {Denton, Brian T. and Miller, Andrew J. and Balasubramanian, Hari and Huschka, Todd R.},
  title   = {Optimal Allocation of Surgery Blocks to Operating Rooms Under Uncertainty},
  journal = {Operations Research},
  year    = {2010},
  volume  = {58},
  number  = {4},
  pages   = {802--816},
  doi     = {10.1287/opre.1090.0791}
}

@article{JebaliDiabat2017,
  author  = {Jebali, Aida and Diabat, Ali},
  title   = {A Chance-Constrained Operating Room Planning with Elective and Emergency Cases under Downstream Capacity Constraints},
  journal = {Computers \& Industrial Engineering},
  year    = {2017},
  volume  = {114},
  pages   = {329--344},
  doi     = {10.1016/j.cie.2017.07.015}
}

@article{RoshanaeiLuongAlemanUrbach2017,
  author  = {Roshanaei, Vahid and Luong, Curtiss and Aleman, Dionne M. and Urbach, David R.},
  title   = {Collaborative Operating Room Planning and Scheduling},
  journal = {INFORMS Journal on Computing},
  year    = {2017},
  volume  = {29},
  number  = {3},
  pages   = {558--580},
  doi     = {10.1287/ijoc.2017.0745}
}

@article{WangZhangTang2019,
  author  = {Wang, Yu and Zhang, Yu and Tang, Jiafu},
  title   = {A Distributionally Robust Optimization Approach for Surgery Block Allocation},
  journal = {European Journal of Operational Research},
  year    = {2019},
  volume  = {273},
  number  = {2},
  pages   = {740--753},
  doi     = {10.1016/j.ejor.2018.08.037}
}

@article{BandiGupta2020,
  author  = {Bandi, Chaithanya and Gupta, Diwakar},
  title   = {Operating Room Staffing and Scheduling},
  journal = {Manufacturing \& Service Operations Management},
  year    = {2020},
  volume  = {22},
  number  = {5},
  pages   = {958--974},
  doi     = {10.1287/msom.2019.0781}
}

@article{ShehadehPadman2021,
  author  = {Shehadeh, Karmel S. and Padman, Rema},
  title   = {A Distributionally Robust Optimization Approach for Stochastic Elective Surgery Scheduling with Limited Intensive Care Unit Capacity},
  journal = {European Journal of Operational Research},
  year    = {2021},
  volume  = {290},
  number  = {3},
  pages   = {901--913},
  doi     = {10.1016/j.ejor.2020.09.001}
}

@article{CarewNagarajanShechterArnejaSkarsgard2021,
  author  = {Carew, Stephanie and Nagarajan, Mahesh and Shechter, Steven and Arneja, Jugpal and Skarsgard, Erik},
  title   = {Dynamic Capacity Allocation for Elective Surgeries: Reducing Urgency-Weighted Wait Times},
  journal = {Manufacturing \& Service Operations Management},
  year    = {2021},
  volume  = {23},
  number  = {2},
  pages   = {407--424},
  doi     = {10.1287/msom.2019.0846}
}

@article{NaderiRoshanaeiBegenAlemanUrbach2021,
  author  = {Naderi, Bahman and Roshanaei, Vahid and Begen, Mehmet A. and Aleman, Dionne M. and Urbach, David R.},
  title   = {Increased Surgical Capacity without Additional Resources: Generalized Operating Room Planning and Scheduling},
  journal = {Production and Operations Management},
  year    = {2021},
  volume  = {30},
  number  = {8},
  pages   = {2608--2635},
  doi     = {10.1111/poms.13397}
}

@article{ZhouParlarVerterFraser2021,
  author  = {Zhou, Yun and Parlar, Mahmut and Verter, Vedat and Fraser, Shannon},
  title   = {Surgical Scheduling with Constrained Patient Waiting Times},
  journal = {Production and Operations Management},
  year    = {2021},
  volume  = {30},
  number  = {9},
  pages   = {3253--3271},
  doi     = {10.1111/poms.13427}
}

@article{AzarCarrascoMondschein2022,
  author  = {Azar, Macarena and Carrasco, Rodrigo A. and Mondschein, Susana},
  title   = {Dealing with Uncertain Surgery Times in Operating Room Scheduling},
  journal = {European Journal of Operational Research},
  year    = {2022},
  volume  = {299},
  number  = {1},
  pages   = {377--394},
  doi     = {10.1016/j.ejor.2021.09.010}
}

@article{Shehadeh2022Wasserstein,
  author  = {Shehadeh, Karmel S.},
  title   = {Data-Driven Distributionally Robust Surgery Planning in Flexible Operating Rooms over a {W}asserstein Ambiguity},
  journal = {Computers \& Operations Research},
  year    = {2022},
  volume  = {146},
  pages   = {105927},
  doi     = {10.1016/j.cor.2022.105927}
}

@article{WangZhangZhouTang2023,
  author  = {Wang, Yu and Zhang, Yu and Zhou, Minglong and Tang, Jiafu},
  title   = {Feature-Driven Robust Surgery Scheduling},
  journal = {Production and Operations Management},
  year    = {2023},
  volume  = {32},
  number  = {6},
  pages   = {1921--1938},
  doi     = {10.1111/poms.13949}
}

@article{WangZhangTang2024,
  author  = {Wang, Yu and Zhang, Yu and Tang, Jiafu},
  title   = {Wasserstein Distributionally Robust Surgery Scheduling with Elective and Emergency Patients},
  journal = {European Journal of Operational Research},
  year    = {2024},
  volume  = {314},
  number  = {2},
  pages   = {509--522},
  doi     = {10.1016/j.ejor.2023.10.026}
}

@article{FuQiYangYe2024,
  author  = {Fu, Xiaojin and Qi, Jin and Yang, Chen and Ye, Han},
  title   = {Elective Surgery Sequencing and Scheduling Under Uncertainty},
  journal = {Manufacturing \& Service Operations Management},
  year    = {2024},
  volume  = {26},
  number  = {3},
  pages   = {893--910},
  doi     = {10.1287/msom.2022.0029}
}

@article{DengShenDenton2019,
  author  = {Deng, Yan and Shen, Siqian and Denton, Brian T.},
  title   = {Chance-Constrained Surgery Planning Under Conditions of Limited and Ambiguous Data},
  journal = {INFORMS Journal on Computing},
  year    = {2019},
  volume  = {31},
  number  = {3},
  pages   = {559--575},
  doi     = {10.1287/ijoc.2018.0835}
}

@article{TsangEtAl2025,
  author  = {Tsang, Man Yiu and Shehadeh, Karmel S. and Curtis, Frank E. and Hochman, Beth R. and Brentjens, Tricia E.},
  title   = {Stochastic Optimization Approaches for an Operating Room and Anesthesiologist Scheduling Problem},
  journal = {Operations Research},
  year    = {2025},
  volume  = {73},
  number  = {3},
  pages   = {1430--1458},
  doi     = {10.1287/opre.2022.0258}
}

@article{FreemanMeloukMittenthal2016,
  author  = {Freeman, Nickolas K. and Melouk, Sharif H. and Mittenthal, John},
  title   = {A Scenario-Based Approach for Operating Theater Scheduling Under Uncertainty},
  journal = {Manufacturing \& Service Operations Management},
  year    = {2016},
  volume  = {18},
  number  = {2},
  pages   = {245--261},
  doi     = {10.1287/msom.2015.0557}
}

@article{GulDentonFowlerHuschkaBard2011,
  author  = {Gul, Serhan and Denton, Brian T. and Fowler, James W. and Huschka, Todd R. and Bard, Jonathan F.},
  title   = {Bi-Criteria Scheduling of Surgical Services for an Outpatient Surgery Center},
  journal = {Production and Operations Management},
  year    = {2011},
  volume  = {20},
  number  = {3},
  pages   = {406--417},
  doi     = {10.1111/j.1937-5956.2011.01232.x}
}

@article{MakRongZhang2015,
  author  = {Mak, Ho-Yin and Rong, Ying and Zhang, Jiawei},
  title   = {Appointment Scheduling with Limited Distributional Information},
  journal = {Management Science},
  year    = {2015},
  volume  = {61},
  number  = {2},
  pages   = {316--334},
  doi     = {10.1287/mnsc.2013.1881}
}

@article{BertsimasPopescu2005,
  author    = {Bertsimas, Dimitris and Popescu, Ioana},
  title     = {Optimal inequalities in probability theory: A convex optimization approach},
  journal   = {SIAM Journal on Optimization},
  volume    = {15},
  number    = {3},
  pages     = {780--804},
  year      = {2005},
  doi       = {10.1137/S1052623401399903}
}

@article{MohajerinKuhn2018,
  author    = {Mohajerin Esfahani, Peyman and Kuhn, Daniel},
  title     = {Data-driven distributionally robust optimization using the {W}asserstein metric: performance guarantees and tractable reformulations},
  journal   = {Mathematical Programming},
  volume    = {171},
  number    = {1--2},
  pages     = {115--166},
  year      = {2018},
  doi       = {10.1007/s10107-017-1172-1}
}

@misc{KeehanAndersonWiesemann2025,
  author       = {Keehan, Dominic S. T. and Anderson, Edward J. and Wiesemann, Wolfram},
  title        = {Don't Look Back in Anger: {W}asserstein Distributionally Robust Optimization with Nonstationary Data},
  year         = {2025},
  howpublished = {arXiv preprint arXiv:2510.18566},
  url          = {https://arxiv.org/abs/2510.18566}
}

@article{DexterEpsteinMarsh2001,
  author  = {Dexter, Franklin and Epstein, Richard H. and Marsh, H. Michael},
  title   = {A Statistical Analysis of Weekday Operating Room Anesthesia Group Staffing Costs at Nine Independently Managed Surgical Suites},
  journal = {Anesthesia \& Analgesia},
  year    = {2001},
  volume  = {92},
  number  = {6},
  pages   = {1493--1498},
  doi     = {10.1097/00000539-200106000-00028}
}

@article{ChildersMaggardGibbons2018,
  author  = {Childers, Christopher P. and Maggard-Gibbons, Melinda},
  title   = {Understanding Costs of Care in the Operating Room},
  journal = {JAMA Surgery},
  year    = {2018},
  volume  = {153},
  number  = {4},
  pages   = {e176233},
  doi     = {10.1001/jamasurg.2017.6233}
}

@article{Moody2020,
  author  = {Moody, Alastair E. and Gurnea, Taylor P. and Shul, Craig P. and Althausen, Peter L.},
  title   = {True Cost of Operating Room Time: Implications for an Orthopaedic Trauma Service},
  journal = {Journal of Orthopaedic Trauma},
  year    = {2020},
  volume  = {34},
  number  = {5},
  pages   = {271--275},
  doi     = {10.1097/BOT.0000000000001688}
}

@article{Smith2022,
  author  = {Smith, Tyler and Evans, Justin and Moriel, Karla and Tihista, Mikel and Bacak, Christopher and Dunn, John and Rajani, Rajiv and Childs, Benjamin},
  title   = {Cost of {OR} Time is \${46.04} per Minute},
  journal = {Journal of Orthopaedic Business},
  year    = {2022},
  volume  = {2},
  number  = {4},
  pages   = {10--13},
  doi     = {10.55576/job.v2i4.23}
}

\ECSwitch

\ECHead{Proofs and Supplemental Material}

\section{Cantelli: Tightness Construction}
\label{ec:cantelli}
Cantelli's inequality is tight: for any $a>0$, there exists a distribution with mean $\mu$ and variance $\sigma^2$ that achieves equality in \eqref{eq:cantelli}. One such extremal distribution is the two-point distribution
\begin{equation}
\label{eq:cantelli-tight-dist}
X =
\begin{cases}
\mu - \dfrac{\sigma^2}{a} & \text{with probability } \dfrac{a^2}{\sigma^2+a^2}, \\[10pt]
\mu + a & \text{with probability } \dfrac{\sigma^2}{\sigma^2+a^2}.
\end{cases}
\end{equation}
A direct verification shows this distribution has mean $\mu$, variance $\sigma^2$, and satisfies
$\Pr(X-\mu\ge a)=\sigma^2/(\sigma^2+a^2)$.

\section{Wasserstein $W_\infty$: Proof of Proposition~\ref{prop:Winf-buffer}}
\label{ec:winf}
\proof{Proof of Proposition~\ref{prop:Winf-buffer}.}
Under the constraint $W_\infty(Q,\hat{P}_N)\le \rho$, each Dirac mass at $x_i$ can be transported to any point in $[x_i-\rho,\;x_i+\rho]$. For any threshold $\tau$, the exceedance probability is maximized by transporting every point to $x_i+\rho$, which produces a shifted empirical distribution. Therefore,
\[
\sup_{Q:W_\infty(Q,\hat{P}_N)\le \rho} Q(X>\tau)
=
\hat{P}_N(X>\tau-\rho).
\]
Let $\hat{Q}_{1-\alpha}$ denote the empirical $(1-\alpha)$-quantile. Choosing $\tau=\hat{Q}_{1-\alpha}+\rho$ ensures
$\hat{P}_N(X>\tau-\rho)=\hat{P}_N(X>\hat{Q}_{1-\alpha})\le \alpha$,
so the buffer is feasible. Conversely, for any
$p<\hat{Q}_{1-\alpha}+\rho$ we have
$\hat{P}_N(X\le p-\rho)<1-\alpha$ by definition of the (left)
empirical quantile, hence the worst-case exceedance
$\hat{P}_N(X>p-\rho)$ exceeds $\alpha$; thus
$\hat{Q}_{1-\alpha}+\rho$ is also the smallest feasible buffered
duration, which proves the claim.
\Halmos
\endproof

\section{Wasserstein $W_1$: Proof of Proposition~\ref{prop:W1-tail}}
\label{ec:w1-proof}
\proof{Proof of Proposition~\ref{prop:W1-tail}.}
Under $\widehat P_N=(1/N)\sum_{i=1}^N\delta_{x_i}$, each sample $x_i$ carries
mass $1/N$. Any mass already above $\tau$ may be left in place at zero cost,
contributing $\frac1N\sum_{i=1}^N \mathbf 1\{x_i>\tau\}$ to the tail probability.
For each $i$ with $x_i\le \tau$, let $z_i\in[0,1]$ denote the fraction of the
atom at $x_i$ transported into $(\tau,\infty)$; such mass moves at least distance
$\tau-x_i$, so any coupling with average transport cost at most $\rho$ satisfies
$\frac1N\sum_{i:x_i\le\tau}(\tau-x_i)z_i\le \rho$. Hence the added tail mass is at
most the fractional knapsack value
$\frac1N \max_z\{\sum_{i:x_i\le\tau} z_i:
\sum_{i:x_i\le\tau}(\tau-x_i)z_i\le N\rho,\;0\le z_i\le 1\}$. Conversely, moving
those fractions from $x_i$ to $\tau+\delta$ and letting $\delta\downarrow0$
attains this value, so the bound is exact. Since each unit of moved mass has the
same value and cost $\tau-x_i$, the knapsack is solved greedily---spend the
budget on the atoms closest to $\tau$ from below, with at most one partial
atom---an $O(N)$ scan on sorted samples.
\Halmos
\endproof

\section{Sample-path Feasibility: Full Proof of Proposition~\ref{prop:sample-path}}
\label{ec:samplepath}

\MV{The probabilistic bounds asserted in the proposition are already established
in Section~\ref{sec:probabilistic-interp} of the main paper: the
independence bound~\eqref{eq:eps-bound} and the dependence-robust union
bound~\eqref{eq:union-bound-main}. What remains is the underlying sample-path inclusion $E_d\subseteq F_d$, from which
$\Pr(F_d^c)\le\Pr(E_d^c)$ and hence both bounds follow.}

\proof{Proof of Proposition~\ref{prop:sample-path}.}
Fix a day $d$ and a feasible solution of M$_b$. For each surgery $j\in\mathcal J_d$,
let $(r(j),k(j),t(j))$ denote its selected room, surgeon, and menu level, so that
$w_{j\,d\,r(j)\,k(j)\,t(j)}=1$. By constraints
\eqref{eq:core-z}--\eqref{eq:core-dsg}, this implies $z_{j\,d\,r(j)}=1$ and
$y_{j\,d\,k(j)}=1$, and the buffered duration assigned to surgery $j$ is
$\delta_j:=p_{j\,r(j)\,k(j)}^{(t(j))}=\delta^{\mathrm{rm}}_{j\,d\,r(j)}=\delta^{\mathrm{sg}}_{j\,d\,k(j)}$.

Now fix any outcome $\omega\in E_d$. By definition of $E_d$,
$X_{j\,r(j)\,k(j)}(\omega)\le \delta_j$ for all $j\in\mathcal J_d$.

We prove the three claims.

\medskip
\noindent\textbf{(i) No overlap on any room timeline.}
Fix a room $r\in\mathcal R$ and any two distinct surgeries $j,i\in\mathcal J_d$ assigned to room $r$ on day $d$. Without loss of generality, assume $j<i$. Since $z_{jdr}=z_{idr}=1$,
feasibility of M$_b$ and constraints \eqref{eq:room-seq1}--\eqref{eq:room-seq2} imply
\begin{align}
s^{\mathrm{rm}}_{jdr}+\delta^{\mathrm{rm}}_{jdr}
&\le s^{\mathrm{rm}}_{idr}
+ M_d^{\mathrm{rm}}(1-u^{\mathrm{rm}}_{ji\,dr}),
\label{eq:sp-room-1}\\
s^{\mathrm{rm}}_{idr}+\delta^{\mathrm{rm}}_{idr}
&\le s^{\mathrm{rm}}_{jdr}
+ M_d^{\mathrm{rm}}u^{\mathrm{rm}}_{ji\,dr}.
\label{eq:sp-room-2}
\end{align}
Because $u^{\mathrm{rm}}_{ji\,dr}\in\{0,1\}$, either
$s^{\mathrm{rm}}_{jdr}+\delta^{\mathrm{rm}}_{jdr}\le s^{\mathrm{rm}}_{idr}$ or
$s^{\mathrm{rm}}_{idr}+\delta^{\mathrm{rm}}_{idr}\le s^{\mathrm{rm}}_{jdr}$.
Thus the two surgeries do not overlap in buffered time on room $r$.

Since $\omega\in E_d$, realized durations are no larger than buffered durations. Therefore buffered completion times dominate realized completion times, and the same precedence relation remains valid in realization $\omega$. Hence surgeries $j$ and $i$ do not overlap in realized time on room $r$.

Because the pair $(j,i)$ was arbitrary, no two surgeries assigned to room $r$ overlap in realization $\omega$. Since $r$ was arbitrary, no room timeline overlap occurs on day $d$.

\medskip
\noindent\textbf{(ii) No overlap on any surgeon timeline.}
Fix a surgeon $k\in\mathcal K$ and any two distinct surgeries $j,i\in\mathcal J_d$ assigned to surgeon $k$ on day $d$. Without loss of generality, assume $j<i$. Since $y_{jdk}=y_{idk}=1$,
constraints \eqref{eq:sg-seq1}--\eqref{eq:sg-seq2} imply
\begin{align}
s^{\mathrm{sg}}_{jdk}+\delta^{\mathrm{sg}}_{jdk}
&\le s^{\mathrm{sg}}_{idk}
+ M_d^{\mathrm{sg}}(1-u^{\mathrm{sg}}_{ji\,dk}),
\label{eq:sp-sg-1}\\
s^{\mathrm{sg}}_{idk}+\delta^{\mathrm{sg}}_{idk}
&\le s^{\mathrm{sg}}_{jdk}
+ M_d^{\mathrm{sg}}u^{\mathrm{sg}}_{ji\,dk}.
\label{eq:sp-sg-2}
\end{align}
Again, because $u^{\mathrm{sg}}_{ji\,dk}\in\{0,1\}$, either
$s^{\mathrm{sg}}_{jdk}+\delta^{\mathrm{sg}}_{jdk}\le s^{\mathrm{sg}}_{idk}$ or
$s^{\mathrm{sg}}_{idk}+\delta^{\mathrm{sg}}_{idk}\le s^{\mathrm{sg}}_{jdk}$.
Hence the two surgeries do not overlap in buffered time on surgeon $k$'s timeline.

Since $\omega\in E_d$, realized durations do not exceed buffered durations, so the same precedence relation remains valid in realized time. Therefore surgeries $j$ and $i$ do not overlap on surgeon $k$ in realization $\omega$.

As the pair and surgeon were arbitrary, no surgeon timeline overlap occurs on day $d$.

\medskip
\noindent\textbf{(iii) Completion within planned horizons plus overtime.}
Let $j\in\mathcal J_d$ be any surgery, with selected room $r(j)$ and surgeon $k(j)$.
Since $z_{j\,d\,r(j)}=1$, constraint \eqref{eq:core-room-comp} gives
$s^{\mathrm{rm}}_{j\,d\,r(j)}+\delta^{\mathrm{rm}}_{j\,d\,r(j)}\le H_d + o^{\mathrm{rm}}_{d\,r(j)}$.
Because $\omega\in E_d$,
$X_{j\,r(j)\,k(j)}(\omega)\le \delta^{\mathrm{rm}}_{j\,d\,r(j)}$, and therefore
$s^{\mathrm{rm}}_{j\,d\,r(j)}+X_{j\,r(j)\,k(j)}(\omega)\le H_d + o^{\mathrm{rm}}_{d\,r(j)}$.
Thus surgery $j$ completes within the planned room horizon plus allowed room overtime.

Similarly, since $y_{j\,d\,k(j)}=1$, constraint \eqref{eq:core-sg-comp} gives
$s^{\mathrm{sg}}_{j\,d\,k(j)}+\delta^{\mathrm{sg}}_{j\,d\,k(j)}\le C_{k(j)d}+o^{\mathrm{sg}}_{d\,k(j)}$.
Using again that $\omega\in E_d$,
$X_{j\,r(j)\,k(j)}(\omega)\le \delta^{\mathrm{sg}}_{j\,d\,k(j)}$, hence
$s^{\mathrm{sg}}_{j\,d\,k(j)}+X_{j\,r(j)\,k(j)}(\omega)\le C_{k(j)d}+o^{\mathrm{sg}}_{d\,k(j)}$.
Thus surgery $j$ also completes within the planned surgeon horizon plus allowed surgeon overtime.

Combining (i)--(iii), the realized schedule on day $d$ is feasible for every outcome $\omega\in E_d$.
\Halmos
\endproof

\section{Arbitrary Dependence and Worst-Case Tightness of the Linear Budget}
\label{sec:ec-bonferroni}

This section develops the arbitrary-dependence counterpart of the log-budget interpretation. The marginal Cantelli guarantee holds regardless of dependence; only the day-level aggregation changes, with the union bound replacing the product formula.

Fix a day $d$. For each surgery $j\in\mathcal J_d$, let $E_j := \{X_j \le p_j\}$ and $E_j^c := \{X_j > p_j\}$, where $p_j$ is the buffered duration induced by the chosen assignment and menu level, and set $E_d := \bigcap_{j\in\mathcal J_d} E_j$, $E_d^c := \bigcup_{j\in\mathcal J_d} E_j^c$. If surgery $j$ is assigned level $\alpha_j:=\alpha_{t(j)}$, the single-surgery Cantelli guarantee gives $\Pr(E_j^c)\le \alpha_j$, so without any independence assumption,
\begin{equation}
\label{eq:ec-union-bound}
\Pr(E_d^c)
=
\Pr\!\Bigl(\bigcup_{j\in\mathcal J_d} E_j^c\Bigr)
\le
\sum_{j\in\mathcal J_d}\Pr(E_j^c)
\le
\sum_{j\in\mathcal J_d}\alpha_j.
\end{equation}
This motivates the day-specific linear reliability budget $\sum_{j\in\mathcal J_d}\alpha_j \le \varepsilon_d$, which guarantees $\Pr(E_d^c)\le \varepsilon_d$ under arbitrary dependence. The next proposition shows this bound is worst-case tight: it is the best universal day-level guarantee obtainable from the marginal Cantelli bounds alone.

\begin{proposition}[Worst-case tightness of the linear budget]
\label{prop:bonferroni-tight}
Fix a day $d$ with surgeries $j=1,\dots,m$, buffers
$p_j=\mu_j+\kappa_j\sigma_j$ ($\kappa_j=\sqrt{(1-\alpha_j)/\alpha_j}$,
$\alpha_j\in(0,1)$, $\sigma_j>0$), and $E_j^c:=\{X_j>p_j\}$; assume
$\sum_{j=1}^m\alpha_j\le1$. Then for every $\eta>0$ there is a joint distribution
of $(X_1,\dots,X_m)$ under which each marginal has mean $\mu_j$, variance
$\sigma_j^2$, and $\Pr(E_j^c)\le\alpha_j$, the events $E_1^c,\dots,E_m^c$ are
mutually exclusive, and $\Pr\bigl(\bigcup_{j=1}^m E_j^c\bigr)\ge\sum_{j=1}^m\alpha_j-\eta$.
Consequently no dependence-robust guarantee $\Pr(E_d^c)\le\varepsilon_d$ can hold
uniformly over all joint distributions with the prescribed marginal means and
variances whenever $\varepsilon_d<\sum_{j=1}^m\alpha_j$.
\end{proposition}

\proof{Proof.}
Fix $\eta>0$. For each $j$ pick $\delta_j>0$ and set
$h_j:=\mu_j+\kappa_j\sigma_j+\delta_j$, $\ell_j:=\mu_j-\sigma_j^2/(h_j-\mu_j)$, and
$q_j:=\sigma_j^2/[\sigma_j^2+(h_j-\mu_j)^2]$. The two-point law placing mass $q_j$
at $h_j$ and $1-q_j$ at $\ell_j$ has mean $\mu_j$, variance $\sigma_j^2$, and
$\Pr(X_j>p_j)=q_j$; since $h_j-\mu_j>\kappa_j\sigma_j$ we get $q_j<\alpha_j$, with
$q_j\uparrow\alpha_j$ as $\delta_j\downarrow0$, so the $\delta_j$ can be chosen
with $\sum_j q_j\ge\sum_j\alpha_j-\eta$. Couple these marginals through $m+1$
scenarios: $(X_1,\dots,X_m)=x^{(j)}$ with probability $q_j$, where
$x^{(j)}:=(\ell_1,\dots,\ell_{j-1},h_j,\ell_{j+1},\dots,\ell_m)$, and
$(\ell_1,\dots,\ell_m)$ with the remaining probability $1-\sum_j q_j\ge0$. Each
marginal is the two-point law above, and in every scenario at most one component
exceeds its buffer, so the $E_j^c$ are mutually exclusive and
$\Pr\bigl(\bigcup_{j=1}^m E_j^c\bigr)=\sum_j q_j\ge\sum_j\alpha_j-\eta$.
\Halmos
\endproof

\MV{
\section{\texorpdfstring{Max-Min Reliability and Equalization under Transferable Load}{Max-Min Reliability and Equalization under Transferable Load}}
\label{ec:m2-equalization}

This section clarifies what Model M2 implies, and does not imply, about the distribution of reliability across days.

\subsection{M2 as a max-min formulation and large-weight limit}

Let $\mathcal X$ be the feasible set of schedules, and let
$C(x):=w_{\rm idle}C_{\rm idle}(x)+w_{\rm room}C_{\rm room}(x)
+w_{\rm doc}C_{\rm doc}(x)+w_{\rm start}C_{\rm start}(x)$ denote operating
cost. For day log-budgets $b_d(x)$, write
$m(x):=\min_{d\in D}b_d(x)$. In M2, $b_{\min}\le b_d(x)$ for all $d$ and
$b_{\min}$ is rewarded in the objective. Hence $b_{\min}=m(x)$ at optimality.
Writing $w=w_\varepsilon$, M2 is equivalent to
\begin{equation}
    \label{eq:ec-m2-reduced}
    \min_{x\in\mathcal X}\; C(x)-w m(x).
\end{equation}
Because $e^{b_d(x)}$ is the day-level reliability lower bound and
$1-e^{b_d(x)}$ is the corresponding violation proxy, M2 pays for improving the
least reliable day.

Let $B^*:=\max_{x\in\mathcal X}m(x)$, equivalently
$B^*=\max_{x\in\mathcal X}\min_{d\in D}b_d(x)$, be the best achievable
worst-day log-budget. As $w$ grows, M2 drives the bottleneck day toward $B^*$.
Managerially, M2 is a max-min reliability policy: it protects the weakest day
first. This does not guarantee equal reliability across all days. Equalization
requires transferability, i.e., capacity or workload must be movable from slack
days to the bottleneck day.

\begin{proposition}[Large-weight limit of M2]
\label{prop:ec-m2-maxmin-limit}
Assume $\mathcal X$ is nonempty and that $C$ is bounded on $\mathcal X$:
there exist finite constants $C_{\min}$ and $C_{\max}$ such that
$C_{\min}\le C(x)\le C_{\max}$ for all $x\in\mathcal X$. Let $x^w$ be an
optimal solution of \eqref{eq:ec-m2-reduced}. Then $m(x^w)\to B^*$ as
$w\to\infty$.
If, in addition, $\mathcal X$ is compact and $C$ and $b_d$ are continuous, every
limit point of $\{x^w\}_{w\ge 0}$ belongs to
$\operatorname*{arg\,max}_{x\in\mathcal X} m(x)$.
\end{proposition}

\proof{Proof.}
Let $x^*$ be any maximizer of $m(x)$ over $\mathcal X$, so that $m(x^*)=B^*$.
By optimality of $x^w$ in \eqref{eq:ec-m2-reduced},
$C(x^w)-w\,m(x^w)\le C(x^*)-w B^*$.
Rearranging gives
$0\le B^*-m(x^w)\le (C(x^*)-C(x^w))/w \le (C_{\max}-C_{\min})/w$.
The right-hand side converges to zero, proving the first claim. If
$\mathcal X$ is compact, any subsequence has a convergent subsubsequence. If
$x^{w_k}\to \bar x$, continuity gives
$m(\bar x)=\lim_k m(x^{w_k})=B^*$, so
$\bar x\in\operatorname*{arg\,max}_x m(x)$.
\Halmos
\endproof

Proposition~\ref{prop:ec-m2-maxmin-limit} is the strongest conclusion that
follows from the M2 objective alone. It implies convergence of the worst day to
its best achievable reliability level, but it does not imply
$b_d(x^w)\to B^*$ for every day $d$. Some days may remain strictly above the
worst-day level. Below we provide a simple illustration of this phenomenon.
}

\MV{
\subsection{A Counterexample: M2 Need Not Equalize Reliability}
\label{ec:m2-counterexample-compact}

We give a minimal instance in which M2 does not equalize day-level reliability. Consider two days, one room and one surgeon per day, and one surgery per day, with overtime permitted up to the room cap $\bar O^{\mathrm{rm}}$. Buffered durations are $p_d(\alpha_d)=\mu_d+\sigma_d\kappa(\alpha_d)$ with $\kappa(\alpha)=\sqrt{(1-\alpha)/\alpha}$, and the menu is $\Acal=\{0.01,0.05,0.10\}$. The operating cost is the idle-plus-overtime cost $C_{\mathrm{op}}$ of~\eqref{eq:cop}; with a single room and surgeon per day, room and surgeon overtime coincide, so we write the idle time as $I_d=(H_d-p_d)^+$ and the overtime as $o_d=(p_d-H_d)^+$ and take unit coefficients $c^{\mathrm{idle}}=c^{\mathrm{rm}}=1$, $c^{\mathrm{sg}}=0$, giving $C_{\mathrm{op}}=\sum_d (I_d+o_d)$. Take $H_d=480$, an overtime cap $\bar O^{\mathrm{rm}}=60$, a high-variance bottleneck day $(\mu_1,\sigma_1)=(330,50)$, and a low-variance slack day $(\mu_2,\sigma_2)=(436.41,10)$.

Day~1 is capacity-bound: even $\alpha_1=0.05$ would require $67.9>\bar O^{\mathrm{rm}}$ minutes of overtime, so the only feasible choice is $\alpha_1=0.10$, with $p_1=480=H_d$ and $R_1=0.90$. The worst day is therefore pinned at $R_1=0.90$ and cannot be improved. Day~2 has slack: $\alpha_2=0.05$ fills the day exactly ($p_2=480$, no idle and no overtime, $R_2=0.95$), while $\alpha_2=0.01$ would raise it to $R_2=0.99$ at the price of $55.9$ minutes of overtime. Since day~1 is always the worst day, the worst-case reward $w_\varepsilon b_{\min}=w_\varepsilon\log 0.90$ is constant regardless of the day~2 choice, so M2 simply minimizes cost on the slack day and leaves it at its cheapest level $\alpha_2=0.05$. For every $w_\varepsilon$, M2 selects $(R_1,R_2)=(0.90,0.95)$: the optimum is not equalized (Figure~\ref{ec-fig:m2-counterex-m1m2}, left panel). Note that equalizing is in fact feasible here---setting $\alpha_2=0.10$ would give $(R_1,R_2)=(0.90,0.90)$---but it yields the same worst-day reliability ($b_{\min}=\log 0.90$, pinned by day~1) while wasting $13.59$ idle minutes, so M2 declines it. M2 is thus genuinely max-min: indifferent among schedules with equal worst-day reliability and breaking ties on cost, rather than equality-seeking. This is the failure mode the section is about---when the bottleneck day is capacity-bound and reliability cannot be shifted onto it, M2 protects that day as far as feasibility permits but cannot raise it to the level of the slack day, so reliabilities remain unequal. Equalization re-emerges precisely when capacity is transferable across days, the regime in which the multi-room, multi-surgery instances of the main paper operate and in which M2 is empirically observed to equalize.

It is instructive to contrast the average posture M1 on the same instance. Unlike M2, M1 is rewarded for \emph{any} reliability gain, including on a day that is not the bottleneck. Day~1 remains pinned at $R_1=0.90$, but M1 now pays the overtime to lift the already-safe day~2 from $0.95$ to $0.99$ as soon as $w_\varepsilon>2\times 55.9/(\log 0.99-\log
0.95)\approx 2{,}711$ (the factor $2=|\mathcal D|$ reflects the
$1/|\mathcal D|$ normalization in $\bar b$), choosing $(R_1,R_2)=(0.90,0.99)$ (Figure~\ref{ec-fig:m2-counterex-m1m2}, right panel). Neither posture equalizes, since day~1 is capacity-bound below both; but they differ in \emph{where} reliability is bought. M2 concentrates effort on the bottleneck and spends nothing on the already-protected day, keeping the two days as close as feasibility permits; M1 spreads protection onto the strong day, widening the reliability gap and raising cost. This separation is consistent with the main paper: among the postures, M2 is the one that leans toward equalization, declining to buy reliability where it is not scarce.

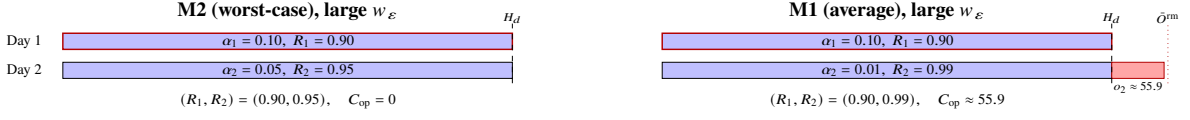
\begin{figure}[t]
\centering
\resizebox{0.95\linewidth}{!}{%
\begin{tikzpicture}[x=0.018cm,y=0.62cm, font=\scriptsize]
\tikzset{used/.style={draw=black,fill=blue!25},
         idle/.style={draw=black,fill=gray!18,pattern=north east lines},
         ot/.style={draw=red!70!black,fill=red!35},
         worst/.style={draw=red!70!black,line width=0.8pt}}
\node[font=\bfseries] at (240,1.20) {M2 (worst-case), large $w_\varepsilon$};
\node[anchor=east] at (-15,0.25) {Day 1};
\draw[used,worst] (0,0) rectangle (480,0.5);
\node at (240,0.25) {$\alpha_1=0.10,\ R_1=0.90$};
\node[anchor=east] at (-15,-0.65) {Day 2};
\draw[used] (0,-0.9) rectangle (480,-0.4);
\node at (240,-0.65) {$\alpha_2=0.05,\ R_2=0.95$};
\draw[dashed] (480,-1.05) -- (480,0.80);
\node[font=\tiny] at (480,0.97) {$H_d$};
\node at (240,-1.60) {$(R_1,R_2)=(0.90,0.95),\quad C_{\mathrm{op}}=0$};
\def\xb{640}
\node[font=\bfseries] at (\xb+240,1.20) {M1 (average), large $w_\varepsilon$};
\draw[used,worst] (\xb+0,0) rectangle (\xb+480,0.5);
\node at (\xb+240,0.25) {$\alpha_1=0.10,\ R_1=0.90$};
\draw[used] (\xb+0,-0.9) rectangle (\xb+480,-0.4);
\draw[ot] (\xb+480,-0.9) rectangle (\xb+535.91,-0.4);
\node at (\xb+240,-0.65) {$\alpha_2=0.01,\ R_2=0.99$};
\draw[dashed] (\xb+480,-1.05) -- (\xb+480,0.80);
\node[font=\tiny] at (\xb+480,0.97) {$H_d$};
\draw[red!70!black,dotted] (\xb+540,-1.05) -- (\xb+540,0.80);
\node[font=\tiny] at (\xb+540,0.97) {$\bar O^{\mathrm{rm}}$};
\node[font=\tiny] at (\xb+508,-1.18) {$o_2\approx 55.9$};
\node at (\xb+240,-1.60) {$(R_1,R_2)=(0.90,0.99),\quad C_{\mathrm{op}}\approx 55.9$};
\end{tikzpicture}%
}
\caption[Worst-case versus average posture on the same instance.]{\MV{Worst-case versus average posture on the same instance. Day~1 is capacity-bound at $R_1=0.90$ under both. M2 leaves the slack day~2 at its cheapest feasible level $R_2=0.95$ (it earns no reward for protecting a non-bottleneck day), whereas M1 spends overtime (red) to raise day~2 to $R_2=0.99$. Neither equalizes; the postures differ in \emph{where} reliability is bought---M2 concentrates on the bottleneck, M1 spreads onto the already-safe day.}}
\label{ec-fig:m2-counterex-m1m2}
\end{figure}
}

\section{Mechanism Diagnostics for the $\alpha$-Menu}
\label{ec:mechanism}

This section provides supporting mechanism diagnostics for the M1--M3 risk-posture comparison reported in Section~\ref{sec:exp-risk-postures} of the main paper. The headline quantitative results are summarized in Tables~\ref{tab:m1m2m3-summary} and~\ref{tab:multi-instance-outcomes}; here we present two complementary visualizations that illuminate how the optimizer translates the day-level reliability budget into per-surgery decisions. \MV{We then add a fixed-schedule analytical relaxation that explains the empirical menu patterns: monotone assignment by variance, depolarization under high reliability prices, and the capacity value of case-specific buffers.}

\subsection{Depolarization of the $\alpha$-Menu under Increasing Reliability Penalty}
\label{ec:depolarization}

As the reliability penalty grows, the solver depolarizes: extreme $\alpha$ choices vanish and all surgeries converge to a moderate band. The reliability differences between M1 and M2 arise from how each model allocates per-surgery risk budgets. Figure~\ref{ec-fig:alpha-stacked} displays the $\alpha$-menu distribution across all seven runs (M1$\times 3w_\varepsilon$ + M2$\times 3w_\varepsilon$ + M3). At low $w_\varepsilon$, the solver makes polarized choices: it assigns many surgeries the most aggressive levels ($\alpha = 0.075$--$0.10$, short buffers) while a few receive conservative levels ($\alpha \le 0.01$, long buffers) --- 46\% of M1's and 56\% of M2's choices fall at these extremes. As $w_\varepsilon$ increases, the distribution does not simply shift toward conservative values; rather, it depolarizes. At $w_\varepsilon = 1{,}000$, the extreme share drops to 20\% (M1) and 48\% (M2), and at $w_\varepsilon = 100{,}000$ it vanishes entirely: every surgery is assigned a moderate $\alpha \in \{0.02, 0.03, 0.04\}$. The solver finds that a uniformly moderate risk allocation --- neither overly aggressive nor overly conservative on any individual surgery --- yields the best trade-off between operational cost and aggregate reliability. M3 exhibits a similar pattern at 80\% moderate, with a small tail at $\alpha = 0.075$ reflecting the freedom that its hard floor constraint leaves the optimizer.

\subsection{Surgery-Size Differentiation within the Moderate Band}
\label{ec:size-correlation}

The moderate band is not uniform: the solver differentiates by surgery characteristics. Figure~\ref{ec-fig:alpha-vs-duration} plots the per-surgery $\alpha$ against expected duration $\mu_j$ for M2 at $w_\varepsilon = 100{,}000$. Longer surgeries receive systematically higher $\alpha$ (Pearson $r = 0.58$): small procedures average $\alpha = 0.025$, medium $\alpha = 0.032$, and large $\alpha = 0.039$. \MV{This is operationally intuitive: because $b_j=\sigma_j\kappa(\alpha_j)$, a higher-variance surgery can receive more absolute buffer time even when it is assigned a larger $\alpha_j$. The $\alpha$-menu enables this surgery-specific fine-tuning; a fixed-$\alpha$ policy imposes the same standardized buffer multiple $\kappa(\alpha)$ on every procedure, so buffer minutes scale linearly with $\sigma_j$, forgoing the efficiency gains from adapting to surgery heterogeneity.} The degree of $\alpha$-differentiation scales with the diversity of the surgical caseload: in settings with greater procedure heterogeneity, the full menu would be exploited.

\begin{figure}[ht]
\centering
\begin{subfigure}[t]{0.49\textwidth}
  \centering
  \includegraphics[width=\textwidth]{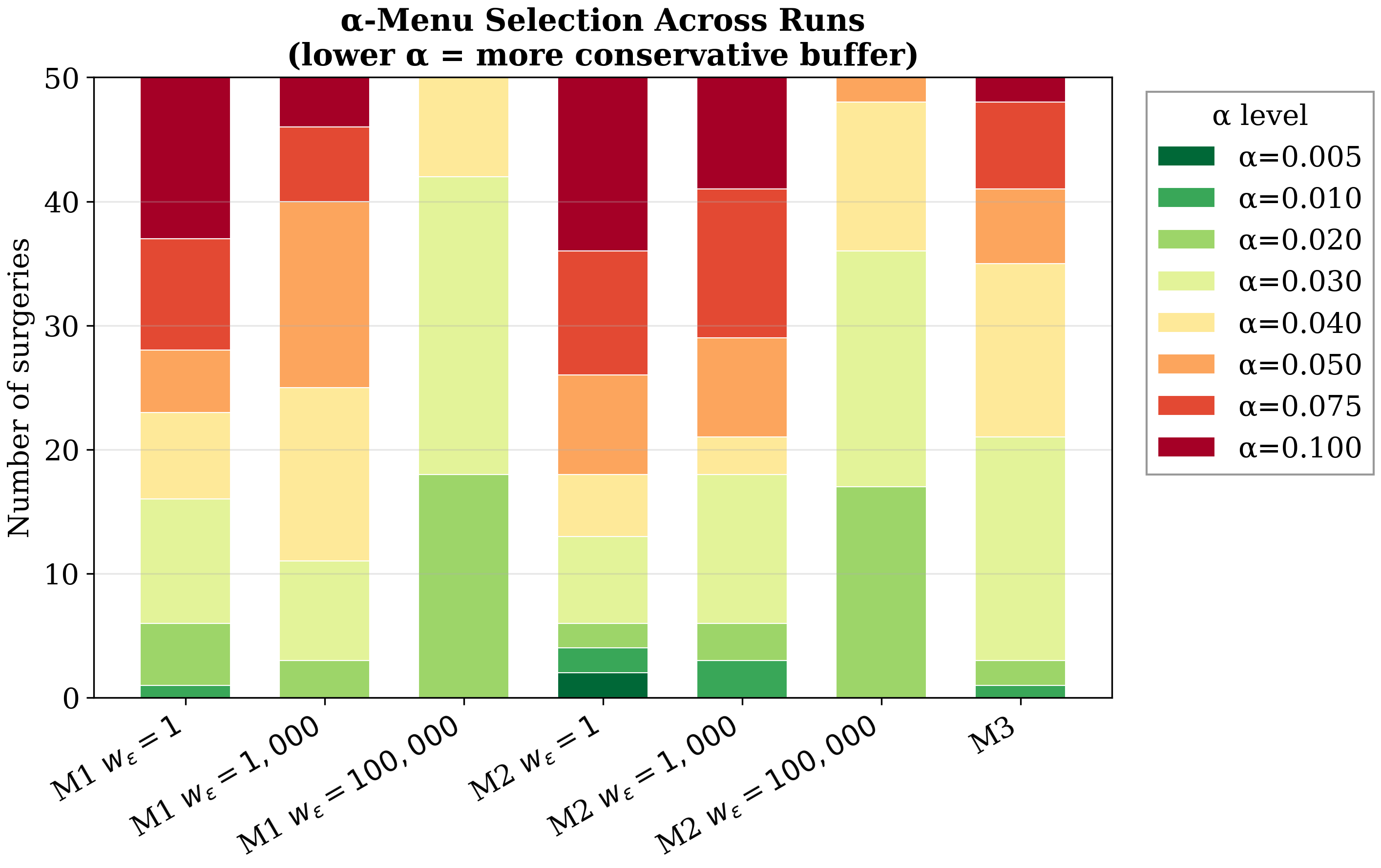}
  \caption{$\alpha$-menu selection across all runs. Lower $\alpha$ (green): longer buffers; higher $\alpha$ (red): shorter buffers.}
  \label{ec-fig:alpha-stacked}
\end{subfigure}\hfill
\begin{subfigure}[t]{0.49\textwidth}
  \centering
  \includegraphics[width=\textwidth]{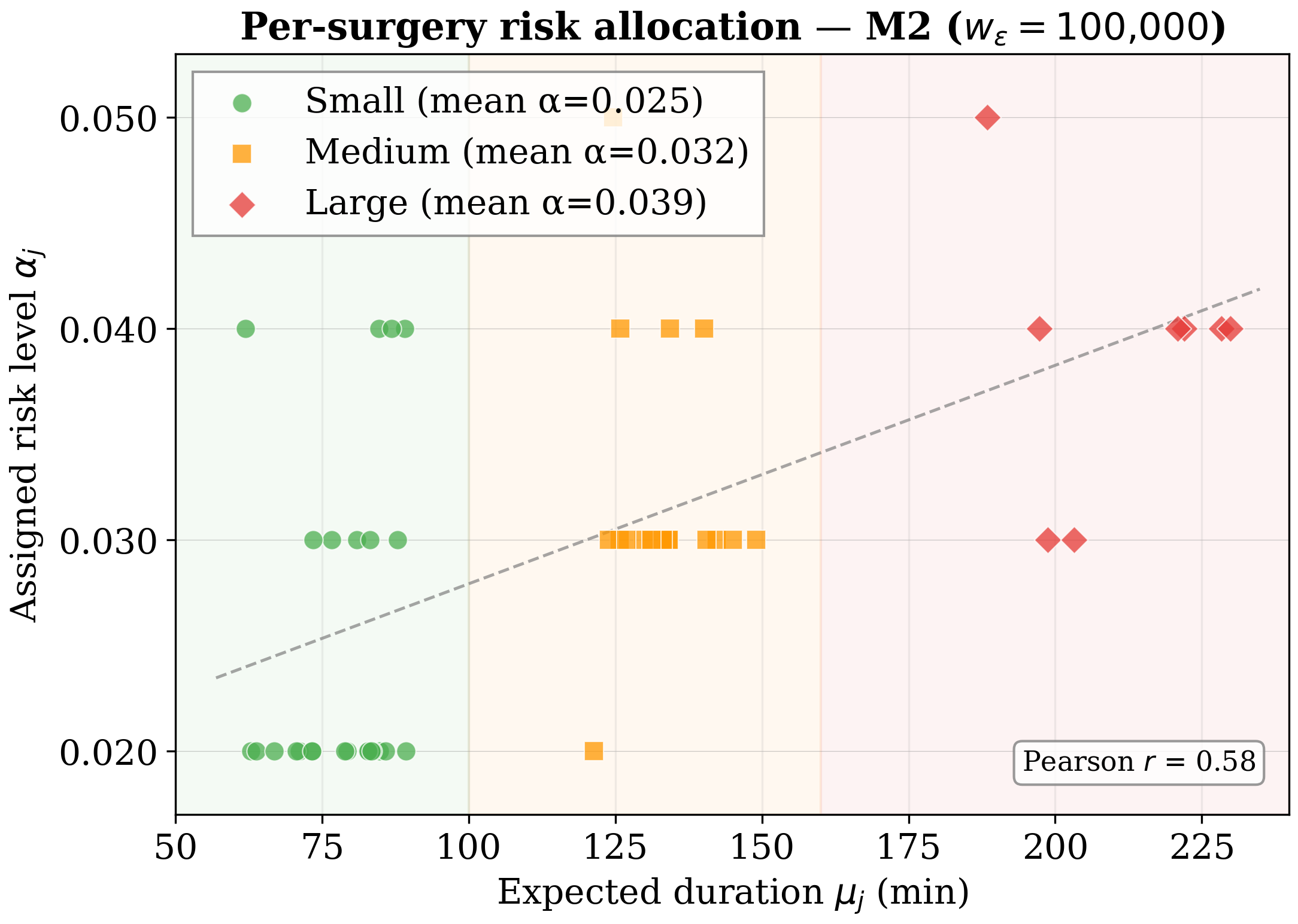}
  \caption{Per-surgery $\alpha$ vs.\ expected duration $\mu_j$ at $w_\varepsilon = 100{,}000$ (M2). Shape and color indicate surgery type. Pearson $r = 0.58$.}
  \label{ec-fig:alpha-vs-duration}
\end{subfigure}
\end{figure}

\MV{
\subsection{Analytical Explanation: Structure and Value of the $\alpha$-Menu}
\label{ec:menu-value-condensed}

The preceding diagnostics show how the optimizer uses the $\alpha$-menu; the
fixed-schedule relaxation below explains why those patterns are not accidental.
Fixing the day composition, room assignment, and within-day order and varying
only the per-surgery reliability levels, we obtain stylized relaxations of the
full MILP that isolate the reliability-allocation mechanism and explain why the
model sorts reliability by variance, why the menu depolarizes at high
reliability prices, the $\sigma^{2/3}$ buffer-scaling rule, and when it improves
on scalar buffer rules. For surgery $j$ the Cantelli engine links buffer minutes
and tail probability by $\alpha_j=\sigma_j^2/(\sigma_j^2+b_j^2)$, equivalently
$b_j=\kappa(\alpha_j)\sigma_j$ with $\kappa(\alpha)=\sqrt{(1-\alpha)/\alpha}$, so
a lower $\alpha_j$ buys stronger protection and more buffer minutes; the key
asymmetry is that reliability consumption depends on $\alpha_j$ while capacity
consumption scales with $\sigma_j\kappa(\alpha_j)$. The fixed-order assumption is
a deliberate simplification---the full MILP also chooses rooms, surgeons,
sequencing, overtime, and day assignments---isolating the reliability-allocation
mechanism conditional on those decisions.

\subsubsection{Monotone menu assignment}

Consider one day with surgeries $j=1,\ldots,m$ and a finite menu
$0<\alpha^1<\cdots<\alpha^T<1$. The day log-budget depends on the selected
multiset through $\sum_j\log(1-\alpha_j)$, whereas total buffered workload is
$\sum_j\sigma_j\kappa(\alpha_j)$.

The cost monotonicity used below is a \emph{tight-load} condition, not a global
property of the main-paper objective, which charges both idle time and overtime.
In a one-resource abstraction with horizon $H$, buffered load $L$, and overtime
$o$, the minimum of $c^{\rm idle}(H+o-L)+c^{\rm ot}o$ is $c^{\rm idle}(H-L)$ when
$L\le H$ and $c^{\rm ot}(L-H)$ when $L\ge H$: an extra buffered minute lowers
measured idle on a slack room but raises overtime or consumes scarce slack once
the resource is tight. The proposition therefore applies where capacity has a
positive shadow price; where a day has abundant idle capacity, the sorting
result becomes secondary.

\begin{proposition}[Monotone reliability assignment]
\label{prop:ec-menu-monotone}
\MV{Suppose all reliability constraints are invariant to permutations of the
selected menu levels, and all remaining $\alpha$-dependent feasibility
constraints depend on the assignment only through total buffered workload and
are preserved when that workload decreases. Assume the comparison is in a
tight-load regime in which operational cost is weakly increasing in total
buffered workload.} Then there
exists an optimal assignment with
$\sigma_i<\sigma_j\Rightarrow \alpha_i^*\le\alpha_j^*$. If the
operational cost is strictly increasing in total buffered workload, no optimal
solution contains an inversion with $\sigma_i<\sigma_j$ and
$\alpha_i^*>\alpha_j^*$.
\end{proposition}

\proof{Proof.}
The function $\kappa(\alpha)=((1-\alpha)/\alpha)^{1/2}$ is strictly decreasing
on $(0,1)$. Suppose $\sigma_i<\sigma_j$ but $\alpha_i>\alpha_j$. Swapping the two
menu levels leaves the day log-budget and all permutation-invariant reliability
constraints unchanged. The workload change from the inverted assignment to the
swapped assignment is
$\Delta=\sigma_i\kappa(\alpha_j)+\sigma_j\kappa(\alpha_i)
-\sigma_i\kappa(\alpha_i)-\sigma_j\kappa(\alpha_j)
=-(\sigma_j-\sigma_i)\{\kappa(\alpha_j)-\kappa(\alpha_i)\}<0$.
Thus the swap weakly improves the objective and, under strict monotonicity,
strictly improves it. Repeating adjacent swaps removes all inversions.
\Halmos
\endproof

\begin{remark}[Managerial reading]
The implication is operational rather than cosmetic: the menu buys probability
protection where it consumes the fewest minutes. High-variance cases may still
receive more absolute buffer time, but they optimally receive weaker tail
probabilities than comparable low-variance cases.
\end{remark}

\subsubsection{Equi-marginal allocation}

The monotone-assignment proposition above is a finite-menu interchange result.
The following stylized continuous relaxation is its positive-shadow-price
counterpart: buffer minutes are treated as scarce capacity, so the problem
minimizes total buffered workload while meeting a reliability budget.
Let $g(\alpha):=-\log(1-\alpha)$ and consider a fixed day with budget
$B>0$:
\begin{equation}
\label{eq:ec-menu-cont}
    \min_{\alpha\in(0,\bar\alpha]^m}
    \sum_{j=1}^m \sigma_j\kappa(\alpha_j)
    \quad\text{s.t.}\quad
    \sum_{j=1}^m g(\alpha_j)\le B .
\end{equation}
We call a solution \emph{interior} if $0<\alpha_j^*<\bar\alpha$ for every
$j$; at such a point the variable-bound multipliers are zero.

\begin{lemma}[Convexity of the continuous relaxation]
\label{lem:ec-menu-convexity}
The function $\kappa$ is strictly convex on $(0,3/4)$, and $g$ is convex and
increasing on $(0,1)$. Hence the continuous relaxation above is a convex
program for any $\bar\alpha\le 3/4$.
\end{lemma}

\proof{Proof.}
The derivatives are
$\kappa'(\alpha)=-1/\{2\alpha^{3/2}(1-\alpha)^{1/2}\}$,
$g'(\alpha)=1/(1-\alpha)$, and $g''(\alpha)=1/(1-\alpha)^2>0$.
For $h(\alpha)=\alpha^{-3/2}(1-\alpha)^{-1/2}$,
$h'(\alpha)=\alpha^{-5/2}(1-\alpha)^{-3/2}(2\alpha-3/2)<0$
on $(0,3/4)$. Since $\kappa'=-h/2$, $\kappa'$ is strictly increasing on this
interval, so $\kappa$ is strictly convex there. The claim follows because the
objective is a nonnegative weighted sum of convex terms and the feasible set is
a sublevel set of the convex function $\sum_j g(\alpha_j)$.
\Halmos
\endproof

\begin{proposition}[Equi-marginal rule and conservative limit]
\label{prop:ec-menu-equi}
\MV{Assume $\sigma_j>0$ for all $j$; the statements below describe interior
optima, with active menu bounds giving the corresponding truncated KKT
allocation.} At any interior optimum of the continuous relaxation with a binding reliability
budget and multiplier $\lambda>0$,
$\sigma_j\varphi(\alpha_j^*)=\lambda$, where
$\varphi(\alpha):=(1-\alpha)^{1/2}/(2\alpha^{3/2})$.
Consequently:
\begin{enumerate}[label=(\roman*),leftmargin=2em]
\item $\alpha_j^*=\varphi^{-1}(\lambda/\sigma_j)$ is strictly increasing
in $\sigma_j$;
\item as the budget tightens ($B\downarrow$), $\lambda$ increases and every
$\alpha_j^*$ decreases;
\item in the conservative regime $\lambda\to\infty$,
\begin{equation}
\label{eq:ec-menu-two-thirds}
\begin{aligned}
    \alpha_j^* &=\left(\frac{\sigma_j}{2\lambda}\right)^{2/3}(1+o(1)),\\
    b_j^* &=(2\lambda)^{1/3}\sigma_j^{2/3}(1+o(1)).
\end{aligned}
\end{equation}
Moreover
$\max_j\alpha_j^*-\min_j\alpha_j^*=O(\lambda^{-2/3})$, while
$\alpha_i^*/\alpha_j^*\to(\sigma_i/\sigma_j)^{2/3}$.
\end{enumerate}
\end{proposition}

\proof{Proof.}
By Lemma~\ref{lem:ec-menu-convexity}, the KKT conditions are necessary and
sufficient when $\bar\alpha\le3/4$. The Lagrangian is
$L(\alpha,\lambda)=\sum_{j=1}^m \sigma_j\kappa(\alpha_j)
+\lambda\bigl(\sum_{j=1}^m g(\alpha_j)-B\bigr)$.
Interior stationarity gives
$0=\sigma_j\kappa'(\alpha_j^*)+\lambda g'(\alpha_j^*)$.
Substituting $\kappa'$ and $g'$ yields
$\lambda=\sigma_j(1-\alpha_j^*)^{1/2}/\{2(\alpha_j^*)^{3/2}\}$,
which is the stated first-order rule,
$\sigma_j\varphi(\alpha_j^*)=\lambda$. Since
$d\log\varphi(\alpha)/d\alpha=-1/\{2(1-\alpha)\}-3/(2\alpha)<0$,
$\varphi^{-1}$ is decreasing, and
$\alpha_j^*=\varphi^{-1}(\lambda/\sigma_j)$ is increasing in $\sigma_j$.

The budget binds: otherwise, increasing some $\alpha_j$ would preserve
feasibility and reduce $\sigma_j\kappa(\alpha_j)$. Define
$H(\lambda):=\sum_{j=1}^m g\{\varphi^{-1}(\lambda/\sigma_j)\}$.
The binding condition is $H(\lambda)=B$. Because $\varphi^{-1}$ is decreasing
and $g$ is increasing, $H$ is strictly decreasing; hence tightening $B$ raises
$\lambda$ and decreases each $\alpha_j^*$.

For the conservative limit, $B\downarrow0$ is equivalent to
$\lambda\to\infty$, so $\alpha_j^*\downarrow0$. Since
$\varphi(\alpha)=\frac12\alpha^{-3/2}(1+O(\alpha))$ as $\alpha\downarrow0$,
the first-order rule gives
$(\alpha_j^*)^{3/2}=\sigma_j(1+O(\alpha_j^*))/(2\lambda)$ and
therefore
$\alpha_j^*=(\sigma_j/(2\lambda))^{2/3}(1+o(1))$. This implies
$\max_j\alpha_j^*-\min_j\alpha_j^*=O(\lambda^{-2/3})$ and
$\alpha_i^*/\alpha_j^*\to(\sigma_i/\sigma_j)^{2/3}$. Finally,
$\kappa(\alpha)=\alpha^{-1/2}(1+O(\alpha))$, so
$b_j^*=\sigma_j\kappa(\alpha_j^*)
=(2\lambda)^{1/3}\sigma_j^{2/3}(1+o(1))$.
\Halmos
\endproof

\begin{remark}[Managerial reading]
High-variance surgeries receive more minutes, but sublinearly
($b_j^*\propto\sigma_j^{2/3}$): a slope in $\sigma$ between uniform-minutes
padding (zero) and a fixed-$\alpha$ rule (one).
\end{remark}

\subsubsection{Risk posture as a day-level price}

The preceding fixed-budget relaxation is most directly related to M3, because
it treats day reliability as a binding target. The same within-day allocation
logic, however, is not tied to that risk posture. For a fixed multi-day
schedule, write
$W_d(\alpha):=\sum_{j\in\mathcal J_d}\sigma_j\kappa(\alpha_j)$ and
$G_d(\alpha):=\sum_{j\in\mathcal J_d}-\log(1-\alpha_j)$. Let
$\mathcal A:=\prod_{d\in D}\prod_{j\in\mathcal J_d}
[\underline\alpha_j,\bar\alpha_j]$ be the continuous analogue of the finite
menu, with $0<\underline\alpha_j<\bar\alpha_j<1$; surgery $j$ is interior if
$\underline\alpha_j<\alpha_j^*<\bar\alpha_j$. Writing
$F(\alpha):=\sum_d W_d(\alpha)$, the fixed-schedule relaxations are
\[
\begin{aligned}
\text{M1:}\quad&
\min_{\alpha\in\mathcal A}
F(\alpha)+\frac{w_\varepsilon}{|D|}
\sum_d G_d(\alpha),\\
\text{M2:}\quad&
\min_{\alpha\in\mathcal A,\ r}
F(\alpha)+w_\varepsilon r,\\
&G_d(\alpha)\le r,\qquad d\in D,\\
\text{M3:}\quad&
\min_{\alpha\in\mathcal A} F(\alpha),\\
&G_d(\alpha)\le B_d,\qquad d\in D .
\end{aligned}
\]
Here $r=-b_{\min}$ and $B_d=-b_{\rm rhs}$, allowing day-specific targets.

\begin{proposition}[Posture invariance of the within-day rule]
\label{prop:ec-menu-posture}
\MV{Within the displayed fixed-schedule relaxation, in which capacity is priced
uniformly across rooms, surgeons, and days,} for any day with positive day-level
reliability price $\lambda_d$, every
interior surgery satisfies
$\sigma_j\varphi(\alpha_j^*)=\lambda_d$, equivalently
$\alpha_j^*=\varphi^{-1}(\lambda_d/\sigma_j)$, with
\[
\lambda_d=
\begin{cases}
w_\varepsilon/|D|, & \text{in M1},\\
\eta_d, & \text{in M2},\\
\xi_d, & \text{in M3}.
\end{cases}
\]
Here $\eta_d\ge0$, $\sum_d\eta_d=w_\varepsilon$, and
$\eta_d(G_d-r)=0$ in M2, while $\xi_d\ge0$ and
$\xi_d(G_d-B_d)=0$ in M3. Days with zero price are boundary cases: the
relaxation assigns no marginal value to additional reliability on those days,
so the interior rule does not apply.
\end{proposition}

\proof{Proof.}
For $j\in\mathcal J_d$, $\partial W_d/\partial\alpha_j
=\sigma_j\kappa'(\alpha_j)$, $\partial G_d/\partial\alpha_j=1/(1-\alpha_j)$,
and $\kappa'(\alpha)=-1/\{2\alpha^{3/2}(1-\alpha)^{1/2}\}$.

\emph{M1.} Interior stationarity gives
$0=\sigma_j\kappa'(\alpha_j^*)
+(w_\varepsilon/|D|)/(1-\alpha_j^*)$, which rearranges to
$\sigma_j\varphi(\alpha_j^*)=w_\varepsilon/|D|$.

\emph{M2.} Let $\eta_d\ge0$ multiply $G_d(\alpha)-r\le0$. The Lagrangian is
$\sum_d W_d(\alpha)+w_\varepsilon r+\sum_d\eta_d\{G_d(\alpha)-r\}$.
Interior stationarity gives
$\sigma_j\varphi(\alpha_j^*)=\eta_d$; stationarity in $r$ gives
$\sum_d\eta_d=w_\varepsilon$, and complementarity gives
$\eta_d(G_d-r)=0$. Hence only worst-risk days can have positive price. If
$\eta_d=0$, the equation has no interior solution because $\sigma_j>0$ and
$\varphi(\alpha)>0$.

\emph{M3.} Let $\xi_d\ge0$ multiply $G_d(\alpha)-B_d\le0$. The Lagrangian is
$\sum_d W_d(\alpha)+\sum_d\xi_d\{G_d(\alpha)-B_d\}$.
Interior stationarity gives
$\sigma_j\varphi(\alpha_j^*)=\xi_d$, while complementarity gives
$\xi_d(G_d-B_d)=0$. If $\xi_d=0$, the same boundary logic applies. Since
$\varphi$ is strictly decreasing, every positive day price yields
$\alpha_j^*=\varphi^{-1}(\lambda_d/\sigma_j)$.
\Halmos
\endproof

\begin{remark}[Managerial reading]
The postures differ only in how they price days, not in how a priced day
allocates reliability across its surgeries.
\end{remark}

\subsubsection{Variance heterogeneity as an audit for menu value}

The menu is most valuable when the day is variance-heterogeneous. To illustrate
this mechanism in closed form, consider a second stylized relaxation in which a
planner has $S$ buffer minutes to allocate on a fixed day, with
$\sum_j b_j\le S$. The exact Cantelli tail contribution induced by buffer
$b_j$ is $\alpha_j(b_j)=\sigma_j^2/(\sigma_j^2+b_j^2)$. \MV{Here the day-level risk is the linear sum $\sum_j\alpha_j(b_j)$---the union-bound, dependence-robust budget of Section~\ref{sec:ec-bonferroni}, rather than the independence (log) budget used earlier in this subsection. At the small reliability levels on the menu the two nearly coincide ($-\log(1-\alpha)=\alpha+O(\alpha^2)$), so the allocation rule derived below is the same under either budget.} We use the convex
surrogate
\[
\bar\varepsilon_d(b):=\sum_j\frac{\sigma_j^2}{b_j^2}
\ge
\sum_j\frac{\sigma_j^2}{\sigma_j^2+b_j^2}.
\]
For one surgery, the inflation factor is
$(\sigma_j^2/b_j^2)/\{\sigma_j^2/(\sigma_j^2+b_j^2)\}
=1/(1-\alpha_j(b_j))$: it is about $1.111$ at $\alpha_j=0.10$ and about
$1.010$ at $\alpha_j=0.01$.

\begin{proposition}[Closed-form heterogeneity audit]
\label{prop:ec-menu-audit}
Let $S>0$ and $\sigma_j>0$ for $j=1,\ldots,m$. The solution of
$\min_{b>0}\{\bar\varepsilon_d(b):\sum_jb_j\le S\}$ is
\begin{equation}
\label{eq:ec-menu-audit-opt}
    b_j^*
    =
    \frac{S\sigma_j^{2/3}}{\sum_i\sigma_i^{2/3}},
    \qquad
    \bar\varepsilon_d^*
    =
    \frac{\left(\sum_{j=1}^m\sigma_j^{2/3}\right)^3}{S^2}.
\end{equation}
Relative to the optimal allocation in \eqref{eq:ec-menu-audit-opt}, define the
multiplicative loss of a feasible rule $b$ as
$\Gamma(b):=\bar\varepsilon_d(b)/\bar\varepsilon_d^*$.
Thus $\Gamma(b)=1$ means that the rule attains the optimal surrogate value,
whereas $\Gamma(b)=1.25$ means that it produces a 25\% larger surrogate
day-level failure bound using the same total buffer budget $S$. For the best
fixed-$\alpha$ rule, $b_j^{\rm fix}=S\sigma_j/\sum_i\sigma_i$, and for the
uniform-minutes rule, $b_j^{\rm unif}=S/m$. Their loss factors are
\[
\begin{aligned}
\Gamma_{\rm fix}
&:=
\frac{\bar\varepsilon_d(b^{\rm fix})}{\bar\varepsilon_d^*}
=
\frac{m(\sum_j\sigma_j)^2}
{(\sum_j\sigma_j^{2/3})^3},\\
\Gamma_{\rm unif}
&:=
\frac{\bar\varepsilon_d(b^{\rm unif})}{\bar\varepsilon_d^*}
=
\frac{m^2\sum_j\sigma_j^2}
{(\sum_j\sigma_j^{2/3})^3}.
\end{aligned}
\]
Both factors are at least one, with equality if and only if all variances are
equal.
\end{proposition}

\proof{Proof.}
Recall $\bar\varepsilon_d(b)=\sum_{j=1}^m \sigma_j^2/b_j^2$ with $b_j>0$ and
$\sum_{j=1}^m b_j\le S$. Each term $\sigma_j^2/b_j^2$ is decreasing in $b_j$, so
the budget is active at the optimum: $\sum_{j=1}^m b_j=S$. Form the Lagrangian
\[
L(b,\lambda)=\sum_{j=1}^m \frac{\sigma_j^2}{b_j^2}
+\lambda\Bigl(\sum_{j=1}^m b_j-S\Bigr).
\]
The first-order condition is
$\partial L/\partial b_j=-2\sigma_j^2/b_j^3+\lambda=0$, hence
$\lambda=2\sigma_j^2/b_j^3$, so $b_j^3=2\sigma_j^2/\lambda$ and
\[
b_j=\Bigl(\tfrac{2}{\lambda}\Bigr)^{1/3}\sigma_j^{2/3};
\]
the optimal buffers are proportional to $\sigma_j^{2/3}$. Imposing
$\sum_j b_j=S$ gives $(2/\lambda)^{1/3}\sum_{i=1}^m\sigma_i^{2/3}=S$, i.e.
$(2/\lambda)^{1/3}=S/\sum_{i=1}^m\sigma_i^{2/3}$, and therefore
\[
b_j^*=\frac{S\,\sigma_j^{2/3}}{\sum_{i=1}^m\sigma_i^{2/3}}.
\]
Substituting into the objective,
\[
\begin{aligned}
\bar\varepsilon_d(b^*)
&=\sum_{j=1}^m
\frac{\sigma_j^2}{\bigl(S\sigma_j^{2/3}/\sum_i\sigma_i^{2/3}\bigr)^2}\\
&=\frac{\bigl(\sum_i\sigma_i^{2/3}\bigr)^2}{S^2}\sum_{j=1}^m\sigma_j^{2/3}
=\frac{\bigl(\sum_{j=1}^m\sigma_j^{2/3}\bigr)^3}{S^2},
\end{aligned}
\]
the stated $\bar\varepsilon_d^*$. Because each map $b_j\mapsto\sigma_j^2/b_j^2$
is strictly convex on $b_j>0$, $\bar\varepsilon_d$ is strictly convex and this
first-order solution is the unique global optimum.

For the fixed-$\alpha$ rule, every surgery carries the same standardized
multiplier: the Cantelli buffer above the mean is $b_j=\sigma_j\kappa(\alpha)$
with $\kappa(\alpha)=\sqrt{(1-\alpha)/\alpha}$, so a common $\alpha$ gives
$b_j=c\,\sigma_j$ with common constant $c=\kappa(\alpha)>0$. Imposing the same
total budget, $\sum_{j=1}^m b_j=c\sum_{j=1}^m\sigma_j=S$, fixes
$c=S/\sum_{i=1}^m\sigma_i$, and therefore
$b_j^{\rm fix}=S\sigma_j/\sum_{i=1}^m\sigma_i$. Substituting,
$\bar\varepsilon_d^{\rm fix}=\sum_j\sigma_j^2/(b_j^{\rm fix})^2
=m(\sum_i\sigma_i)^2/S^2$; dividing by
$\bar\varepsilon_d^*$ gives $\Gamma_{\rm fix}$. For uniform minutes,
$b_j=S/m$, so
$\bar\varepsilon_d^{\rm unif}=m^2\sum_j\sigma_j^2/S^2$, and division by
$\bar\varepsilon_d^*$ gives $\Gamma_{\rm unif}$.

Finally, applying the standard mean inequality to the positive numbers
$\sigma_j$ gives
$(m^{-1}\sum_j\sigma_j^{2/3})^3\le(m^{-1}\sum_j\sigma_j)^2
\le m^{-1}\sum_j\sigma_j^2$, with equality if and only if all $\sigma_j$ are
equal. Rearranging yields the two lower bounds and their equality condition.
\Halmos
\endproof

\begin{remark}[Managerial reading]
The audit factors are computable from historical duration variances before
solving the full MILP: values near one indicate a nearly homogeneous day, for
which scalar buffer rules approximate the menu, while large values indicate that
the same total slack can be reallocated more effectively. Figure~\ref{fig:menu-dose-response}
of the main paper illustrates this across four dispersion profiles indexed by
this factor (their $\Gamma$ is the $\Gamma_{\rm fix}$ defined here): the menu's
certified advantage over the uniform policies widens as $\Gamma$ grows.
\end{remark}

}

\section{Robustness of the Cascade-Control Regression}
\label{ec:cascade-robustness}

This section reports robustness checks for the cascade-control regression of Section~\ref{sec:case-study-mechanism} of the main paper. The headline finding --- that the robust schedule attenuates the position-dependent build-up of propagation captured by $\beta_6 < 0$ --- holds under six alternative specifications.

(R1) Clustering standard errors at the room-day level (rather than instance) yields $F = 74.7$ ($p < 10^{-3}$); the substantive coefficients are identical.

(R2) Adding surgery-level controls for $\mu_j$ and $\sigma_j$ leaves $\beta_4$ and $\beta_6$ essentially unchanged.

(R3) Excluding the worst-calibrated instance (where the mean realized-to-planned ratio reaches 1.20) preserves the result.

(R4) Replacing the dependent variable with completion-time delay (realized end minus planned end) produces a similar pattern, with $\beta_6 = -0.034$ ($p < 10^{-3}$).

(R5) Allowing a quadratic in position adds no additional explanatory power: the cascade is approximately linear in position, and the linear interaction terms are stable.

(R6) An alternative specification that defines the upstream shock as $\max(0, X_i - p_{i})$ relative to each schedule's own planned duration --- and is therefore conflated with the buffer's surgery-level absorption effect --- shows the same position-dependent attenuation pattern ($\beta_5 = +0.057$, $\beta_6 = -0.057$, both $p < 10^{-3}$). The propagation rates themselves differ between the two specifications, as expected: under the buffer-relative measure, robust schedules show propagation rates near 1.0 at every position because by the time a residual shock survives the buffer, the day's slack has been exhausted --- which is itself informative about how aggressively the buffer absorbs shocks at the surgery level. The cascade-flattening result ($\beta_6 < 0$) is robust to either definition.

\section[Buffer Engines: Modularity and a Calibrated Plug-in Check]{\MV{Buffer-Engine Modularity and Held-Out Calibration}}
\label{ec:buffer-engines}

\MV{
A central design choice of Section~\ref{sec:uncertainty-model} is
that the MILP is buffer-agnostic: the optimization receives only the
buffered durations $\{p_{jrk}^{(t)}\}$ and is indifferent to how they
were produced. Cantelli is the engine used in the hospital case study
because it needs only $(\mu_{jrk},\sigma_{jrk})$ and is distribution
free over all laws with those moments. Wasserstein engines require
samples and a radius choice, but they fit the same interface. This
appendix therefore makes a narrow point: $W_\infty$ and $W_1$ are
retained as data-rich extensions of the buffer layer, not as a
validation.
}

\MV{
We verify the plug-in interface in a controlled synthetic experiment.
For each surgery we draw estimation, calibration, and held-out test
samples \GRN{($n_E=n_C=200$, $n_T=500$)} from the synthetic generator of
Section~\ref{sec:computational-experiments}. The schedule uses
$\Acal=\{0.01,0.05,0.10\}$ and the same M3 model with
$\varepsilon_{\rm target}=0.40$. The empirical-quantile engine uses the
estimation quantile directly. For each Wasserstein engine
$e\in\{W_\infty,W_1\}$ we choose, from a finite grid $G_e$ of candidate radii \MV{(in the synthetic experiments, expressed as dimensionless multiples of the estimated mean, $G_{W_\infty}=\{0,\allowbreak 0.005,\allowbreak 0.01,\allowbreak 0.02,\allowbreak 0.03,\allowbreak 0.05,\allowbreak 0.075,\allowbreak 0.10,\allowbreak 0.125,\allowbreak 0.15,\allowbreak 0.175,\allowbreak 0.20,\allowbreak 0.25,\allowbreak 0.30,\allowbreak 0.35,\allowbreak 0.40,\allowbreak 0.45,\allowbreak 0.50\}$ and $G_{W_1}=\{0,\allowbreak 0.001,\allowbreak 0.0025,\allowbreak 0.005,\allowbreak 0.01,\allowbreak 0.02,\allowbreak 0.03,\allowbreak 0.05\}$)},
a single radius
$\eta_{e}^{*}=\min\{\eta\in G_e:v_{e,t}(\eta)\le\alpha_t\ \text{for all }t\}$,
where $\alpha_t$ ranges over the menu $\Acal$ and $v_{e,t}(\eta)$ is the
\emph{calibration exceedance} --- \MV{the pooled, observation-weighted
fraction, across all surgery--room--surgeon cells, of} calibration-sample
observations that fall above the buffer engine $e$ produces at menu level
$\alpha_t$ with radius $\eta$\MV{; cells thus contribute in proportion to their
sample counts, so this criterion controls the aggregate exceedance rather than
each cell's individually}. \MV{Formally, writing $\mathcal V$ for the set of feasible surgery--room--surgeon cells, $N_C$ for the number of calibration draws per cell, $X^{(n)}_{jrk,C}$ for the $n$-th calibration observation in cell $(j,r,k)$, and $p^{(t)}_{jrk,e}(\eta)$ for the buffered duration engine $e$ produces at menu level $\alpha_t$ with radius $\eta$, we have $v_{e,t}(\eta)=\tfrac{1}{|\mathcal V|N_C}\sum_{(j,r,k)\in\mathcal V}\sum_{n=1}^{N_C}\mathbf 1\{X^{(n)}_{jrk,C}>p^{(t)}_{jrk,e}(\eta)\}$.}
That is, $\eta_e^*$ is the smallest grid
radius whose buffers meet the nominal target $\alpha_t$ at every menu level; if
no grid value qualifies, we use the largest radius and record the miss. The
test sample is used only after scheduling.
Concretely, the calibrated multiplier sets the ambiguity radius
$\rho_{jrk}=\eta_e^*\,\hat\mu_{jrk,E}$ for every cell --- proportional to the
estimated mean, so that $\eta$ is dimensionless across surgeries of different
size --- with $W_\infty$ shifting the empirical quantile by $\rho_{jrk}$ and
$W_1$ taking the worst-case tail buffer over the $W_1$ ball of radius
$\rho_{jrk}$ (Propositions~\ref{prop:Winf-buffer} and~\ref{prop:W1-tail}).
We stress that $\eta_e^*$ is fixed by held-out calibration to match nominal
coverage, \emph{not} by a measure-concentration or finite-sample guarantee on
the radius as in \citet{MohajerinKuhn2018}; this appendix demonstrates the
plug-in interface, not a certified ambiguity set.
The stationary regime uses one log-normal law throughout, while the
tail-drift regime draws estimation data from the baseline law and
calibration/test data from a mildly tail-contaminated future law, in which
each duration is independently inflated by a factor $U[1.25,1.75]$ with
probability $0.10$ and left unchanged otherwise. Thus
the design is intentionally limited: log-normal data make Cantelli
conservative by construction, and the drift is detectable because it
appears in the calibration sample. We use a single tight, high-variance
instance \GRN{design} --- 25 surgeries across 3 rooms and 4 surgeons, with a
large-case-heavy mix\GRN{, drawn at three seeds} --- that separates the engines strongly but solves only to
large MIP gaps, so we report \GRN{these} costs as incumbents rather than certified
optima.

\paragraph{Out-of-sample evaluation.} Each solved schedule is replayed on the
held-out test draws: within a room, each case starts at the later of its planned
start and its predecessor's realized finish, so delays propagate down the day.
From the replay we report, per engine, the selected-buffer exceedance $\Pr(X>p)$
(the fraction of test draws in which a case's realized duration exceeds its
assigned buffer at its selected~$\alpha$, pooled over all scheduled cases regardless of menu level); the \emph{day violation} rate (the
fraction of day--draw pairs in which at least one case on the day exceeds its
buffer); \emph{overtime} (the realized minutes a room's last case runs past the
day horizon, averaged over draws); and \emph{p95 delay} (the $95$th-percentile
case start delay, realized minus planned start, averaged over draws). The
\emph{buffer/case} column is the mean planned buffer above the \GRN{estimated ($E$-sample)} case mean.
}

\MV{
\paragraph{From samples to buffers.}
The left (buffer) columns of Table~\ref{ec:tab-appH-engines} show the buffers
averaged over all cases. Cantelli is roughly twice the true $1\%$ tail
\GRN{in the stationary regime (about $1.5\times$ under drift)};
the empirical quantile is cheapest but, computed on $E$, barely moves under
drift, while the calibrated $W_\infty$ and $W_1$ buffers inflate to track
the heavier tail seen in $C$.
}
\begin{table}[H]\centering
\caption[Buffer engines]{\MV{The four buffer engines on this instance\GRN{, with
single-$\eta$ calibration for $W_\infty$ and $W_1$} (3-seed mean). \emph{Left:} buffers (min), mean
over all cells, with the true $1-\alpha$ quantile for reference. \emph{Right:}
realized out-of-sample metrics; buffer/case is incumbent-based (large MIP gaps),
$\Pr(X{>}p)$ is the overall exceedance --- the fraction of case-realizations whose duration exceeds the assigned buffer, pooled over all selected menu levels --- overtime and delay in minutes.}}
\label{ec:tab-appH-engines}
\MV{\small\setlength{\tabcolsep}{4pt}
\resizebox{\textwidth}{!}{%
\begin{tabular}{ll rrr !{\vrule width 0.4pt} rrrrr}
\toprule
 \rowcolor{hdrgray}& & \multicolumn{3}{c!{\vrule width 0.4pt}}{Buffers (min)}
   & \multicolumn{5}{c}{Realized (3-seed mean)} \\
\cmidrule(lr){3-5}\cmidrule(lr){6-10}
\rowcolor{hdrgray}Regime & Engine & $\alpha{=}.01$ & $.05$ & $.10$
 & buffer/case & $\Pr(X{>}p)$ & day viol. & overtime & p95 del. \\
\midrule
\multirow{5}{*}{Stationary}
 & Cantelli            & 469.9 & 296.8 & 254.8 & 139.3 & 0.3\% &  1.5\% &  0.36 & 0.00 \\
 & Empirical quantile  & 242.1 & 216.4 & 202.4 &  83.7 & 1.5\% &  7.3\% &  2.21 & 0.06 \\
 & $W_\infty$          & 250.0 & 224.2 & 210.2 &  91.4 & 1.0\% &  4.9\% &  1.08 & 0.03 \\
 & $W_1$               & 272.0 & 228.9 & 211.4 & 113.1 & 0.3\% &  1.6\% &  0.41 & 0.00 \\
 & \textit{true quantile} & \textit{245.5} & \textit{216.2} & \textit{202.6} & & & & & \\
\addlinespace
\multirow{5}{*}{Tail drift}
 & Cantelli            & 469.3 & 296.5 & 254.5 & 140.4 & 2.4\% & 11.7\% & 13.77 & 0.29 \\
 & Empirical quantile  & 243.9 & 216.6 & 201.8 &  85.6 & 5.7\% & 25.6\% & 28.61 & 1.18 \\
 & $W_\infty$          & 319.4 & 292.0 & 277.2 & 138.7 & 1.7\% &  8.0\% &  7.16 & 0.18 \\
 & $W_1$               & 366.5 & 254.5 & 227.7 & 138.3 & 3.7\% & 17.4\% & 10.52 & 1.36 \\
 & \textit{true quantile} & \textit{312.6} & \textit{248.2} & \textit{221.3} & & & & & \\
\bottomrule
\end{tabular}}
}
\end{table}

\MV{
The experiment supports exactly the modularity claim: sample-based engines can
be calibrated before scheduling and inserted into the same $\alpha$-menu MILP
through $p_{jrk}^{(t)}$ without reformulation, and each is then validated out of
sample on identical instances. Because the interface is modular, the same
instance also lets us \emph{observe} how the engines behave. \MV{As the exceedance
columns of Table~\ref{ec:tab-appH-engines} show, under tail drift the empirical
quantile under-covers most, while the distribution-free Cantelli and the
calibrated Wasserstein engines keep exceedance lower ($W_\infty$ best).} \GRN{Day
violation rises in parallel; it is the realized out-of-sample analogue of the
day-reliability budget $\varepsilon\le 0.40$, and reflects both how much of that
budget a schedule spends and how well its buffers cover the realized durations.}

We report these as observations, not a ranking. This is a single synthetic
instance \GRN{design (three seeds)} solved only to incumbent (not certified) cost, so we make no dominance
claim between ambiguity sets, and the analysis does not replace the
Cantelli-based hospital study. Which ambiguity geometry and calibration rule best
convert a drift signal into out-of-sample coverage --- across instances, drift
regimes, and under genuine, unobserved nonstationarity --- is left to future work,
as is replacing our calibrated radius with a principled rule, whether
concentration-based \citep{MohajerinKuhn2018} or data-reweighted; the weighted
empirical framework of \citet{KeehanAndersonWiesemann2025} is a natural candidate
to plug into this same buffer layer.
}

\end{document}